\documentclass[10pt,reqno,a4paper]{amsart}
\usepackage{amssymb,mathtools,times,enumerate,a4wide,xcolor}
\usepackage[final]{hyperref}
\usepackage{datetime,comment}

\usepackage[cp1250]{inputenc}
\usepackage[T1]{fontenc}

\numberwithin{equation}{section}

\newtheorem{theorem}{Theorem}[section]
\newtheorem{lemma}[theorem]{Lemma}
\newtheorem{proposition}[theorem]{Proposition}
\newtheorem{corollary}[theorem]{Corollary}

\newtheorem*{claim}{Claim}
\newtheorem*{claima}{Claim A}
\newtheorem*{claimb}{Claim B}

\newtheorem*{LNtheorem*}{Luzin--Novikov Theorem}
\newtheorem*{SRtheorem*}{Saint Raymond Theorem}

\newtheorem{theorem**}{Theorem}[subsection]
\newtheorem{lemma**}[theorem**]{Lemma}
\newtheorem{proposition**}[theorem**]{Proposition}

\theoremstyle{definition}
\newtheorem{definition}[theorem]{Definition}
\newtheorem{notation}[theorem]{Notation}
\newtheorem{remark}[theorem]{Remark}

\newtheorem{definition**}[theorem**]{Definition}
\newtheorem{notation**}[theorem**]{Notation}
\newtheorem{remark**}[theorem**]{Remark}
\newtheorem{fact**}[theorem**]{Fact}

\def\cC{\mathcal C}
\def\cH{\mathcal H}
\def\cN{\mathcal N}
\def\cP{\mathcal P}
\def\cQ{\mathcal Q}
\def\cS{\mathcal S}
\def\cX{\mathcal X}
\def\cY{\mathcal Y}
\def\cZ{\mathcal Z}

\def\fR{\mathbin{\mathfrak R}}

\def\bbA{\mathbb A}
\def\bbJ{\mathbb J}

\newcommand{\al}{\alpha}
\newcommand{\be}{\beta}
\newcommand{\ga}{\gamma}
\newcommand{\Ga}{\Gamma}
\newcommand{\de}{\delta}
\newcommand{\De}{\Delta}
\newcommand{\ep}{\varepsilon}
\newcommand{\ph}{\varphi}

\newcommand{\io}{\iota}

\newcommand{\la}{\lambda}
\newcommand{\om}{\omega}
\newcommand{\Om}{\Omega}
\newcommand{\bPi}{\boldsymbol\Pi}
\newcommand{\si}{\sigma}
\newcommand{\Si}{\Sigma}
\newcommand{\bSi}{\boldsymbol\Sigma}

\newcommand{\Th}{\Theta}

\def\Seqm{\mathcal S^*}
\def\Seq{\mathcal S}

\def\w{\,{}^\wedge}
\def\R{\mathbin{\mathfrak R}}

\def\id{\operatorname{id}}
\def\LIM{\operatorname{LIM}}

\def\graph{\operatorname{graph}}
\newcommand{\cl}{\overline}
\newcommand{\su}{\subset}
\newcommand{\sm}{\setminus}
\newcommand{\wh}{\widehat}
\newcommand{\wt}{\widetilde}
\newcommand{\set}{;\ }

\newcommand{\hPsi}{\wh{\Psi}}
\newcommand{\hpsi}{\wh{\psi}}

\newcommand{\tpsi}{\widetilde{\psi}}

\newcommand{\Rleft}{\overset{\scriptstyle\leftarrow}{R}}
\newcommand{\rleft}{\overset{\scriptstyle\leftarrow}{r}}
\newcommand{\rright}{\overset{\scriptstyle\rightarrow}{r}}
\newcommand{\fleft}{\overset{\scriptstyle\leftarrow}{f}}
\newcommand{\fright}{\overset{\scriptstyle\rightarrow}{f}}
\newcommand{\eleft}{\overset{\scriptstyle\leftarrow}{e}}
\newcommand{\eright}{\overset{\scriptstyle\rightarrow}{e}}

\begin{document}
\title[There is no bound on Borel classes of the graphs in the Luzin-Novikov theorem]{There is no bound on Borel classes of the graphs\\ in the Luzin-Novikov theorem}
\author{P. Holick\'y}
\address{Charles University, Faculty of Mathematics and Physics, Department of Mathematical Analysis,
Soko\-lov\-sk\'{a}~83, Prague~8, 186~75, Czech Republic}
\email{holicky@karlin.mff.cuni.cz}
\author{M. Zelen\' y}
\address{Charles University, Faculty of Mathematics and Physics, Department of Mathematical Analysis,
Soko\-lov\-sk\'{a}~83, Prague~8, 186~75, Czech Republic}
\email{zeleny@karlin.mff.cuni.cz}


\thanks{The research was supported in part by the grant GA\v{C}R 15-08218S}

\begin{abstract}
We show that for every ordinal $\alpha \in [1, \omega_1)$ there is a closed set $F \subset 2^\om \times \om^\om$
such that for every $x \in 2^\om$ the section $\{y\in \om^\om\set (x,y) \in F\}$
is a two-point set and $F$ cannot be covered by countably many graphs
$B(n) \subset 2^\om \times \om^\om$ of functions of the variable $x \in 2^\om$ such that each $B(n)$ is in the additive Borel class $\bSi^0_\al$.
This rules out the possibility to have a quantitative version of the Luzin-Novikov theorem. The construction is a modification of the method of Harrington who invented it to show that there exists a countable $\Pi^0_1$ set in $\om^\om$ containing a non-arithmetic singleton. By another application of the same method we get closed sets excluding a quantitative version of the Saint Raymond theorem on Borel sets with $\si$-compact sections.
\end{abstract}

\maketitle

\section{Main result}

The main goal of this paper is to consider a quantitative version of the following classical uniformization theorems:
the Luzin--Novikov theorem (\cite{Luzin-Novikov}, cf. \cite[18.10]{Kechris})  for Borel sets with countable sections and
the Saint Raymond theorem (\cite[Corollaire 10]{SaintRaymond}, cf. \cite[35.46]{Kechris}) for Borel sets with $\si$-compact sections.

\begin{LNtheorem*}
Let $X$ and $Y$ be Polish spaces and $B$ be a Borel subset of $X \times Y$ such that for every $x \in X$ the section $B_x := \{y\in Y\set (x,y)\in B\}$ is countable. Then $B$ can be covered by countably many Borel sets $B(n) \su B$, $n\in\om$, such that the section $B(n)_x$ contains at most one point for every $n \in \omega$ and $x \in X$.
\end{LNtheorem*}

\begin{SRtheorem*}
Let $X$ and $Y$ be Polish spaces and $B$ be a Borel subset of $X \times Y$ such that for every $x \in X$ the section $B_x$ is $\si$-compact.
Then $B$ can be covered by countably many Borel sets $B(n) \su B$, $n \in \omega$, such that $B(n)_x$ is compact for every $n \in \omega$ and $x \in X$.
\end{SRtheorem*}

Suppose that $X$ and $Y$ are Polish spaces. In the case of the Luzin-Novikov theorem we are interested in the question whether for every ordinal $1 \le \beta < \omega_1$ there
is an ordinal $1 \leq \ph(\be) < \om_1$ such that for every $B \subset X \times Y$ in the multiplicative Borel class $\bPi^0_\be$ with countable sections $B_x$, $x\in X$,
we can find graphs $B(n), n \in \omega,$ such that $B = \bigcup_{n \in \omega} B(n)$ and each $B(n)$ belongs to
$\bSi^0_{\ph(\be)}$ (cf. Remark~\ref{R:why-graphs}).
As for the Saint Raymond theorem, we are interested in the question whether for every ordinal $1 \le \beta < \omega_1$ there is an ordinal
$1 \le \ph(\be) < \om_1$ such that
for every $B \subset X \times Y$ in $\bPi^0_\be$ with $\sigma$-compact sections $B_x$, $x\in X$, we can find sets $B(n), n \in \omega,$ with compact sections $B(n)_x, x \in X$,
such that $B = \bigcup_{n \in \omega} B(n)$ and each $B(n)$ belongs to $\bSi^0_{\ph(\be)}$.

Our main results say that no countable upper bound $\ph(\be)$ exists even if $B$ is closed with two-point sections in the case of the Luzin-Novikov theorem,
and if $B$ is closed with countable discrete sections in the case of the Saint Raymond theorem. The main results of the paper read as follows.

\begin{theorem}\label{T:main-theorem-Luzin}
Let $1 \leq \alpha < \omega_1$ be an ordinal. Then there is a closed set $F \subset 2^{\om} \times \om^{\om}$ such that for every $x \in 2^\om$
the section $F_x$ is a two-point set and $F$ cannot be covered by countably many sets $B(n) \subset F$, $n\in\om$, such that each $B(n)$ is
in the Borel class $\bSi^0_\al$ and the section $B(n)_x$ contains at most one point for every $n \in \omega$ and $x\in 2^\om$.
\end{theorem}

Using a homeomorphic embedding of $\om^\om$ onto a subset $2^{\om}\sm D$ of $2^{\om}$, where $D$ is a countable dense subset of $2^{\om}$, we get the following corollary.

\begin{corollary}\label{C:Luzin-Novikov}
Let $1 \leq \alpha < \omega_1$ be an ordinal.  Then there is a compact set $K \subset 2^\om \times 2^\om$ such that for every $x \in 2^\om$ the section $K_x$ is countable and $K$ cannot be covered by countably many sets $B(n) \subset K$, $n\in\om$,
such that each $B(n)$ is in the Borel class $\bSi^0_\al$ and the section $B(n)_x$ contains at most one point for every $n \in \omega$ and $x\in 2^\om$.
\end{corollary}

\begin{theorem}\label{T:SR}
Let $1 \leq \alpha < \omega_1$ be an ordinal. Then there is a closed set $F \subset 2^{\om} \times \om^{\om}$ such that for every $x \in 2^\om$ the section $F_x$ is a countable discrete set and $F$ cannot be covered by countably many sets $B(n) \subset F$, $n\in\om$,
such that each $B(n)$ is in the Borel class $\bSi^0_\al$ and the section $B(n)_x$ is compact for every $n \in \omega$ and $x\in 2^\om$.
\end{theorem}

Let us notice that we get an example disproving the possibility of a quantitative version of the Saint Raymond theorem for sets with $F_\si$ sections
(\cite[35.45]{Kechris}).

\begin{corollary}\label{C:Saint-Raymond}
Let $1 \leq \alpha < \omega_1$ be an ordinal.  Then there is a $G_\de$ set $G \subset 2^\om \times 2^\om$ such that for every $x \in 2^\om$ the section
$G_x$ is countable and relatively discrete and $G$ cannot be covered by countably many sets
$B(n) \subset G$, $n \in \om$,
such that each $B(n)$ is in the Borel class $\bSi^0_\al$ and the section $B(n)_x$ is closed for every $n \in \omega$ and $x\in 2^\om$.
\end{corollary}

The construction of the desired examples of closed subsets of $2^\om\times\om^\om$ is a modification of the method of Harrington.
Besides solving other problems concerning computability he used this method in \cite{Harrington} to show that there exists a countable  $\Pi^0_1$ (effectively closed)
set in $\om^\om$ containing non-arithmetic singleton (\cite[Problem~63]{Friedman}).
He pointed out in \cite{Harrington2} also the possibility to use his method to get for every ordinal $1 \leq \al < \omega_1$ a closed subset $F$ of $\om^\om\times\om^\om$ such that
all sections $F_x$, $x\in\om^\om$, contain exactly two points and there is no function $f\colon \om^\om\to\om^\om$ such that its graph is
a $\bSi_{\al}^0(\omega^{\omega} \times \omega^{\omega})$ subset of $F$. Theorem~\ref{T:main-theorem-Luzin} implies this result.
In unpublished manuscripts of Gerdes (\cite{Gerdes}) and of Hjorth (\cite{Hjorth}), Harrington's method is presented in two different generalized ways.
We use some ideas from these manuscripts, too.

The main results are inferred in Section~\ref{S:proofs-of-main-results} using auxiliary propositions which are proved in the next sections.
In Section~\ref{S:ordinal-structure}, following the ideas of \cite{Gerdes}, we describe a structure of the interval of ordinals
which is suitable for the application to Harrington's construction for countable ordinals.
In the paper we use effective descriptive theory.
We refer to Moschovakis' book \cite{Moschovakis}, however facts repeatedly used in proofs will be recalled in Section~\ref{S:recursive-background}
(e.g., recursive sets and functions, good universal sets and functions, Kleene's recursion theorems).
In Section~\ref{S:Turing-jumps} we present a version of Turing jumps and infer their properties needed later on.
Section~\ref{S:construction-of-T-alpha} is devoted to constructions of closed sets showing that naturally modified assertions
of Theorems~\ref{T:main-theorem-Luzin} and \ref{T:SR} for $\alpha = 1$ hold in an effective sense. These sets will serve as a starting point for the general construction.
In Section~\ref{S:reduction-of-Borel-classes} we reduce ``effectively Borel sets'' to ``effectively open sets'' using Turing jumps.
Then we present our modification of Harrington's construction (Section~\ref{S:Harrington-method}).

\begin{remark}\label{R:why-graphs}
Notice that $B(n)$'s in Theorem~\ref{T:main-theorem-Luzin} and Corollary~\ref{C:Luzin-Novikov} are graphs of Borel functions $f_n$ defined on Borel subsets of $X$ because the projections restricted to $B(n)$'s are continuous and one-to-one mappings of $B(n)$'s onto the domains of $f_n$'s (see, e.g., \cite[Corollary~15.2]{Kechris}).

Let us remark that Theorem~\ref{T:main-theorem-Luzin} remains true if we replace the assumption on the Borel class of $B(n)$'s
by $\bSi_{\alpha}^0$-measurability of  the \emph{mappings} $f_n$. This follows easily from classical results of descriptive set theory.
Indeed, let $H \subset 2^{\om}$ be a Borel set which is not of the class $\bSi^0_{\alpha}$.
Then there is an injective continuous function $g$ mapping a closed subset $E$ of $\om^{\om}$ onto $H$ (see, e.g., \cite[13.7]{Kechris}).
The graph $F$ of $f=g^{-1}\colon H \to E$ is closed in $2^{\om}\times \om^{\om}$ with at most one-point sections $F_x$, $x \in 2^{\om}$.
If we consider mappings $f_n \colon H_n \to \omega^{\omega}, n \in \omega$, with $\bigcup_{n \in \omega} \graph f_n = F$,
then at least one of these mappings is not $\bSi_{\alpha}^0$-measurable since otherwise $H$ would be of the class $\bSi^0_{\alpha}$.
\end{remark}

\section{Proof of the main results}\label{S:proofs-of-main-results}

\begin{notation}\label{N:alfa}
We fix a limit ordinal $\al < \om_1$.
\end{notation}

\begin{notation}\label{N:notation-1}
We use further the notation $\cN:=\om^{\om}$ and $\cC:=2^\om$.
By $\cS$ we denote the set $\om^{<\om}$ of all finite sequences of elements of $\om$ including the empty one. Given $k\in\om$
and $s\in \cN\cup\cS$ we write $k\w s$ for the sequence beginning by $k$ followed by $s$.
We denote by $\cl{n}$ the constant sequence in $\cN$ which attains the value $n$.
If $T \subset \cS$ is a tree, then $[T]$ denotes the set of all infinite branches of $T$.
\end{notation}

The choice of $\varepsilon \in \cN$ from Notation~\ref{N:epsilon} ensures that certain sets and mappings related to $[0,\alpha]$
become semirecursive in $\ep$ or recursive in $\ep$.

In Section~\ref{S:Turing-jumps} we define, for every $x\in\cN$ and $\be\in [0,\al]$, the ``$\be$-th Turing jump'' $y \in \cN \mapsto y^{(\be)}_x\in\cC\su\cN$  depending on $\ep$.
These mappings serve as the main tool for further constructions, mainly due to Proposition~\ref{P:jump-and-classes-1}.
We also write simply $y^{(\beta)}$ instead of  $y^{(\beta)}_{\cl{0}}$.

In Section~\ref{S:construction-of-T-alpha} we prove the following statements, where
the symbol $[\phantom{a}]_1$ denotes the mapping from $\cN$ to $\cN$  introduced in Notation~\ref{N:identification}.
The meaning of the notion ``$\Si_1^0([z]_1,\ep)$-recursive function'' is explained in Definitions~\ref{D:relativizations},  \ref{D:recursive-function}, and \ref{D:Ga-in-extended-spaces}(a).

\begin{proposition}\label{P:E^LN}
There exists a $\Si^0_1(\ep)$-recursive set $T_\al = T_\al^{LN} \su \cS\times\cN$ such that the following properties are satisfied for every $x \in \cN$.
\begin{itemize}
\item[\upshape (i)] The set $T_{\alpha}^x := \{t\in\cS\set (t,x)\in T_\al\}$ is a tree.

\item[\upshape (ii)] The set $\bigl\{[z]_1 \in \cN \set k \in \omega, k \w z \in [T^x_\al]\bigr\}$ is a singleton.

\item[\upshape (iii)] If $k \in \omega$ and $k \w z \in [T_{\alpha}^x]$, then the Turing jump $x^{(\alpha)}$ is $\Si_1^0([z]_1,\ep)$-recursive.
\end{itemize}

Moreover, the set $E^{LN} := \{(x,y) \in \cN^2 \set y \in [T_{\alpha}^x]\}$ satisfies
\begin{itemize}
\item[\upshape (iv)] For every $k \in \om$ we have that $\{z \in \cN \set k \w z \in E^{LN}_x\}$ is a two-point set.

\item[\upshape (v)] There is no $\Si^0_1(x,\ep)$ set $W$ in $\cN \times \om$ such that
\begin{itemize}
  \item $\bigcup_{n \in \omega} W^n \supset E^{LN}_x$ and

  \item the set $\{z \in \cN \set k\w z\in W^n \cap E^{LN}_x\}$ contains at most one point for every $k,n \in \omega$.
\end{itemize}
\end{itemize}
\end{proposition}

\begin{proposition}\label{P:E^SR}
There exists a $\Si^0_1(\ep)$-recursive set $T_\al = T_\al^{SR} \su \cS\times\cN$ such that the following properties are satisfied for every $x \in \cN$.
\begin{itemize}
\item[\upshape (i)] The set $T_{\alpha}^x$ is a tree.

\item[\upshape (ii)] The set $\bigl\{[z]_1 \in \cN \set k \in \omega, k \w z \in [T^x_\al]\bigr\}$ is a singleton.

\item[\upshape (iii)] If $k \in \omega$ and $k \w z \in [T_{\alpha}^x]$, then the Turing jump $x^{(\alpha)}$ is $\Si_1^0([z]_1,\ep)$-recursive.
\end{itemize}

Moreover, the set $E^{SR} := \{(x,y) \in \cN^2 \set y \in [T_{\alpha}^x]\}$ satisfies
\begin{itemize}
\item[\upshape (iv)] The set $E^{SR}_x$ is closed and discrete.

\item[\upshape (v)] There is no $\Si^0_1(x,\ep)$ set $W$ in $\cN \times \om$ such that
\begin{itemize}
  \item $\bigcup_{n \in \omega} W^n \supset E^{SR}_x$ and

  \item the set $W^n \cap E^{SR}_x$ is finite for every $n \in \omega$.
\end{itemize}
\end{itemize}
\end{proposition}

\begin{remark}
Notice that for every $x \in \cN$ only finite subsets of $E^{SR}_x$ are compact.
\end{remark}

In Section~\ref{S:reduction-of-Borel-classes} we prove the following result for
the mapping $\Upsilon^x\colon \cN \to \cN$  defined by $\Upsilon^x(y) = y_x^{(\alpha)}$ for $x \in \cN$.

\begin{proposition}\label{P:jump-and-classes-1}
If $A(n) \su \cN^2$, $n \in \om$, are in $\bSi^0_{\al}$, then there are $a \in \cC$ and $H \in \Si^0_1(a,\ep)$ in $\cN \times \om$ such that
for every $n \in \omega$ we have $A(n)_a = (\Upsilon^a)^{-1}(H^n)$.
\end{proposition}

The main construction in Section~\ref{S:Harrington-method} culminates by proving the following proposition in Subsection~\ref{SS:properties-of-Xi}.

\begin{proposition}\label{P:hPsi-properties}
Let $T_\al \su \cS\times\cN$ be a $\Si^0_1(\ep)$-recursive set such that
\begin{itemize}
\item[\upshape (i)] the set $T_\al^x$ is a tree for every $x \in \cN$,
\item[\upshape (ii)] the set $\bigl\{[z]_1 \in \cN \set k \in \omega, k \w z \in [T^x_\al]\bigr\}$ is a singleton, and
\item[\upshape (iii)] if $k \in \omega$ and $k\w z \in [T^x_\al]$, then the Turing jump $x^{(\alpha)}$ is $\Si_1^0([z]_1,\ep)$-recursive.
\end{itemize}
Using the notation $E := \bigl\{(x,y) \in \cN^2 \set y \in [T_{\alpha}^x]\bigr\}$, we get that
there are a $\Pi^0_1(\ep)$ set $F$ in $\cN \times \cN$ and functions $\Xi^x\colon E_x \to F_x$, $x \in \cN$, such that
\begin{itemize}
\item[\upshape (a)] for every $x \in \cN$ the function $\Upsilon^x\circ\Xi^x\colon E_x \to \cN$ is $\Si^0_1(\ep)$-recursive on $E_x$ and
\item[\upshape (b)] for every $x\in\cN$ and $k\in\om$ the mapping $\Xi^x_k$ defined by $\Xi^x_k(z)=\Xi^x(k\w z)$ is a homeomorphism of $(E_x)_k = \{z \in \cN \set k \w z \in E_x\}$
onto $(F_x)_k = \{z \in \cN \set k \w z \in F_x\}$.
\end{itemize}
\end{proposition}

\begin{corollary}\label{C:PSIBOREL}
Let $E$, $F$, and $\Xi^x, x \in \cN$, be as in Proposition~\ref{P:hPsi-properties}.
Then for every sequence $(B(n))_{n \in \omega}$ of $\bSi^0_{\al}$ subsets of $\cN^2$,
there are $a \in \cC$ and $W \in \Sigma^0_1(a,\ep)$ in $\cN \times \om$ such that
\[
(\Xi^a)^{-1}(B(n)_a) = W^n \cap E_a.
\]
\end{corollary}

\begin{proof}
Let $B(n) \subset \cN^2$ be in $\bSi^0_{\al}$ for every $n \in \om$.
Then there are $a\in \cC$ and a $\Si^0_1(a,\ep)$ set $H$ in $\cN \times \omega$ such that
$B(n)_a = (\Upsilon^a)^{-1}(H^n)$ by Proposition~\ref{P:jump-and-classes-1}.
By Proposition~\ref{P:hPsi-properties}(a), \cite[3D.1]{Moschovakis}, and the substitution property (cf. Fact~\ref{F:partialrecursivefunctions}\eqref{I:substitution-property})
there exists a $\Si^0_1(a,\ep)$ set $W$ in $\cN\times\om$ such that, for every $n \in \omega$, we have
\[
(\Xi^a)^{-1}(B(n)_a) = (\Upsilon_a\circ \Xi^a)^{-1}(H^n) = W^n \cap E_a. \qedhere
\]
\end{proof}

We are now going to apply the properties of $E^{LN}$ and $E^{SR}$ (from Propositions~\ref{P:E^LN} and \ref{P:E^SR})
together with Proposition~\ref{P:hPsi-properties} and Corollary~\ref{C:PSIBOREL} to prove our main results.

\begin{proof}[{\bf Proof of Theorem~\ref{T:main-theorem-Luzin}}]
Let $E^{LN}$ be the set from Proposition~\ref{P:E^LN},
$F^{LN}$ and $\Xi^x$, $x \in \cN$, be the corresponding set and the mappings from Proposition~\ref{P:hPsi-properties}.
The set $F^{LN}$ is closed since it is $\Pi_1^0(\ep)$.
We set $\wt{F} := F^{LN} \cap (\cC\times\cN)$.

We first prove that $\wt{F}$ cannot be covered by countably many $\bSi^0_\al$ subsets $B(n)$, $n\in\om$, of $\cC\times\cN$ with at most one point sections.
Towards contradiction assume that $B(n), n \in \omega$, are $\bSi^0_{\al}$ subsets of $\wt{F}$ such that $\wt{F} = \bigcup_{n\in\omega} B(n)$
and $B(n)_x$  contains at most one point for every $x \in \cC$ and $n \in \omega$.
By Corollary~\ref{C:PSIBOREL} there are $a \in \cC$ and a $\Si^0_1(a,\ep)$ set $W \subset \cN \times \om$ such that $(\Xi^a)^{-1}(B(n)_a) = W^n \cap E^{LN}_a$ for every $n \in \omega$.
Since $\wt{F} = \bigcup_{n\in\omega} B(n)$ and $\Xi^a:E^{LN}_a\to F_a$, we get $\bigcup_{n \in \omega} W^n \supset E^{LN}_a$.
Therefore, due to injectivity of $\Xi^a$ and to the preservation of the first coordinate by $\Xi^a$,
the set $\{z \in \cN \set k \w z\in E^{LN}_a \cap W^n\}$ contains at most one point for every $k, n \in \omega$.
Using (v) of Proposition~\ref{P:E^LN}, we get a contradiction.

We have $\wt{F} = \bigcup_{k \in \omega} \{(x,k \w z) \in \cC \times \cN \set (x, k \w z) \in \wt{F}\}$.
So there exists $k' \in \omega$  such that $F := \{(x, z) \in \cC \times \cN \set (x, k' \w z) \in \wt{F}\}$ cannot be covered by countably many uniformizations from $\bSi^0_\al$ as well.
This $F$ has moreover exactly two-point sections $F_x$ as images by the injective mappings $\Xi^x_{k'}$
(see Proposition~\ref{P:hPsi-properties}(b)) of the two-point sets
$(E^{LN}_x)_{k'}$ for all $x \in \cC$ (see (iv) of Proposition~\ref{P:E^LN}).
The set $F^{LN}$ is closed and therefore $\wt{F}$ is closed. Thus $F$ is closed as well.
\end{proof}

\begin{proof}[{\bf Proof of Corollary~\ref{C:Luzin-Novikov}}]
Let us consider the set $F$ from Theorem~\ref{T:main-theorem-Luzin} and a homeomorphism $h\colon \cN\to \cC\sm D$, where $D$ is a countable dense subset of $\cC$, see, e.g., \cite[Theorem~7.7]{Kechris}. The set
\[
K = \cl{\bigl\{(x,h(y))\in \cC \times \cC\set (x,y)\in F\bigr\}}
\]
is compact, it has countable sections $K_x$, $x\in \cC$, because $\{(x,h(y))\set (x,y)\in F\}$ is closed in $\cC\times (\cC\sm D)$ and $D$ is countable. Suppose that we have a cover of $K$ by countably many $\bSi^0_{\al}$ sets $B(n) \subset \cC \times \cC$ such that $B(n)_x$ contains at most one point for every $n \in \omega, x \in \cC$.
We may notice that the $\bSi^0_{\al}$ sets $\wt B(n) := \{(x,y)\in \cC\times \cN\set (x,h(y))\in B(n)\}$, $n \in \omega$,
cover $F$ and $\wt B(n)_x$ contains at most one point for every $n \in \omega, x \in \cC$, which is a contradiction with the choice of $F$.
\end{proof}

\begin{remark}\label{R:no-compact-example}
Notice that for every compact $K\su\cN\times\cN$ the functions $x \mapsto \min K_x$ and $x \mapsto \max K_x$ are lower semicontinuous and upper semicontinuous on the compact set $\{x\in\cN\set K_x\ne\emptyset\}$, respectively. Therefore, if $K_x$, $x\in\cN$, contain at most two points, then we can cover $K$ by two graphs of functions of the first class,
i.e., with $\bPi^0_2$ graphs.
\end{remark}

\begin{proof}[{\bf Proof of Theorem~\ref{T:SR}}]
Let $E^{SR}$ be the set from Proposition~\ref{P:E^SR}, $F^{SR}$ and $\Xi^x$, $x \in \cN$, be the corresponding set and mappings from Proposition~\ref{P:hPsi-properties}.
The set $F^{SR}$ is closed since it is $\Pi_1^0(\ep)$.
We define a closed set $F = F^{SR} \cap (\cC\times\cN)$ and we first prove that it cannot be covered by countably many $\bSi^0_\al$ subsets of $\cC\times\cN$ with compact sections.

Towards contradiction assume that $B(n), n \in \omega$, are $\bSi^0_{\al}$ subsets of $\cC \times\cN$ such that $B(n)_x$ is compact for every $x \in \cN, n \in \omega$ and
$F = \bigcup_{n\in\omega} B(n)$. By Corollary~\ref{C:PSIBOREL} there are $a \in \cC$ and a $\Si^0_1(a,\ep)$ set $W \subset \cN \times \om$ such that $(\Xi^a)^{-1}(B(n)_a) = W^n \cap E^{SR}_a$ for every $n \in \omega$.
We claim that the set $W \subset \cN \times \om$ and $a \in \cC$ have the two properties from Proposition~\ref{P:E^SR}(v),
which is in contradiction with the choice of $T_\al^{SR}$.
Since $F_a = \bigcup_{n \in \omega} B(n)_a$, we get the first property from Proposition~\ref{P:E^SR}(v).
By Proposition~\ref{P:hPsi-properties}(b) the set $W^n \cap E^{SR}_a$ is compact for every $n \in \omega$.
Since $E^{SR}_a$ is discrete, we get that the sets $W^n \cap E^{SR}_a$, $n \in \omega$, are finite.
Thus $W$ and $a$ fulfil also the second condition from Proposition~\ref{P:E^SR}(v), a contradiction.

The sections $F_x, x \in \cC$, are discrete due to Proposition~\ref{P:E^SR}(iv) and Proposition~\ref{P:hPsi-properties}(b).
\end{proof}

\begin{proof}[{\bf Proof of Corollary~\ref{C:Saint-Raymond}}]
Let $h\colon \cN\to \cC\sm D$ be a homeomorphism, where $D$ is a countable dense subset of $\cC$.
Then the set
\[
G = \bigl\{(x,h(y))\in \cC\times\cC\set (x,y)\in F\bigr\}
\]
with $F$ from Theorem~\ref{T:SR} has the desired properties.
\end{proof}

\section{An auxiliary structure of \texorpdfstring{$[0,\al]$}{}}\label{S:ordinal-structure}

We describe two mappings, $\diamond$ and $\ell$, which will play the same role for us as they do in \cite{Gerdes}. Nevertheless, our application does not need any information on the effectiveness of these mappings as we will see. Let $A$ be a set of ordinals. Let
us denote the set of all limit ordinals in $A$ by $\LIM(A)$.

\begin{lemma}[auxiliary structure of ordinal intervals]\label{L:diamond}
Let $\la$ be  a countable limit ordinal and $\rho<\la$ be an isolated ordinal.
Then there are a surjection $\diamond_{\rho,\la}\colon [\rho,\la) \to \LIM((\rho,\la])$, a mapping $\ell_{\rho,\la}\colon [\rho,\la] \to \om \setminus \{0\}$, and a mapping $a_{\rho,\la}\colon \LIM((\rho,\lambda]) \times \omega \to [\rho,\lambda)$ such that the following conditions are satisfied.
\begin{itemize}
  \item[\upshape (a)] For every $\be \in [\rho,\la)$, there exists $s \in \mathbb N$ such that $\be < \diamond_{\rho,\la}(\be) < \diamond_{\rho,\la}^2(\be)<\dots<\diamond_{\rho,\la}^{s}(\be)=\la$.

  \item[\upshape (b)] For every $\mu \in \LIM((\rho,\la])$, we have
      \begin{itemize}
      \item[--] $\{a_{\rho,\lambda}(\mu,k)\}_{k=0}^{\infty}$ is an increasing sequence with limit $\mu$,
      \item[--] $\diamond_{\rho,\la}^{-1}(\mu) = \{a_{\rho,\lambda}(\mu,k) \set k \in \om\}$,
      \item[--] $\ell_{\rho,\la}(a_{\rho,\lambda}(\mu,k))=\ell_{\rho,\la}(\mu)+k$, $\ell_{\rho,\la}(\la) = 1$,
                and $\ell_{\rho,\la}(\rho) = 1$.
      \end{itemize}

  \item[\upshape (c)] For every $\be < \lambda$ and every $\ga \in (\be,\diamond_{\rho,\la}(\be))$, we have $\ell_{\rho,\la}(\ga)>\ell_{\rho,\la}(\be)$ and $\diamond_{\rho,\la}(\ga)\le\diamond_{\rho,\la}(\be)$.

  \item[\upshape (d)] For every $\beta \in [\rho,\la)$, we have $\ell_{\rho,\la}(\be+1)=\ell_{\rho,\la}(\be)+1$.
 \end{itemize}
\end{lemma}

\begin{proof}
We proceed by transfinite induction on limit ordinal $\la < \omega_1$ to define $\diamond_{\rho,\la}, \ell_{\rho,\la}$, and $a_{\rho,\lambda}$ as above.
For $\lambda = \omega$ and $\rho<\om$, we set $\diamond_{\rho,\om}(\be) = \om$, $\ell_{\rho,\om}(\be)=\be-\rho+1$ for $\be \in [\rho,\omega)$, $a_{\rho,\om}(\omega,k) = \rho+k$ for $k \in \omega$, and $\ell_{\rho,\omega}(\omega) = 1$.
Now assume that $\la > \omega$ and that we have fixed mappings $\diamond_{\rho,\mu}$, $\ell_{\rho,\mu}$, and $a_{{\rho,\mu}}$ for every limit ordinal $\mu < \la$ and every isolated ordinal $\rho<\mu$.
We first choose an increasing sequence $\{\la_k\}_{k=0}^{\infty}$ of ordinals such that $\la_0=\rho$, for each $k$ we have that $\la_{k+1}$ is a limit ordinal
or $\la_{k+1} = \la_k + 1$, and $\lim_{k \to \infty} \la_k = \la$ as follows.
First we choose an increasing sequence $\{\xi_k\}_{k=0}^{\infty}$ of ordinals such that $\xi_0 = \rho$ and $\lim_{k \to \infty} \xi_k = \lambda$.
We set $\lambda_0 = \rho$ and suppose that $\lambda_k \leq \xi_k$ has already been defined.
If $(\lambda_k,\xi_{k+1}]$ is a nonempty finite set, then  we set $\lambda_{k+1} = \lambda_k + 1$.
If $(\lambda_k,\xi_{k+1}]$ is infinite, then we define $\lambda_{k+1}$ as the maximal limit ordinal in $(\lambda_k,\xi_{k+1}]$.
The above construction ensures that for any $k \in \omega$ there exists $k' \in \omega$ with $\xi_{k+1} \leq \lambda_{k'}$.

Then we put $\diamond_{\rho,\la}(\la_k) = \la$, $\ell_{\rho,\la}(\la_k) = k+1$, $\ell_{\rho,\la}(\la) = 1$, and $a_{\rho,\la}(\la,k)=\la_k$.
It remains to define the values $\diamond_{\rho,\la}(\beta)$ and $\ell_{\rho,\la}(\beta)$
for ordinal $\beta \in [\lambda_k + 1,\lambda_{k+1})$, $k \in \omega$,
and the values $a_{\rho,\la}(\beta,j)$ for $\beta \in \LIM([\lambda_k + 1,\lambda_{k+1}])$, $k,j \in \omega$.

If $\lambda_{k+1}$ is not a limit ordinal, then $[\la_k+1,\la_{k+1}) = \emptyset$ and $\LIM([\la_k+1,\la_{k+1}]) = \emptyset$
since $\lambda_{k+1} = \lambda_k + 1$. Suppose that $\la_{k+1}$ is a limit ordinal.
Using the induction hypothesis, we already have all the quantities $\diamond_{\la_k+1,\la_{k+1}}$, $\ell_{\la_k+1,\la_{k+1}}$, and $a_{\la_k+1,\la_{k+1}}$ defined. We define
\begin{alignat*}{2}
\diamond_{\rho,\la}(\be) &= \diamond_{\la_k+1,\la_{k+1}}(\be)  && \qquad \textrm{ for } \be \in [\la_k+1,\la_{k+1}), \\
\ell_{\rho,\la}(\be)     &= k+1+\ell_{\la_k+1,\la_{k+1}}(\be) && \qquad \textrm{ for } \be \in [\la_k+1,\la_{k+1}), \\
a_{\rho,\la}(\beta,j)    &= a_{\la_k+1,\la_{k+1}}(\be,j)       && \qquad \textrm{ for } \be \in \LIM([\la_k+1,\la_{k+1}]), j \in \omega.
\end{alignat*}
Thus $\diamond_{\rho,\la}$, $\ell_{\rho,\la}$, and $a_{\rho,\la}$ are defined.
We may easily verify that the required properties (a)--(d) are fulfilled following the inductive construction.
\end{proof}

\begin{notation}\label{N:diamant-el-a}
Let us recall that we fixed a countable limit ordinal $\al$ in Notation~\ref{N:alfa}.
For the rest of this paper we employ the following notation $\diamond := \diamond_{0,\alpha}$, $\ell := \ell_{0,\alpha}$, and $a := a_{0,\alpha}$.
\end{notation}

Figure~\ref{Fi:diamond} illustrates the behavior of the mappings $\diamond$ and $a$.

\begin{proposition}[$\diamond$-paths]\label{P:diamond-path}
For every $L \in \om, L >0$, there exists a unique $n(L) \in \omega$ and a unique finite sequence
$\{\beta_L^i\}_{i=0}^{n(L)}$ of ordinals such that
\begin{itemize}
\item[\upshape (a)] $0 = \be_L^{n(L)} < \dots < \be_L^0 = \al$,

\item[\upshape (b)] if $\beta_L^i$, $i < n(L)$, is isolated, then $\ell(\be^i_L) \le L$, $\be^i_L = \be^{i+1}_L+1$, and $\ell(\be^{i+1}_L)<L$.

\item[\upshape (c)] if $\beta_L^i$, $i < n(L)$, is limit, then $\ell(\beta_L^i)\le L$ and $\be_L^{i+1}$ is the unique ordinal with  $\diamond(\be_L^{i+1}) = \be_L^i$ and $\ell(\be_L^{i+1}) = L$.
\end{itemize}
\end{proposition}

\begin{proof}
Fix $L \in \om, L > 0$. We set $\beta_L^0 = \alpha$. Suppose that $\beta_L^i > 0$ with $\ell(\beta_L^i) \leq L$ has already been defined.
If $\be_L^i$ is isolated, we define $\be_L^{i+1}$ by $\be_L^{i+1}+1 = \be_L^{i}$.
By Lemma~\ref{L:diamond}(d) we have $\ell(\be^{i+1}_L) + 1 = \ell(\be^i_L)$. Thus we have $\ell(\be^{i+1}_L) < L$.
If $\be_L^i$ is a limit ordinal, then by Lemma~\ref{L:diamond}(b) there exists a unique $k \in \omega$ such that $\ell(a(\beta^i_L,k)) = L$.
We set $\beta_{L}^{i+1} = a(\beta^i_L,k)$.
In finite number $n(L)$ of steps we get $\be_L^{n(L)} = 0$ since the constructed sequence is strictly decreasing.
\end{proof}

\begin{figure}[h!]
\centering\includegraphics[scale=0.6,trim={20 250 50 385},clip]{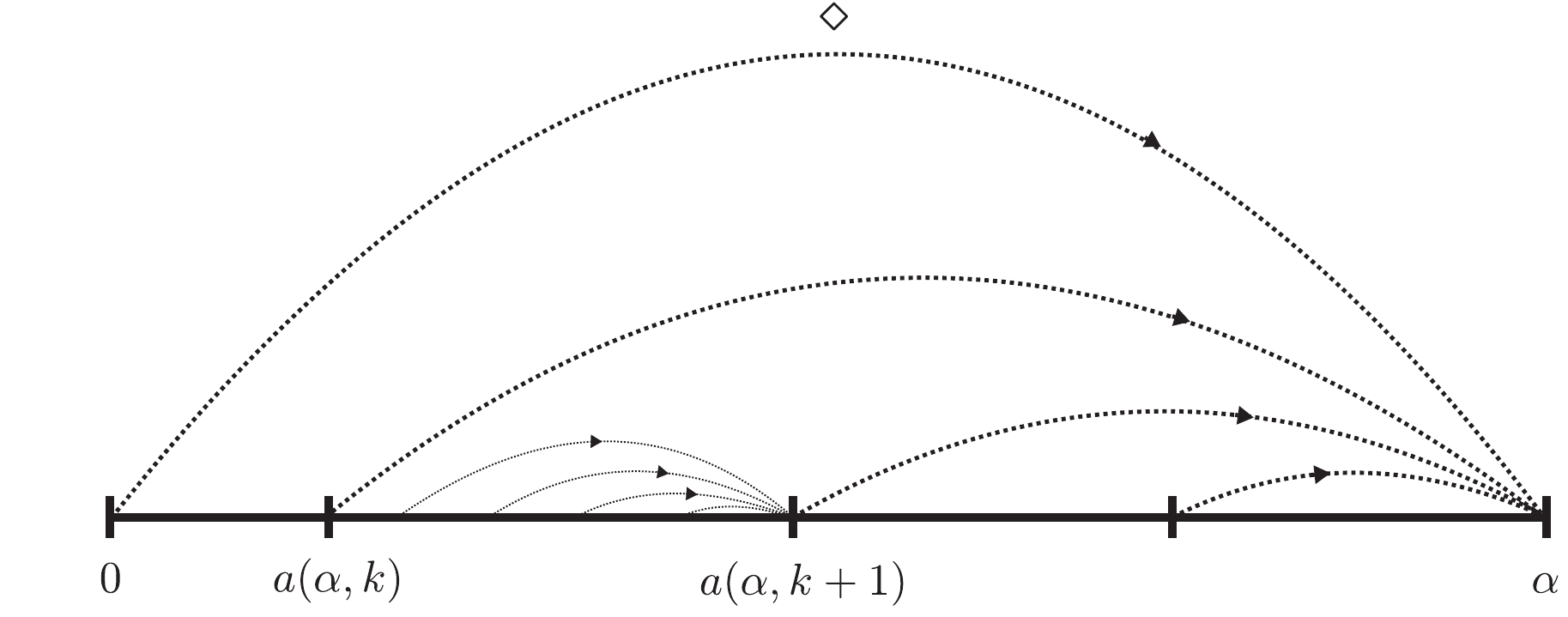}
\caption{}\label{Fi:diamond}
\end{figure}

\begin{lemma}\label{L:subsequence}
Let $L \in \omega, L>0$. Then $\{\beta_L^i\}_{i=0}^{n(L)}$ is a subsequence of $\{\beta_{L+1}^j\}_{j=0}^{n(L+1)}$.
\end{lemma}

\begin{proof}
We will proceed by downward induction on $i$. For $i = n(L)$ we have $\beta_{L}^{n(L)} = \beta_{L+1}^{n(L+1)} = 0$.
Suppose that $\beta_L^{i+1} = \beta_{L+1}^{j+1}$ for some $i < n(L)$ and $j < n(L+1)$.
We distinguish two possibilities.

If $\ell(\beta_L^{i+1}) = L$, then $\diamond(\be^{i+1}_L)=\be^i_L$ by Proposition~\ref{P:diamond-path}(c) and Lemma~\ref{L:diamond}(c) implies that $\ell(\gamma) > \ell(\beta_L^{i+1}) = L$ for every $\gamma \in (\beta_L^{i+1},\diamond(\beta_L^{i+1}))$.
Further, since $\ell(\beta^{j+1}_{L+1}) = L < L+1$, we have $\beta^j_{L+1} = \beta_{L+1}^{j+1}+1$ by Proposition~\ref{P:diamond-path}(b),(c). Consider the set $I = \{k \leq j\set \beta_{L+1}^k \leq \beta_L^i\}$. The set $I$ is nonempty since $j \in I$. Let $k_0$ be the minimum of $I$. Then by  Lemma~\ref{L:diamond}(c) and Proposition~\ref{P:diamond-path}(b),(c) we have $\beta_{L+1}^{k_0} = \beta_L^i$.

If $\ell(\be_L^{i+1}) = \ell(\beta_{L+1}^{j+1}) < L$, then $\beta^i_{L} = \beta_{L}^{i+1}+1 = \beta_{L+1}^{j+1}+1 = \beta^j_{L+1}$
by Proposition~\ref{P:diamond-path}(b),(c).
\end{proof}

\begin{lemma}\label{L:interval}
We have $\bigl\{\beta_L^j\set L > 0, j \in \{0,\dots,n(L)\}\bigr\} = [0,\alpha]$.
\end{lemma}

\begin{proof}
Set $B = \{\beta_L^j\set L > 0, j \in \{0,\dots,n(L)\}\}$. Towards contradiction assume that
there is $\gamma \in [0,\alpha] \setminus B$.
Since $\alpha \in B$, we have  $[\gamma,\alpha] \cap B \neq \emptyset$.  Thus we can define $\beta = \min ([\gamma,\alpha] \cap B)$.
Since $\beta \in B$ one can find $L > 0$ and $i \in \omega$ such that $\beta = \beta_L^i$.
By Proposition~\ref{P:diamond-path}(a), $\ga>0$ and $\beta_L^{i+1} < \gamma < \beta_L^i$.
Using Proposition~\ref{P:diamond-path}(b) and (c), we get that $\beta_L^i$ is a limit ordinal and $\diamond(\beta_L^{i+1}) = \beta_L^i$.
By Lemma~\ref{L:diamond}(b), there is a $k\in\om$ such that $\mu=a(\be^i_L,k) > \gamma$ and $\diamond(\mu) = \beta_L^i = \be$.
For $L':= \ell(\mu)$ we have $L'> L$ by Lemma~\ref{L:diamond}(c) and, using Lemma~\ref{L:subsequence}, $\beta_L^i = \beta_{L'}^j$ for some $j \in \{0,\dots,n(L')\}$.
This implies that $\beta_{L'}^{j+1} = \mu$ by Proposition~\ref{P:diamond-path}(c).
Consequently, using Lemma~\ref{L:diamond}(b), we get $\mu \in [\gamma,\be) \cap B$.
This is a contradiction with the definition of $\beta$.
\end{proof}

\begin{notation}\label{N:L}
If $0 \le \beta< \gamma \le\al$ are arbitrary, by Lemmas~\ref{L:subsequence} and \ref{L:interval} we find
the smallest $L' > 0$ such that $\{\be,\ga\} \subset \bigl\{\be^i_{L'} \set i \in \{0,\dots,n(L')\}\bigr\}$.
Let us denote it by $L(\be,\ga)$.
\end{notation}

\begin{lemma}\label{L:sequence}
For every $0 \leq \ga \leq \be \leq \alpha$ there exists a finite increasing sequence $\{\delta_j\}_{j=0}^n$ such that $\delta_0 = \gamma$, $\delta_n = \be$, $\delta_{j+1} = \diamond(\delta_j)$ whenever $\diamond(\delta_j) \leq \beta$ and $\delta_{j+1} = \delta_j + 1$ otherwise.
\end{lemma}

\begin{proof}

Let $0 \leq \gamma \leq \beta \leq \alpha$. We define $\delta_0 = \gamma$ and if $\delta_j$ is defined and $\delta_j < \beta$, then we set
\[
\delta_{j+1} =
\begin{cases}
\diamond(\delta_j) &\text{if } \diamond(\delta_j) \leq \beta, \\
\delta_j+1         &\text{if } \diamond(\delta_j) > \beta.
\end{cases}
\]
If the above sequence is finite, then we are done. Towards contradiction we assume that the sequence of $\delta_j$'s
is infinite. Denote $\beta' = \sup\{\delta_j \set j \in \omega\}$. By definition we have $\beta'\leq \beta$.
Since $\{\delta_j\}_{j \in \omega}$ is increasing, $\beta'$ is a limit ordinal. Thus we can find $k,j_0 \in \omega$ such that
$\delta_{j_0} \leq a(\beta',k) < \delta_{j_0+1}$.
If $\delta_{j_0} = a(\beta',k)$, then $\diamond(\delta_{j_0}) = \diamond(a(\beta',k)) = \beta' \leq \beta$.
Thus $\delta_{j_0+2} > \delta_{j_0+1} = \diamond(\delta_{j_0}) = \beta'$, a contradiction with the definition of $\beta'$.
Thus we have $\delta_{j_0} < a(\beta',k) < \delta_{j_0+1}$. This implies that $\delta_{j_0+1} = \diamond(\delta_{j_0})$.
By Lemma~\ref{L:diamond}(c) we get $\beta' = \diamond(a(\beta',k)) \leq \delta_{j_0+1} < \delta_{j_0+2}$, a contradiction with the definition of $\beta'$.
\end{proof}

\section{Recursive sets and functions}\label{S:recursive-background}

Effective descriptive set theory is one of the basic tools used in the proofs.
The basic material can be found in \cite[Chpt. 3 and 7]{Moschovakis}.
For the reader's convenience we recall in Subsection~\ref{SS:Moschovakis} the most important notations and results following \cite{Moschovakis}.
In Subsection~\ref{SS:auxiliary-spaces} we define some particular spaces and explain how to use results from Subsection~\ref{SS:Moschovakis} also in this context.
The classes $\Pi_1^0$ and $\Sigma_2^0$ are mentioned in Subsection~\ref{SS:higher-classes}.

\subsection{Notation and known facts}\label{SS:Moschovakis}

As \emph{basic spaces} we consider just $\om$ and $\cN = \om^{\om}$. Finite products $X = \prod_{i=0}^k X_i$, where $k \in \omega$,
and each $X_i$ is either $\om$ or $\cN$, are called \emph{product spaces}.
A product space $X$ is of \emph{type $0$} if $X = \omega^k$ for some $k \in \omega$, $k\ge 1$.

\medskip
\centerline{\emph{The letters $X, Y, Z$ stand for arbitrary product spaces in Section~\ref{S:recursive-background}.}}
\medskip

\noindent
The \emph{class of semirecursive sets} is denoted by $\Si^0_1$ and the symbol $\Si^0_1\restriction X$ denotes the family of $\Sigma_1^0$ subsets of $X$.

We recall how {\it relativized classes of semirecursive sets} can be defined.
First we introduce the following notation which we often use further on.

\begin{notation**}
A function $f$ is a \emph{partial function from $A$ to $B$} if the domain $D(f)$ of $f$ is contained in $A$ and the range of $f$ is contained in $B$,
in notation $f\colon A \rightharpoonup B$.
\end{notation**}

\begin{notation**}[sections]\label{N:sections}
Given a set $V\su A\times B$ we use the notation $V_a=\{b\in B\set (a,b)\in V\}$ and $V^b=\{a\in A\set (a,b)\in V\}$.
If $W\su A\times B\times C$, we use also the notation $W^c_a=\{b\in B\set (a,b,c)\in W\}$ and $W^{b,c} = (W^c)^b$.

Let $f\colon A \times B \rightharpoonup E$, $g\colon A \times B \times C \rightharpoonup E$, and $h\colon A \times B \times C \times D \rightharpoonup E$ be partial mappings. We use the notation $f_a$, $g^b$, $h^{c,d}$ etc., to denote the mappings $f_a\colon b \mapsto f(a,b)$, $g^c\colon (a,b) \mapsto g(a,b,c)$, $h^{c,d}\colon (a,b) \mapsto h(a,b,c,d)$ etc., respectively.
\end{notation**}

\begin{definition**}\label{D:relativizations}
Let $a,b \in \cN$. We define the relativizations of the class $\Si^0_1$ in $X$, cf. \cite[3D]{Moschovakis}, by
\begin{align*}
\Si^0_1(a)  \restriction X &= \{H^a \set H\in\Si^0_1\restriction (X\times\cN)\} \quad \text{and} \\
\Si^0_1(a,b)\restriction X &= \{H^{a,b} \set H\in\Si^0_1\restriction (X\times\cN^2)\}.
\end{align*}
\end{definition**}

\begin{notation**}\label{N:gamma}
In Section~\ref{S:recursive-background} the symbol $\Ga$ denotes one of the classes $\Si^0_1$, $\Si^0_1(a)$, or $\Si^0_1(a,b)$.
\end{notation**}

\begin{definition**}\label{D:recursive}
A set $R \subset X$ is called {\it $\Ga$-recursive} in $X$ if $R$ and $X \setminus R$ are in $\Ga$.
We also say that $R$ is {\it recursive} if $R$ is $\Si^0_1$-recursive. See \cite[3C]{Moschovakis}.
\end{definition**}

The canonical basis $\{N(X,j) \set j \in \om\}$ for an arbitrary product space $X$ is defined in \cite[3B]{Moschovakis}.
We will use the following descriptions of $\Ga$ sets (cf. \cite[3C.4, 3C.5]{Moschovakis}).

\begin{fact**}[representations of semirecursive sets]\label{F:Gstar} \hfil
\begin{enumerate}[(a)]
\item For every $\Ga$ set $H \subset X$ there is a $\Ga$-recursive set $R_H$ in $X\times\om$ such that
$y\in H$ if and only if there is $n\in\om$ such that $(y,n)\in R_H$.

\item A set $H \su X$ is in $\Ga\restriction X$ if and only if there is a $\Ga$ set $H^*$ in $\om$ such that
$
H = \bigcup_{j \in H^*} N(X,j).
$
\end{enumerate}
\end{fact**}

When proving semirecursivity or recursivity we often use the following known facts without explicit reference.

\begin{fact**}[preservation of semirecursivity]\label{F:properties-of-Gamma}
The family $\Gamma \restriction X$ contains all semirecursive subsets of $X$, it is closed with respect to finite unions, finite intersections,
$\exists^{\leq}$, $\forall^{\leq}$, $\exists^{\omega}$ (\cite[3C.1, 3D.7]{Moschovakis}),
and it is $\omega$-parametrized (cf. \cite[1D]{Moschovakis} for the definition and \cite[3F.6]{Moschovakis} with regard to Definition~\ref{D:relativizations} for the proof).
\end{fact**}

\begin{fact**}[preservation of recursivity of sets]\label{F:zachovani-recursivita}
The family of $\Gamma$-recursive sets in $X$ is closed with respect to complements, finite unions, finite intersections, $\exists^{\leq}$, and $\forall^{\leq}$
(see \cite[3C.3]{Moschovakis}).
\end{fact**}

We are going to recall the notion of $\Ga$-recursive functions.

\begin{definition**}\label{D:recursive-function}
A partial function $f\colon X \rightharpoonup Y$ is {\it $\Ga$-recursive on its domain} if and only if there exists
a $\Ga$ set $\De\su X\times\om$ which \emph{computes} $f$ on $D(f)$, i.e., for every $x \in D(f)$ we have
\[
\forall j \in \omega\colon f(x) \in N(Y,j) \Leftrightarrow (x,j)\in \De.
\]
We say that $f\colon X \rightharpoonup Y$ is \emph{$\Ga$-recursive} if it is $\Ga$-recursive on $D(f)$ and $D(f)\in \Ga \restriction X$.
We say that $f\colon X \rightharpoonup Y$ is \emph{recursive} if it is $\Si^0_1$-recursive.
\end{definition**}

\begin{fact**}[properties of the class of $\Ga$-recursive functions]\label{F:partialrecursivefunctions}\hfil
\begin{enumerate}[\upshape (a)]
\item\label{I:min} For every partial function $g\colon X\times\om\rightharpoonup\om$ which is $\Ga$-recursive on its domain, the partial function
\[
f(x) = \min\{i \in \omega \set g(x,i)=0 \land \forall j<i\colon g(x,j) > 0\}
\]
is $\Gamma$-recursive on its domain (\cite[7A.5]{Moschovakis}).

\item\label{I:recursivity-to-N} A partial function $f\colon X \rightharpoonup \cN$ is $\Ga$-recursive on its domain if and only if
$f(x)(n) = f^*(x,n)$, where $f^* \colon X\times\om\rightharpoonup \om$ is $\Ga$-recursive on $D(f^*) = D(f) \times \omega$
(\cite[3G.4(ii)]{Moschovakis}).

\item\label{I:coordinates} Let $l \in \omega$. A partial function $f\colon X \to Y_0 \times \dots \times Y_l$ is $\Gamma$-recursive on its domain if and only if
\[
f(x) = (f_0(x), \dots, f_l(x))
\]
with $f_0, \dots, f_l$ $\Gamma$-recursive on their (common) domain (\cite[3G.4(iii)]{Moschovakis}).

\item\label{I:recursivity-to-omega} A partial function $f \colon X \rightharpoonup \om$ is $\Ga$-recursive on its domain if and only if
there exists a $\Ga$ set $Q \su X \times \om$ such that $f(x) = n \Leftrightarrow (x,n) \in Q$ for every $x \in D(f)$ (\cite[3G.4(i)]{Moschovakis}).

\item\label{I:composition} The composition of partial functions which are $\Ga$-recursive on their domains is a partial function which is $\Ga$-recursive on its domain
(\cite[3G.1, 3G.2]{Moschovakis}).

\item\label{I:substitution-property}  The class $\Ga$ satisfies the \emph{substitution property}, i.e., for every partial function $f\colon X \rightharpoonup Y$, which is $\Ga$-recursive on its domain, and $Q \subset Y$ in $\Ga$, there is $P\su X$ in $\Ga$ such that $f^{-1}(Q) = P \cap D(f)$ (\cite[3G.2]{Moschovakis}).
\end{enumerate}
\end{fact**}

The following consequence of the above facts on preservation under ``recursive quantification'' turns out to be convenient.

\begin{lemma**}\label{L:bounded-quantifiers}
Let $V \subset X \times \omega$ be a $\Gamma$ set ($\Gamma$-recursive set, respectively) and $f\colon X \to \omega$ be $\Gamma$-recursive.
Then the sets
\begin{align*}
A &= \{x \in X \set \forall n \in \omega, n \leq f(x)\colon (x,n)\in V\} \quad \text{ and  } \\
B &= \{x \in X \set \exists n \in \omega, n \leq f(x)\colon (x,n)\in V\}
\end{align*}
are $\Gamma$ sets ($\Gamma$-recursive sets, respectively). If we replace the inequality $\leq$ by $<$ in the definition of $A$ and $B$, we get the same conclusion.
\end{lemma**}

\begin{proof}
Set $S = \{(x,n) \in X \times \omega \set n \leq f(x)\}$.
Using Fact~\ref{F:partialrecursivefunctions}\eqref{I:coordinates},\eqref{I:composition} and recursiveness of projections
(see \cite[3D.4(i)]{Moschovakis}), we get that the function $(x,n) \mapsto (n,f(x))$ is $\Gamma$-recursive.
The set $\{(n,m) \in \omega^2 \set n \leq m\}$ is recursive by \cite[3A.3]{Moschovakis}.
Thus using Fact~\ref{F:partialrecursivefunctions}\eqref{I:substitution-property}, we get that $S$ is a $\Gamma$ set.

First suppose that $V$ is in $\Gamma$. Then the set $\forall^{\leq} \, V$ is in $\Gamma$ by Fact~\ref{F:properties-of-Gamma}.
Since $A = \exists^{\omega} (S \cap \forall^{\leq} \, V)$ and $B = \exists^{\omega} (V \cap S)$, the sets $A$ and $B$ are
$\Gamma$ sets by Fact~\ref{F:properties-of-Gamma}.

Now assume that $V$ is $\Gamma$-recursive. It remains to verify that the complements of $A$ and $B$ are in $\Gamma$.
We may use the previous results since the complement of $V$ is in $\Gamma$ and
\begin{align*}
X \sm B  &= \{x \in X \set \forall n \in \omega, n \leq f(x)\colon (x,n) \notin V\} \quad \text{ and  } \\
X \sm A  &= \{x \in X \set \exists n \in \omega, n \leq f(x)\colon (x,n) \notin V\}.
\end{align*}

An analogous reasoning gives the result if we replace $\leq$ by $<$ in the definition of $A$ and $B$.
\end{proof}

\begin{fact**}[good parametrizations and recursion theorems]\label{F:recursionthms} \hfil
\begin{enumerate}[\upshape (a)]
\item We may associate with each $X$ and each space $Z$ of type $0$ a set $G^X_{\Ga} \subset \omega \times X$ in $\Ga$ and
a recursive function $S^{Z,X}_{\Ga} \colon \om \times Z \to \om$ such that
\begin{itemize}
\item[--] $G^X_{\Ga}$ is universal for $\Ga \restriction X$, i.e., a set $A \subset X$ is in $\Gamma$ if and only if there exists $e \in \omega$ with $A = (G^X_{\Ga})_e$,

\item[--] for every $e \in \omega$, $z \in Z$, and $x \in X$ we have
\[
(e,z,x) \in G^{Z\times X}_{\Ga} \Leftrightarrow \bigl(S^{Z,X}_{\Ga}(e,z),x\bigr) \in G^X_{\Ga},
\]

\item[--] if $H \in \Ga\restriction (\om\times X)$, then there is $e^*\in\om$ such that $H_{e^*} = (G^X_{\Gamma})_{e^*}$
\end{itemize}
(cf. \cite[3H.4]{Moschovakis}).

\item We may associate with spaces $X$, $Y$ and with any space $Z$ of type $0$ a partial function $U^{X,Y}_{\Ga}\colon \om\times X \rightharpoonup Y$, which is $\Ga$-recursive on its domain, and a recursive function $S^{Z,X,Y}_{\Ga}\colon \om\times Z\to\om$ such that
\begin{itemize}
\item[--] a partial function $f\colon X \rightharpoonup Y$ is $\Ga$-recursive on its domain if and only if there is $e\in\om$ such that
$f(x)=U^{X,Y}_{\Ga} (e,x)$ whenever $f(x)$ is defined,

\item[--] for every $e\in\om$, $z\in Z$, $x \in X$ with $(e,z,x) \in D(U^{Z\times X,Y}_{\Ga})$, we have
\[
U^{Z\times X,Y}_{\Ga}(e,z,x) = U^{X,Y}_{\Ga}(S^{Z,X,Y}_{\Ga}(e,z),x),
\]
and

\item[--] if $f\colon \om \times X \rightharpoonup Y$ is $\Gamma$-recursive on its domain, then there is $e^*\in\om$ satisfying $f(e^*,x)=U^{X,Y}_{\Ga}(e^*,x)$ whenever $f(e^*,x)$ is defined
(see \cite[7A.6]{Moschovakis}).
\end{itemize}
\end{enumerate}
\end{fact**}

Since the following lemma is essential to enable to work with nonrecursive ordinals, we explain how it follows from the above recalled facts.
Let us recall that the spaces of type $0$ are endowed with the discrete topology.

\begin{lemma**}[making sets and functions relatively recursive, cf. {\cite[3E.4]{Moschovakis}}]\label{F:relativization}
Let $H_n$ be open in $X_n$ and $f_n\colon Y_n \rightharpoonup Z_n$ be continuous,
where $X_n, Y_n, Z_n$ are product spaces for every $n \in \omega$.
Then there exists $a \in \cC$ such that for every $n \in \omega$ the set $H_n$ is in $\Sigma_1^0(a)$ and $f_n$ is $\Sigma_1^0(a)$-recursive on its domain.

Moreover, $a \in \cC$ can be found such that $H_n$ is $\Sigma_1^0(a)$-recursive for every $n \in \omega$ for which $H_n$ is clopen.
\end{lemma**}

\begin{proof}
Let us consider the sets
$H_n^*=\{j\in\om\set N(X_n,j)\su H_n\}$ and $\De_n=\{(y,l)\in Y_n\times\om\set f_n(y)\in N(Z_n,l)\}$.

Since the sets $H_n$ are open and the sets $N(X_n,j)$, $j\in\om$, form a basis of $X_n$, we have $H_n = \bigcup_{j\in H_n^*} N(X_n,j)$.
Due to Fact~\ref{F:Gstar}(b), all $H_n$'s are $\Sigma_1^0(\si)$ sets if $\si\in\cN$ is such that each $H_n^*$ is a $\Sigma_1^0(\si)$ subset of $\om$.

Since the sets $N(Z_n,l)$, $l\in\om$, form a basis of $Z_n$, the set $\De_n$ computes $f_n$ on its domain.
As each $f_n$ is continuous, the set $\De_n$ is relatively open in $D(f_n)\times \om$.
Thus there is an open set $\wt{\De}_n$ in $Y_n\times\om$ such that $\De_n=(D(f_n)\times\om)\cap \wt{\De}_n$.
So the set $\wt{\De}_n$ computes $f_n$ and if the set
$\De^*_n = \{j \in \om \set N(Y_n \times \om,j) \su \wt{\De}_n\}$ is $\Sigma_1^0(\si)$ set for some $\si \in \cN$,
then $f_n$ is $\Sigma_1^0(\si)$-recursive on its domain.

Let us consider the subset $P$ of $\om^2$ such that $P_{2n}=\De^*_n$ and $P_{2n+1}=H_n^*$.
By \cite[3E.4]{Moschovakis} there is $\si \in \cN$ and $G \subset \cN \times \omega^2$ in $\Sigma_1^0$ such that $P = G_\si$.
Therefore $P$ is a $\Sigma_1^0(\si)$ set in $\om^2$.
It follows that every $P_{2n}=\De^*_n$ is in $\Sigma_1^0(\si)$ and every $P_{2n+1}=H_n^*$ is in $\Sigma_1^0(\si)$.
Thus all sets $H_n$ are in $\Si^0_1(\si)$ and all functions $f_n$ are $\Sigma_1^0(\si)$-recursive.

Having $\si = (\si_k)_{k=0}^{\infty} \in \cN$, we consider $a\in\cC$ defined as $a=s_0\w s_1\w s_2\w\dots$, where $s_{2k}$ is the constant zero sequence of length $\si_k$ and $s_{2k+1}=(1)$ for $k\in\om$. It is not difficult to verify that $\si$ is recursive in $a$.
Now, we use the fact that each $\Si^0_1(\si)$ subset of $\om^2$ is in $\Si^0_1(a)$ provided $\si$ is recursive in $a$, see \cite[3E.15]{Moschovakis}.

The last assertion follows from the previous one. It is sufficient to add complements of those $H_n$'s, which are clopen, to the list of considered sets.
\end{proof}

\subsection{Recursive sets and functions on auxiliary spaces}\label{SS:auxiliary-spaces}

Here we add several auxiliary spaces to the basic spaces $\omega$ and $\cN$ and we show how to use results from the previous subsection for their finite products.

\begin{notation**}[$\cS$, $\cS^*$, $\cP$, $\cP_\infty$]\label{D:sequences}\hfill\break
\noindent
(a) The space of all finite sequences of elements of $\om$ including the empty sequence is denoted by $\cS$.
The empty sequence is denoted by $\emptyset$.
Let $s, t \in \Seq$ and $\nu \in \Seq \cup \cN$. We denote the concatenation of sequences $s$ and $t$ by $s \w t$
and we write $s \preceq \nu$ if $\nu$ is an extension of $s$.
We use $\cN(s)$ to denote the set $\{\si\in\cN\set s \preceq \si\}$.
The symbol $|s|$ denotes the length of $s$.
By $s_i$ or $s(i)$ we denote the $i$th coordinate of $s$, where $i < |s|$.
If $n \in \omega$ and $|s| \ge n$, then the symbol $s|n$ denotes the finite sequence $(s_0,s_1,\dots,s_{n-1})$.

We add an auxiliary element ``$-$'' to the set $\Seq$ and we denote $\Seqm = \Seq \cup \{-\}$.
Concatenation of elements of $\Seqm$ is defined as above, where ``$-$'' is interpreted as the empty sequence.
The length of ``$-$'' is $0$.

\medskip\noindent
(b) The set of all finite sequences of elements of $\Seqm$ is denoted by $\cP$. We denote $\cP_{\infty} = (\Seqm)^{\om}$.
The length of $p \in \cP$ is denoted by $|p|$.
The $i$th coordinate of $p \in \cP\cup\cP_{\infty}$ is denoted by $p_i$ or $p(i)$ for $i<|p|$ or $i\in\om$ if $p\in\cP_\infty$.
The restriction $p|n$ is defined analogically as in the previous item, too.
The extension relation $\preceq$ on $\cP \cup \cP_{\infty}$ is defined in an obvious way.
The set $\{\pi\in \cP_{\infty} \set p \preceq \pi\}$ for $p \in \cP$ is denoted by $\cP_{\infty}(p)$.
The concatenation of $p,q \in \cP$ is an element of $\cP$ and is denoted by $p \w q$.
If $p \in \cP \cup \cP_{\infty}$, then the symbol $\wh{p}$ denotes the element of $\cS \cup \cN$, which is the concatenation of $p_0, p_1, \dots$ in the sense defined in the previous item.
\end{notation**}

To define recursive and semirecursive sets also in $\cS$, $\cS^*$, $\cP$, $\cP_\infty$, we choose and fix bijections between these spaces and the
basic ones. We do it in such a way that Lemma~\ref{L:bounded-quantifiers} can be easily applied.

\begin{notation**}[identifications with basic spaces]\label{N:bijectionsandGamma} \hfill
\begin{enumerate}[(a)]
  \item Let $c_\om$ be the identity mapping on $\om$.

  \item Let $c_\cN$ be the identity mapping on $\cN$.

  \item The partial ordering $<_*$ on $\cS$ is defined by $s<_* t$ if and only if $|s| \le |t|$, $s_i \le t_i$ for all $i < |s|$, and $s \ne t$.
    Let a bijection $c_{\cS^*}\colon \cS^*\to\om$ be chosen such that $c_{\cS^*}(-)=0$, $c_{\cS^*}(\emptyset)=1$, and $c_{\cS^*}(s) < c_{\cS^*}(t)$ whenever $s <_* t$, see Remark~\ref{R:ordering-star-explanation}.

  \item Let the bijection $c_\cS\colon\cS\to\om$ be defined by $c_\cS(s) = c_{\cS^*}(s)-1$.

  \item Let the bijection $c^\cS_{\cP}\colon \cP \to\cS$ be defined by $c^\cS_{\cP}(p) = \bigl(c_{\cS^*}(p_i)\bigr)_{i < |p|}$.
  We have that $c^\cS_{\cP}(p) \preceq c^\cS_{\cP}(q)$ if $p\preceq q$.  We also define the bijection $c_\cP = c_\cS \circ c^\cS_{\cP}\colon \cP\to\om$.

  \item Let the bijection $c_{\cP_\infty}\colon \cP_\infty\to\cN$ be defined by $c_{\cP_\infty}(\pi) = \bigl(c_{\cS^*}(\pi_i)\bigr)_{i=0}^\infty$.

  \item Let $c_{[0,\al]}\colon [0,\al]\to\om$ be an arbitrary bijection.
\end{enumerate}
\end{notation**}

\begin{remark**}\label{R:ordering-star-explanation}
We explain why $c_{\cS^*}\colon \cS^*\to\om$ with the properties from (c) really exists. We define $c_{\cS^*}^{-1}(n)\in\cS^*$ by induction on $n\in\om$.
Let $c_{\cS^*}^{-1}(0)=-$ and $c_{\cS^*}^{-1}(1)=\emptyset$.
Let $(s_k)_{k \geq 1}$ be an enumeration of nonempty elements of $\cS$.
Suppose that $n\ge 2$ and we already have defined the elements $c_{\cS^*}^{-1}(l)$ for $l < n$.
Let $k \geq 1$ be the smallest integer $k$ such that $s_k\notin\{c_{\cS^*}^{-1}(l)\set l<n\}$. We find a minimal element $s_{k(n)}$ of the finite set
$\{s_l\in\cS\set s_l\le_* s_k\}\sm\{c_{\cS^*}^{-1}(l)\set l<n\}$ with respect to the partial ordering $<_*$ and define $c_{\cS^*}^{-1}(n)=s_{k(n)}$.
It is easy to check that each $s_k$, $k\geq 1$, was chosen for some $n>1$ by this procedure and the extra condition on the order $<_*$ is satisfied.
\end{remark**}

We will use \emph{well-orderings} of $\cS$, $\cS^*$, and especially of $\cP$ defined as follows.

\begin{notation**}[ordering of $\cS^*$ and $\cP$]\label{N:usporadanicP}\hfill
\begin{enumerate}[(1)]
\item For $s,t\in\cS^*$ we define $s<_{\cS^*} t$ whenever $c_{\cS^*}(s)<c_{\cS^*}(t)$.
Of course, the restriction of $<_{\cS^*}$ to $\cS^2$ coincides with the partial order defined by $c_{\cS}(s)<c_{\cS}(t)$ due to the definition of $c_\cS$. By the choice of $c_{\cS^*}$ in Notation~\ref{N:bijectionsandGamma}(c), we have that $s<_{\cS^*} t$ if $s<_* t$ for $s,t\in\cS$.

\item For $p,q \in \cP$ we define $p <_{\cP} q$ whenever $c_\cP(p) < c_\cP(q)$, or equivalently $c^\cS_\cP(p) <_\cS c^\cS_\cP(q)$.
The item (1) implies that $c_\cP(p) < c_\cP(q)$ if $c^\cS_\cP(p) <_* c^\cS_\cP(q)$.
We may easily check that the following properties follow from the properties of the mappings
$c_{\cS^*}$, $c_\cS$ and the definition of $c_\cP$.
\begin{itemize}
\item [(a)] The number of elements of the set
$\{q \in \cP \set q <_{\cP} p\}$ is bounded by $c_\cP(p)$,

\item [(b)] if $p,q \in \cP$ and $q \preceq p$, then $q \leq_{\cP} p$,

\item [(c)] if $p\in \cP$, $m \in \om$, and $\eta\in\cS$, then $p \w ((m),\emptyset) \le_{\cP} p \w ((m),\eta)$.
\end{itemize}
\end{enumerate}
\end{notation**}

\begin{definition**}\label{D:Ga-in-extended-spaces}\hfil
\begin{enumerate}[(a)]
\item By \emph{extedend basic space} we mean any of the spaces $\om, \cN, \cS, \cS^*, \cP, [0,\al]$, and $\cP_{\infty}$.

\item Every space of the form $\cX = \prod_{i=0}^k \cX_i$, where $k \in \om$ and $\cX_i$ is an extended basic space for every $i \leq k$,
is called \emph{extended product space} and we define the mapping $c_{\cX}$ on $\cX$ by $c_{\cX}\bigl((x_i)_{i \leq k}\bigr) = \bigl(c_{\cX_i}(x_i)\bigr)_{i\leq k}$.
If $\cX_i \in \{\om,\cS,\cS^*,\cP,[0,\al]\}$ for every $i \leq k$, we say that $\cX$ is of type $0$.
Then for every extended product space $\cX$ we extend the class $\Ga$ by
\begin{equation*}
\Ga\restriction \cX = \{\cH\su \cX\set c_{\cX}(\cH) \in \Ga \restriction \cX^c\},
\end{equation*}
where $\cX^c = c_{\cX}(\cX)$ is the corresponding product space.

\item Let $\cX, \cY$ be extended product spaces and $f\colon \cX \rightharpoonup \cY$.
We say that $f$ is \emph{$\Gamma$-recursive on its domain} if the mapping
$c_{\cY} \circ f \circ c_{\cX}^{-1} \colon \cX^c \rightharpoonup \cY^c$ is $\Gamma$-recursive on its domain,
where $\cX^c$ and $\cY^c$ correspond to $\cX$ and $\cY$ as above.
We say that $f$ is \emph{$\Gamma$-recursive} if $f$ is $\Gamma$-recursive on $D(f)$ and $D(f) \in \Gamma\restriction \cX$.
\end{enumerate}
\end{definition**}

\begin{remark**}\label{R:recursivity-of-cP}
Observe that if $\cX$ is an extended product space, then the mapping $c_{\cX}\colon \cX \to \cX^c$ is $\Gamma$-recursive.
\end{remark**}

\begin{remark**}\label{R:extended-facts}
Now we can formulate statements on extended product spaces $\cX,\cY,\cZ,\dots$ reformulating the statements on the corresponding product spaces $X,Y,Z,\dots$ (see Definition~\ref{D:Ga-in-extended-spaces}), e.g., those recalled in \ref{SS:Moschovakis}, in an obvious way.
In particular, we define for any extended product space $\cX$ and for any extended product space $\cZ$ of type $0$ also the corresponding
\[
G^{\cX}_\Ga =\{(e,x)\in\om\times\cX\set (e,c_{\cX}(x))\in G^X_\Ga\} \quad  \textrm{ and } \quad
S^{\cZ,\cX}_{\Ga}(e,z) = S^{Z,X}_{\Ga}(e,c_{\cZ}(z))
\]
and we get the analogous properties from Fact~\ref{F:recursionthms}(a) for them.

Similarly we define for the extended product spaces $\cX$, $\cY$, and for $\cZ$ of type $0$ also
\[
U^{\cX,\cY}_{\Ga}(e,x) = c_{\cY}^{-1} \circ U^{X,Y}_{\Ga}(e,c_{\cX}(x)) \quad  \textrm{ and } \quad
S^{\cZ,\cX,\cY}_{\Ga}(e,z) = S^{Z,X,Y}_{\Ga}(e,c_{\cZ}(z)),
\]
and we get analogous properties from Fact~\ref{F:recursionthms}(b) for them.
\end{remark**}

\begin{notation**}[natural bases of extended product spaces]
In the particular case of $\cN$ the fixed presentation in \cite{Moschovakis} leads to the canonical basis $N(\cN,j)$, $j\in\om$.
We however prefer to work with the natural basis in $\cN$ and also in the spaces $\omega, \cS^*, \cP, \cP_{\infty}, [0,\alpha]$,
namely, with the family $\{\cN(s) \set s \in \cS\}$ in $\cN$, $\{\cP_{\infty}(p) \set p \in \cP\}$ in $\cP_{\infty}$,
and $\bigl\{\{x\} \set x \in \cX\bigr\}$ if $\cX$ is a space of type $0$.

For any extended basic space $\cX$ we define $\cX^0$ as follows.
\begin{enumerate}
  \item $\cX^0 = \cX$ if $\cX$ is of type $0$,

  \item $\cX^0 = \cS$ if $\cX = \cN$,

  \item $\cX^0=\cP$   if $\cX=\cP_\infty$.
\end{enumerate}
Moreover, for each $s \in \cX^0$ we define $\cX(s)=\{s\}$ if $\cX$ is of type $0$ and we of course have $\cX(s) = \cN(s)$ if $\cX = \cN$, and
$\cX(s)  =\cP_\infty(s)$ if $\cX = \cP_{\infty}$.

If $\cX = \prod_{i \leq k} \cX_i$, where $\cX_i$, $i \leq k$,
is an extended basic space, then $\cX^0=\prod_{i\le k} \cX_i^0$ and
$\cX(s) = \cX(s_0,\dots, s_k) = \prod_{i\leq k} \cX_i(s_i)$, where $s = (s_0,\dots,s_k) \in \cX^0 = \prod_{i \leq k} \cX_i^0$.
\end{notation**}

\begin{notation**}[bases of product spaces] Let $\{N(X,j) \set j \in \omega\}$ be the fixed canonical basis for a product space $X$
(see the note after Definition~\ref{D:recursive}).
Since we have
\[
\{N(X,j) \set j \in \omega\} \setminus \{\emptyset\} = \{X(s) \set s \in X^0\},
\]
there exists a unique partial function  $s^X\colon \omega \rightharpoonup X^0$ such that $D(s^X) = \{j \in \omega \set N(X,j) \neq \emptyset\}$ and $N(X,j) = X(s^X(j))$ whenever $j \in D(s^X)$.
\end{notation**}

\begin{notation**}[identification of $\omega$ and $\omega^2$]\label{N:identification}
We use a fixed bijection $n\mapsto [n] = ([n]_0,[n]_1)$ of $\om$ and $\om\times\om$
such that the  mapping $n\mapsto [n]_1$ is increasing on $\{n\in\om\set [n]_0=k\}$ for every $k\in\om$.
The corresponding inverse mapping is denoted $[\cdot,\cdot]^{-1}$.
Thus, if $s \in \Seq$ and $k \in \omega$, the formula $j \mapsto s([k,j]^{-1})$ defines a finite (possibly empty) sequence.
Moreover, we use the following coordinatewise application of $[\cdot]$.
For $i \in \{0,1\}$ and $t \in \Seq \cup \cN$ we define the sequence $[t]_i \in \Seq \cup \cN$ by $[t]_i(j) = [t(j)]_i$ for $j = 0, 1, \dots$.
\end{notation**}

We need and use (sometimes without quotation) well known facts concerning recursivity and semirecursivity of sets and functions on basic spaces which can be found in \cite{Moschovakis}.
Using Lemma~\ref{F:relativization}, we get Lemma~\ref{L:volba-epsilon} which ensures relative recursivity in an $\ep\in\cN$ of a number of sets and functions.
In fact, the only relations and functions in the lemma which need the relativization because of essential reasons are the statements
related to $[0,\al]$ to get our main results for the (not necessarily recursive) ordinal $\al$.

\begin{lemma**}\label{L:volba-epsilon}
There is $\ep \in \cN$ such that the following statements hold.
\begin{enumerate}[\upshape (a)]
\item The relation of equality on each extended product space of type $0$ is $\Si^0_1(\ep)$-recursive.

\item The projections $x \in \prod_{k=0}^n \cX_k \mapsto x_i\in\cX_i$, $i \leq n$, of the extended product spaces to the factors are $\Si^0_1(\ep)$-recursive.

\item The partial function $s^{X}$ is $\Si^0_1(\ep)$-recursive and moreover its domain is $\Si^0_1(\ep)$-recursive for every product space $X$.

\item The relations $\preceq$ on $\cS$ and $\preceq$ on $\cP$ are $\Si^0_1(\ep)$-recursive.

\item The following functions and their domains are $\Si^0_1(\ep)$-recursive.
\begin{itemize}
  \item $[\cdot]\colon \om \to \om^2$ from Notation~\ref{N:identification},
  \item $s\in\cS\sm\{\emptyset\}\mapsto \max\{s_i\set 0\le i <|s|\}$, $(i,n)\in\om^2\mapsto \cl{n}|i$ (see
        Notation~\ref{N:notation-1}),
  \item $(i,s) \in \om \times \cS^*\mapsto p \in \cP$, where $|p| = i$ and $p_j = s$ for $j < |p|$,
  \item $(n,s)\in\om\times\cS\mapsto s|n\in\cS$ for $n\le |s|$, $(n,p)\in\om\times\cP\mapsto p|n\in\cP$ for $n\le |p|$,
  \item $(i,s)\in\om\times\cS\mapsto s_i\in\om$ for $i < |s|$, $(i,p) \in \om \times\cP \mapsto p_i\in\cS^*$ for $i < |p|$,
  \item $s \in \cS^* \mapsto |s| \in \om$, $p \in \cP \mapsto |p| \in \om$,
  \item $(s,t)\in (\cS^*)^2 \mapsto s \w t\in\cS$, $p \in \cP \mapsto \wh{p}\in\cS$, $(p,q) \in \cP^2 \mapsto p \w q \in \cP$.
\end{itemize}

\item The sets $\LIM(\alpha)$ in $[0,\al]$ and $\{(\ga,\be)\in [0,\al]^2\set \ga < \be\}$ in $[0,\al]^2$ are $\Si^0_1(\ep)$-recursive.

\item The following functions and their domains are $\Si^0_1(\ep)$-recursive.
      \begin{itemize}
        \item $\ell\colon [0,\al]\to\om$, $\diamond\colon [0,\al)\to [0,\al]$, $L\colon \{(\be,\ga)\in [0,\al]^2 \set \be<\ga\}\to\om$,
        \item $\be \mapsto \be + 1$ on $[0,\al)$, and $a\colon \LIM(\al) \times \om \to [0,\al]$.
      \end{itemize}

\item The clopen relation $\preceq \ \subset \cS \times \cN$ is $\Si^0_1(\ep)$-recursive.

\item The following continuous partial functions are $\Si^0_1(\ep)$-recursive.
\begin{itemize}
  \item $(n,\si)\in\om\times\cN \mapsto \si_n \in \om$, $(n,\pi)\in\om\times\cP_{\infty}\mapsto \pi_n \in \cS^*$,
  \item $(n,\si)\in\om\times\cN \mapsto \si|n \in \cS$, $(n,\pi)\in\cP_\infty \mapsto \pi|n\in \cP$, and
  \item $\pi \in \cP_{\infty} \mapsto \wh{\pi} \in \cN$, $(p,\pi) \in \cP \times \cP_\infty\mapsto p\w\pi \in \cP_{\infty}$.
\end{itemize}
\end{enumerate}
\end{lemma**}

\begin{notation**}\label{N:epsilon}
We fix one such $\ep$ from Lemma~\ref{L:volba-epsilon} and we use $\Ga(\ep)$ for $\Si^0_1(\ep)$ or $\Si^0_1(a,\ep)$ for $a\in \cN$ in the rest of this section.
\end{notation**}

\newpage

\begin{lemma**}\label{L:natural-base}\hfil
\begin{enumerate}[\upshape (a)]
\item A set $\cH\su \cX$ is a $\Ga(\ep)$ set in an extended product space $\cX$ if and only if there is a $\Ga(\ep)$ set $\cH^*$ in $\cX^0$ such that
\[
\cH=\bigcup \bigl \{\cX(s)\set s\in\cH^*\bigr\}.
\]

\item For every extended product spaces $\cX$ and $\cY$, a mapping $f\colon \cX \rightharpoonup \cY$ is
$\Ga(\ep)$-recursive on its domain if and only if there is a $\Ga(\ep)$ set $\De\su \cX\times\cY^0$ which computes $f$ on its domain, i.e.,
for every $\xi \in D(f)$ we have
\[
\forall t \in \cY^0 \colon f(\xi) \in \cY(t) \Leftrightarrow  (\xi,t) \in \De.
\]
\end{enumerate}
\end{lemma**}

\begin{proof}
(a) We assume that $\cX$ is an extended product space and we use the notation $\cX^c = c_{\cX}(\cX)$.
We are going to prove the following equivalences.
\[
\begin{split}
\cH\in \Ga(\ep)\restriction \cX
&\Leftrightarrow H := c_{\cX}(\cH)\in\Ga(\ep)\restriction \cX^c \\
&\Leftrightarrow \exists \wt{H} \in\Ga(\ep)\restriction\om   \colon   H=\bigcup_{j\in \wt{H}} N(\cX^c,j)\\
&\Leftrightarrow \exists     H^*\in\Ga(\ep)\restriction (\cX^c)^0 \colon   H = \bigcup_{u \in H^*} \cX^c(u)\\
&\Leftrightarrow \exists   \cH^*\in\Ga(\ep)\restriction\cX^0 \colon \cH=\bigcup_{s\in  \cH^*} \cX(s).
\end{split}
\]
The first equivalence is just our definition of $\Ga(\ep)\restriction\cX$ in Definition~\ref{D:Ga-in-extended-spaces}(b).
The second equivalence follows from Fact~\ref{F:Gstar}(b).
The third equivalence uses the fact that $N(\cX^c,j)=\cX^c(s^{\cX^c}(j))$ and the mapping $s^{\cX^c}$ is $\Ga(\ep)$-recursive on its $\Ga(\ep)$-recursive domain $D(s^{\cX^c})\su\om$ by Lemma~\ref{L:volba-epsilon}(c). For the implication $\Rightarrow$, we put $H^*=s^{\cX^c}(\wt{H}\cap D(s^{\cX^c}))$. For the converse implication, we put $\wt{H}=(s^{\cX^c})^{-1}(H^*)$.
In both definitions the class $\Ga(\ep)$ is preserved. In the latter case, we may apply Fact~\ref{F:partialrecursivefunctions}(\ref{I:substitution-property}) to the preimage of $H^*$ by $s^{\cX^c}$.
The range of $\wt{H}$ by $s^{\cX^c}$ is the set
\[
\{u\in(\cX^c)^0\set \exists j\in\om\colon j\in D(s^X)\cap\wt{H} \wedge s^{\cX^c}(j)=u\}.
\]
Thus it is the projection along the type 0 space $(\cX^c)^0$ of a $\Ga(\ep)$ subset of $\om\times (\cX^c)^0$
(cf. Fact~\ref{F:properties-of-Gamma} and Remark~\ref{R:extended-facts}).

It remains to prove the last equivalence. For any extended basic space $\cX$ of type 0,
we have $\cX=\cX^0$, $\cX^c=(\cX^c)^0=\om$, and
$\cX^c(c_{\cX}(s))=\{c_{\cX}(s)\}=c_{\cX}(\cX(s))=c_{\cX}(\{s\})$ for $s\in\cX=\cX^0$.
So $c_{\cX}$ is the bijection which computes the set $H^*$ from $\cH^*$ and conversely.
The recursivity of each of them is equivalent to the recursivity of the other one by Definition~\ref{D:Ga-in-extended-spaces}.

For $\cX=\cN$ we have $\cX^c = \cN$, $c_{\cX}$ is the identity, $\cX^0 = (\cX^c)^0 = \cS$, and $\cX^c(s) = c_{\cX}(\cX(s))$
for $s\in\cS=\cX^0$ shows that identity is the mapping which computes $H^*$ from $\cH^*$ and conversely.
The equivalence of the recursivity of $H^*$ and $\cH^*$ is thus trivial.

For $\cX=\cP_\infty$ we have $\cX^c=\cN$, $\cX^0=\cP$, and
\[
\cN(c_{\cP}^{\cS}(p)) = \cN\bigl((c_{\cS^*}(p_i))_{i< |p|}\bigr) = c_{\cP_\infty}(\cP_\infty(p))
\]
for $p\in\cP=\cX^0$, so $c_{\cP}^{\cS}$ is a bijection which computes $H^*$ from $\cH^*$ and conversely in this case. The equivalence of the recursivity of $H^*$ and $\cH^*$ follows from Definition~\ref{D:Ga-in-extended-spaces} since $c_{\cP}^{\cS}=c_{\cS}^{-1}\circ c_{\cP}$ (see Notation~\ref{N:bijectionsandGamma}(e)).

For extended product spaces the last equivalence can be deduced applying the discussed cases of various extended basic spaces using them coordinate-wise.

\medskip\noindent
(b) We assume that $\cX$, $\cY$ are extended product spaces, and $f\colon \cX\rightharpoonup\cY$ is a partial mapping. As before we use the notation $\cX^c = c_{\cX}(\cX)$ and $\cY^c = c_{\cY}(\cY)$.
We are going to prove the following equivalences.
\[
\begin{split}
f\colon &\cX\rightharpoonup\cY  \textrm{ is recursive on its domain } D(f)\\
\Leftrightarrow & f^c :=c_{\cY}\circ f\circ c_{\cX}^{-1}\colon \cX^c\rightharpoonup \cY^c\textrm{ is recursive on its domain } D(f^c)=c_{\cX}(D(f))\\
\Leftrightarrow & \exists \wt{\De}^{f^c} \in \Ga(\ep)\restriction (\cX^c\times \om)\colon
\bigr[\forall (x,j)\in D(f^c)\times\om\colon f^c(x)\in N(\cY^c,j)\Leftrightarrow (x,j)\in\wt{\De}^{f^c} \bigl]\\
\Leftrightarrow & \exists \De^{f^c} \in \Ga(\ep)\restriction \bigl(\cX^c\times (\cY^c)^0\bigr)\colon
\bigr[\forall (x,v)\in D(f^c)\times (\cY^c)^0\colon f^c(x)\in \cY^c(v) \Leftrightarrow (x,v) \in \De^{f^c}\bigl]\\
\Leftrightarrow & \exists\De^f\in \Ga(\ep)\restriction (\cX\times \cY^0)\colon
\bigr[\forall (\xi,t)\in D(f)\times \cY^0\colon f(\xi)\in\cY(t)\Leftrightarrow (\xi,t) \in \De^f\bigl].
\end{split}
\]

The first equivalence is Definition~\ref{D:Ga-in-extended-spaces}(c). The second equivalence is Definition~\ref{D:recursive-function}.
The third equivalence uses the facts that $N(\cY^c,j)=\cY^c(s^{\cY^c}(j))$ for every $j\in D(s^{\cY^c})\su\om$ and that the mapping $s^{\cY^c}$ is $\Ga(\ep)$-recursive on its recursive domain $D(s^{\cY^c})\su\om$ by Lemma~\ref{L:volba-epsilon}(c).
For the implication $\Rightarrow$, we put $\De^{f^c}=(\id_{\cX^c}\times s^{\cY^c})(\wt{\De}^{f^c})$,
where the mapping $\id_{\cX^c}\times s^{\cY^c}$ is defined by
$(\id_{\cX^c}(x),s^{\cY^c}(j))$ for every $(x,j)\in\wt{\De}^{f^c}$. For the opposite implication we put
$\wt{\De}^{f^c}=(\id_{\cX^c}\times s^{\cY^c})^{-1}(\De^{f^c})$.
In both cases we get a $\Ga(\ep)$-recursive set. Indeed, the mappings $\id_{\cX^c}$ is $\Ga(\ep)$-recursive and $s^{\cY^c}$ is $\Ga(\ep)$-recursive on its recursive domain in $\om$ by Lemma~\ref{L:volba-epsilon}(c). Thus, using moreover $\Ga(\ep)$-recursivity of the projections of $\cX \times \om$,  the mapping $\id_{\cX^c} \times s^{\cY^c}$ is $\Ga(\ep)$-recursive on its domain
by Fact~\ref{F:partialrecursivefunctions}(b),(e).
Then the range of $\id_{\cX^c}\times s^{\cY^c}$ is the set
\[
\bigr\{(x,v)\in \cX^c\times (\cY^c)^0\set \exists j\in\om\colon s^{\cY^c}(j)=v\bigr\}
\]
which is in $\Ga(\ep)$ by Fact~\ref{F:properties-of-Gamma} since it is the projection of a $\Ga(\ep)$ set along $\om$.
The preimage under $\id_{\cX^c}\times s^{\cY^c}$ of $\De^{f^c}$ is $\Ga(\ep)$-recursive due to Fact~\ref{F:partialrecursivefunctions}(\ref{I:substitution-property}).

It remains to prove the last equivalence.
For any extended basic space $\cY$ of type 0, we have $\cY = \cY^0$, $\cY^c = (\cY^c)^0 = \om$,
and $c_{\cY} = c_{\cY^0}$. This gives
\[
\cY^c(c_{\cY}(v)) = \cY^c(c_{\cY^0}(v)) = \{c_{\cY^0}(v)\} = \{c_{\cY}(v)\} = c_{\cY}(\cY(v))
\]
for any $v \in \cY^0$.
For the implication $\Rightarrow$, we set $\De^{f} = (c_{\cX} \times c_{\cY})^{-1}(\De^{f^c})$.
The mapping  $c_{\cX} \times c_{\cY}$ is $\Ga(\ep)$-recursive by Fact~\ref{F:partialrecursivefunctions}(c),(e)
since the mappings $c_{\cX}$ and $c_{\cY}$ are $\Ga(\ep)$-recursive (see Definition~\ref{D:Ga-in-extended-spaces}(c)).
Using Fact~\ref{F:partialrecursivefunctions}(\ref{I:substitution-property}), we get that $\De^{f}$ is in $\Ga(\ep)$.
To verify that $\De^f$ computes $f$, we take $\xi \in D(f)$ and $t \in \cY^0$. Then we have
\[
\begin{split}
f(\xi) \in \cY(t)
&\Leftrightarrow c_{\cY} \circ f(\xi) \in c_{\cY}(\cY(t)) \\
&\Leftrightarrow f^c(c_{\cX}(\xi)) \in \cY^c (c_{\cY}(t)) \\
&\Leftrightarrow (c_{\cX}(\xi),c_{\cY}(t)) \in \De^{f^c} \\
&\Leftrightarrow (\xi,t) \in \De^f.
\end{split}
\]
This shows that $\De^f$ computes $f$.

For the opposite implication $\Leftarrow$, we set $\De^{f^c} = (c_{\cX} \times c_{\cY})(\De^{f})$.
The mappings $c_{\cX}^{-1}$ and $c_{\cY}^{-1}$ are $\Ga(\ep)$-recursive.
Thus the mapping $(c_{\cX} \times c_{\cY})^{-1}$ is $\Ga(\ep)$-recursive and we can proceed in a similar way as before.

For $\cY=\cN$ we have $\cY^c = \cN$, so $\cY^0 = (\cY^c)^0 = \cS$ and
the mapping $c_{\cX} \times \id_{\cS}$ maps $\De^f$ computing $f$ to a set $\De^{f^c}$ computing $f^c$ and conversely.
The equivalence of $\Ga(\ep)$-recursivity of these sets can be verified using the arguments of the preceding paragraph again.

For $\cY=\cP_\infty$ we have $\cY^0 = \cP$, $\cY^c = \cN$, and  $(\cY^c)^0 = \cS$.
We use the mapping $c_{\cX} \times c_{\cP}^{\cS}$ and its inverse to verify the last equivalence.

For extended product spaces the equivalence can be deduced from those concerning the extended basic spaces used coordinate-wise.
\end{proof}

\begin{notation**}\label{N:bigvee}
For every sequence $(z_k)_{k \in \om}$ of elements of $\cN$, we define
the element $\bigvee_{k\in\om} z_k$ from $\cN$ by $(\bigvee_{k\in\om} z_k)(n) = z_{[n]_0}([n]_1)$, $n \in \omega$.
\end{notation**}

In the next lemma we are avoiding the need of considering $\cN^\om$ as another extended product space explicitly.

\begin{lemma**}\label{L:bigvee}
If $r\colon \cX\times\om\rightharpoonup\cN$ is $\Ga(\ep)$-recursive on its domain,
then $\bigvee_{k\in\om} r(z,k)\colon \cX \rightharpoonup \cN$ is $\Ga(\ep)$-recursive on its domain.
\end{lemma**}

\begin{proof}
The corresponding partial function $r^*\colon \cX \times \om^2 \rightharpoonup \omega$ defined by $r^*(z,k,n) = r(z,k)(n)$ is $\Ga(\ep)$-recursive on its domain
by Fact~\ref{F:partialrecursivefunctions}\eqref{I:recursivity-to-N} and $\bigl(\bigvee_{k\in\om} r(z,k)\bigr)(n) = r^*(z,[n]_0,[n]_1)$.
Since $r^*$ and $[\cdot]$ are $\Ga(\ep)$-recursive on their domains, the lemma is verified.
\end{proof}

\begin{remark**}\label{R:upper-bounds}
To apply Lemma~\ref{L:bounded-quantifiers} in the extended spaces, we several times need to use some properties of the orderings of $\cP$ and $\cS^*$ defined in Notation~\ref{N:usporadanicP}.
Notice first of all that the sets $\{p\set p<_\cP q\}$ for $q\in\cP$ and $\{s\set s<_{\cS^*} t\}$ for $t\in\cS^*$ are finite.

In the proof of Lemma~\ref{L:recursivity-of-f} we need several times the properties of the following upper bounds with respect to $\le_\cP$ for particular subsets of $\cP$.
Let $s\in\cS$ be nonempty. We first put $m(s)=\max\{|s_j|\set j<|s|\}\in\om$. If $s\in\{-,\emptyset\}$, we put $m(s)=0$.
Then we define $M_\cS(s)=(\cl{m(s)}\restriction |s|)\in\cS$, i.e., the constant sequence of length $|s|$ which attains the value $m(s)$.
Finally, we define $M_\cP(s)=(\cl{M_\cS(s)}\restriction 2|s|+1)\in\cP$, i.e., the constant sequence of length $2|s|+1$ which attains the value $M_\cS(s)$.

Now, we make the following observation. Let $q\in\cP$ be such that $\wh{q}\preceq s$ and $|q|\le 2|\wh{q}|+1$. Then $q\le_\cP M_\cP(s)$.
Notice first that we have $|q|\le 2|s|+1$ and $|q_i|\le |s|$ for $i<2|s|+1$.
Since $m(q_i)\le m(s)$ and $|q_i|\le |s|$, we get $q_i\le_* M_\cS(s)$ for every $i<2|s|+1$.
By Notation~\ref{N:usporadanicP}(1), we get $q_i\le_\cS^* M_\cS(s)$ for every $i<2|s|+1$.
It follows that $c_\cP^\cS(q)\le_* c_\cP^\cS(M_\cP(s))$.
Using Notation~\ref{N:usporadanicP}(2), this is equivalent to $q\le_\cP M_\cP(s)$.

Using Lemma~\ref{L:volba-epsilon}, in particular the second line of (e), the mapping $s \mapsto M_{\cS}(s)$ is a $\Si^0_1(\ep)$-recursive function $M_{\cS}$ on $\cS^*$ and
the function $M_\cP$ defined by $s\mapsto M_\cP(s)$ is also $\Si^0_1(\ep)$-recursive.
\end{remark**}

\subsection{The classes \texorpdfstring{$\Pi_1^0$}{} and \texorpdfstring{$\Sigma_2^0$}{}}\label{SS:higher-classes}

Later on we use also the class $\Pi_1^0$ of complements of $\Sigma_1^0$ sets in extended product spaces and
the class $\Sigma_2^0$ of the sets $A$ of the form $A = \{x \in \cX \set \exists n \in \omega\colon (x,n) \in B\}$,
where $\cX$ is an extended product space and $B \in \Pi_1^0 \restriction \cX \times \omega$.
Let $a \in \cN$. Then the symbols $\Pi_1^0(\ep)$, $\Pi_1^0(a,\ep)$, $\Si_2^0(\ep)$, $\Si_2^0(a,\ep)$ stand for the corresponding relativized classes.

The results on the classes $\Si^0_1$, $\Si^0_1(a)$, $\Si^0_1(a,b)$ can be used for higher classes appropriately,
e.g., if $F$ is $\Pi^0_1(a,b)$ in $\cX\times \cZ$, where $\cZ$ is an extended product space of type $0$,
the set
\[
E=\{x\in\cX\set \forall z\in \cZ\colon (x,z)\in F\}
\]
is in $\Pi^0_1(a,b)$.

\section{Turing jumps}\label{S:Turing-jumps}

Let us recall that $\ep\in\cN$ was fixed in Notation~\ref{N:epsilon}.
We adjust the definition of ``Turing jump'' to the setting we are going to use by the following definition of $u'_v$, which can be read as the Turing jump of
$u$ relative to $v$ and our fixed $\ep$.

\begin{definition}\label{D:Turing-jump}
Let $u,v \in \cN$. Then the \textit{Turing jump} $u'_v$ is the characteristic function of the set
\[
\bigl\{n \in \omega \set (n,n,u,v,\ep) \in G^{\om\times\cN^3}_{\Sigma_1^0}\bigr\}.
\]
The \textit{higher Turing jumps} $u^{(\be)}_v, \beta \le \alpha$, are defined inductively by $u^{(0)}_{v} = u$,
$u^{(\be)}_{v} = (u_{v}^{(\gamma)})'_{v}$ for $\be = \gamma+1\le\al$, and $u^{(\be)}_{v} = \bigvee_{k\in\om} u^{(a(\be,k))}_v$ for $\be \in \LIM(\al)$.
(See Lemma~\ref{L:diamond} and Notations~\ref{N:diamant-el-a}, \ref{N:bigvee} for the definition of $a(\beta,k)$.)
We write $u^{(\be)}$ instead of $u^{(\be)}_{v}$ if $v$ is the identically zero sequence $\cl{0}$ in $\cN$.
\end{definition}

We shall need some auxiliary statements on Turing jumps.

\begin{lemma}[relations between Turing jumps]\label{L:jumps}\
\begin{enumerate}[\upshape (1)]
\item\label{I:a-from-a'} There is a partial function $r_{\ref{I:a-from-a'}}\colon\cN\rightharpoonup\cN$ which is $\Si^0_1(\ep)$-recursive on its domain such that
    $u=r_{\ref{I:a-from-a'}}(u'_{v})$ for every $u,v \in \cN$ (cf. \cite[II.1(3)]{Sacks}).

\item\label{I:z(be,a)'-from-(be,a')}
Let $W$ be an extended product space of type $0$ and $q\colon W \times \cN \rightharpoonup \cN$ be a partial function which is $\Si^0_1(\ep)$-recursive on its domain.
Then there is a $\Si^0_1(\ep)$-recursive function $r_{q}\colon W \times \cN \to \cN$ such that $q(w,u)'_{v} = r_{q}(w,u'_{v})$ whenever $q(w,u)$ is defined and $v \in \cN$ (cf. \cite[II.1(4)]{Sacks}).

\item\label{I:a^ga-from-a^be}  There is a partial function $r_{\ref{I:a^ga-from-a^be}}\colon[0,\al]^2\times\cN\rightharpoonup\cN$ which is $\Si^0_1(\ep)$-recursive on its domain such that $u^{(\ga)}_{v}=r_{\ref{I:a^ga-from-a^be}}(\ga,\be,u^{(\be)}_{v})$, whenever $\ga\le\be \leq \alpha$ and $u,v\in\cN$.
\end{enumerate}
\end{lemma}

\begin{proof}
\eqref{I:a-from-a'}
Denote $S := S^{\om^2,\om\times\cN^3}_{\Si_1^0}\colon \om^3\to\om$ (see Fact~\ref{F:recursionthms}(a)).
The set $A = \{(n,l,u) \in \om^2\times\cN\set u(n)=l\}$ is in $\Si^0_1(\ep) \restriction \om^2\times\cN$
by Lemma~\ref{L:volba-epsilon}(i).
With the help  of Lemma~\ref{L:volba-epsilon}(b) we can deduce that the set
$B = \{(n,l,k,u,w) \in \om^3 \times \cN^2 \set (n,l,u)\in A\}$ is a $\Si^0_1(\ep)$ subset of $\om^3\times \cN^2$.
By Definition~\ref{D:relativizations}, there is a $\Si^0_1$ set $C\su\om^3\times\cN^3$ such that $B=C^{\ep}$.
Using Fact~\ref{F:recursionthms}(a), we find $e^* \in \om$ such that $C=(G^{\om^3\times{\cN^3}}_{\Si_1^0})_{e^*}$ and so
$B = C^{\ep}=(G^{\om^3\times{\cN^3}}_{\Si_1^0})_{e^*}^{\ep}$.

For every $n,l\in\om$ and $u,v \in\cN$ we have
\begin{equation}\label{E:property-i}
\begin{split}
u(n)=l &\Leftrightarrow (n,l,u)\in A \ \Leftrightarrow \ (n,l,S(e^*,n,l),u,v)\in B  \\
&\Leftrightarrow (e^*,n,l,S(e^*,n,l),u,v,\ep)\in G^{\om^3\times{\cN^3}}_{\Si_1^0} \\
&\Leftrightarrow (S(e^*,n,l),S(e^*,n,l),u,v,\ep)\in G^{\om\times\cN^3}_{\Si_1^0} \\
&\Leftrightarrow u'_v\bigl(S(e^*,n,l)\bigr) = 1.
\end{split}
\end{equation}
Let us make comments on equivalences in \eqref{E:property-i}.
First we used the definition of $A$. The second equivalence uses that the fact $(n,l,k,u,w) \in B$
does not depend on $k,w$. Thus we can make a particular choice of these coordinates.
The next equivalence follows from the choice of $e^*$. Then we use a property of good parametrizations (Fact~\ref{F:recursionthms}(a) and Remark~\ref{R:extended-facts}).
Finally we use the definition of Turing jump. Let us note that this standard trick of adding some coordinates will be used again couple of times.

The function $g$ from $\om^2 \times \cN$ to $\om$ defined by $g(n,l,z) = z\bigl(S(e^*,n,l)\bigr)$ is $\Si^0_1(\ep)$-recursive since
it is a composition of $\Si^0_1(\ep)$-recursive functions $(n,l,z) \in \om^2 \times \cN \mapsto (S(e^*,n,l),z) \in \om\times\cN$ and $(j,z)\in\om\times\cN\mapsto z(j)\in\om$ (see  Fact~\ref{F:partialrecursivefunctions}(\ref{I:composition})).
So
\[
r_{\ref{I:a-from-a'}}(z)(n) = \min\{l\in\om\set z(S(e^*,n,l))=1\}
\]
defines a partial function which is $\Si^0_1(\ep)$-recursive on
\[
D(r_1) = \{z\in\cN\set \forall n\in\om \ \exists l \in \om \colon z\bigl(S(e^*,n,l)\bigr)=1\}
\]
by Fact~\ref{F:partialrecursivefunctions}(\ref{I:min}). Using  \eqref{E:property-i}, we have $r_1(u'_v) = u$ for every $u,v \in \cN$.

\bigskip\noindent
(\ref{I:z(be,a)'-from-(be,a')}) Since $q$ is a partial function which is $\Sigma_1^0(\varepsilon)$-recursive on its domain,
there exists $Q \in \Sigma_1^0(\varepsilon)\restriction W \times \cN \times \cS$  such that for every $(w,u) \in D(q)$ and $s \in \cS$ we have
\[
q(w,u) \in \cN(s) \Leftrightarrow (w,u,s)\in Q,
\]
i.e., $Q$ computes $q$.
Let $G = (G^{\om\times\cN^3}_{\Si_1^0})^{\ep}$. Using Lemma~\ref{L:natural-base}(a), we find $G^*\in\Sigma_1^0(\varepsilon)\restriction \om^2 \times \cS^2$ such that
\[
(n_1,n_2,u,v) \in G  \Leftrightarrow  \exists s,t\in\cS \colon (n_1,n_2,s,t)\in G^* \land s \preceq u \land t \preceq v.
\]
Further we set
\[
H = \{(n,w,k,u,v) \in \om \times W \times \om \times \cN^2 \set \exists s,t\in\cS \colon (n,n,s,t) \in G^* \land (w,u,s) \in Q \land t \preceq v\}.
\]
Using in particular Lemma~\ref{L:volba-epsilon}(d), we get that the set $H$ is in $\Sigma_1^0(\varepsilon)$. By Definition~\ref{D:relativizations} and Remark~\ref{R:extended-facts}, there is a $\Sigma_1^0$ set $\Xi\su \om \times W \times \om \times \cN^3$ such that $H=\Xi^{\ep}$. Due to Fact~\ref{F:recursionthms}(a) and Remark~\ref{R:extended-facts}, there exists $e \in \omega$ such that
$\Xi=(G^{\om \times W \times \omega \times \cN^3}_{\Sigma_1^0})_e$.
Thus we have $H=\Xi^{\ep}=(G^{\om \times W \times \omega \times \cN^3}_{\Sigma_1^0})_e^{\ep}$.
Let $S := S^{\omega \times W,\omega \times \cN^3}_{\Sigma_1^0}$ be the $\Sigma_1^0(\varepsilon)$-recursive.
Then for $w\in W$, $u,v \in \cN$, $n \in \om$ with $(w,u) \in D(q)$ we have
\begin{equation}\label{E:property-ii}
\begin{split}
q(w,u)'_v(n)=1
&\Leftrightarrow (n,n,q(w,u),v) \in G \\
&\Leftrightarrow \exists s,t \in \cS\colon (n,n,s,t)\in G^* \land s \preceq q(w,u) \land t \preceq v\\
&\Leftrightarrow \exists s, t \in \cS \colon (n,n,s,t) \in G^* \land (w,u,s) \in Q \land t \preceq v\\
&\Leftrightarrow (n,w,S(e,n,w),u,v) \in H \\
&\Leftrightarrow (e,n,w,S(e,n,w),u,v,\ep)\in G^{\om \times W \times \omega \times \cN^3}_{\Sigma_1^0}\\
&\Leftrightarrow \bigl(S(e,n,w),S(e,n,w),u,v,\ep\bigr) \in G^{\omega \times \cN^3}_{\Sigma_1^0} \ \Leftrightarrow \ u'_v\bigl(S(e,n,w)\bigr) = 1.
\end{split}
\end{equation}
Here we used successively the definition of Turing jump, the choice of $G^*$, the choice of $Q$,
the definition of $H$ and the observation that $(n,w,k,u,v)\in H$ does not depend on the value of the third coordinate,
the choice of $e$, Fact~\ref{F:recursionthms}(a) together with Remark~\ref{R:extended-facts}, and finally the definition of Turing jump again.

For $n \in \om$, $w \in W$, and $z \in \cN$ we define
\[
r_q(w,z)(n) =
\begin{cases}
1  &\text{if } z(S(e,n,w)) = 1, \\
0  &\text{otherwise}.
\end{cases}
\]
By similar argument as in the proof of the assertion (1) the function $g$ from $\omega \times W \times \cN$ to $\om$ defined by $g(n,w,z) = z\bigl(S(e,n,w)\bigr)$ is a $\Sigma_1^0(\varepsilon)$-recursive function defined for every $n \in \omega, w \in W, z \in \cN$.
So $r_q$ is a $\Sigma_1^0(\varepsilon)$-recursive function.
Using \eqref{E:property-ii}, we have $q(w,u)_v' = r_q(w,u_v')$ whenever $q(w,u)$ is defined and $v\in\cN$.

\bigskip\noindent
(\ref{I:a^ga-from-a^be}) Using the function $r_1$ from (1), the function $U := U^{[0,\al]^2\times\cN,\cN}_{\Si_1^0(\ep)}$, and the mappings
$\diamond$ and $a$ from Lemma~\ref{L:diamond}, we define a partial function $f\colon \omega \times [0,\alpha]^2 \times \cN \rightharpoonup \cN$ by
\[
f(e,\ga,\be,z)(m) =
\begin{cases}
z(m)                                                 &\text{if } \ga = \beta, \\
U(e,\diamond(\ga),\be,z)([l,m]^{-1})                 &\text{if } \diamond(\ga) \leq \be \text{ and } \ga = a(\diamond(\ga),l), \\
r_1\bigl(U(e,\ga+1,\be,z)\bigr)(m)                   &\text{if } \ga < \be < \diamond(\ga).
\end{cases}
\]
Using previously stated facts, in particular Lemma~\ref{L:volba-epsilon}, one can verify that the function $f$ is $\Si^0_1(\ep)$-recursive on its domain. By Fact~\ref{F:recursionthms}(b) and Remark~\ref{R:extended-facts}
there is $e^*\in\om$ such that $f(e^*,\ga,\be,z) = U(e^*,\ga,\be,z)$ whenever $f(e^*,\ga,\be,z)$ is defined.
We set $r_3(\ga,\be,z) = f(e^*,\ga,\be,z)$. The partial function $r_3$ is $\Sigma_1^0(\ep)$-recursive on its domain.

Let $0 \leq \ga \le \be \le \alpha$ and $z \in \cN$. Using Lemma~\ref{L:sequence}, we find a finite increasing sequence $\{\delta_j\}_{j=0}^n$ such that $\delta_0 = \gamma$, $\delta_n = \be$,  $\delta_{j+1} = \diamond(\delta_j)$ whenever $\diamond(\delta_j) \leq \beta$ and $\delta_{j+1} = \delta_j + 1$ otherwise.
By definition we see that $(\beta,\beta,z) \in D(r_3)$ for every $z \in \cN$. If $(\delta_{j+1},\beta,z) \in D(r_3)$ then by definition we get that also
$(\delta_j,\beta,z) \in D(r_3)$. Consequently, $(\ga,\beta,z) \in D(r_3)$ for every $z \in \cN$.

Now we check that $r_3(\delta_j,\be,u_v^{(\beta)}) = u_v^{(\de_j)}$, $j = 0,\dots, n$. For $j = n$ we have
\[
r_3(\delta_n,\be,u_v^{(\beta)}) = r_3(\be,\be,u_v^{(\beta)}) = u_v^{(\be)} = u_v^{(\de_n)}.
\]
Suppose that we already have proved the desired equality for $j+1 \leq n$. If $\delta_{j+1} = \diamond(\de_j)$, then there exists a unique $l \in \omega$ with
$a(\diamond(\de_j),l) = \de_j$. Then we have for every $m \in \omega$
\[
r_3(\delta_j,\be,u_v^{(\beta)})(m) = r_3(\delta_{j+1},\be,u^{(\be)}_v)([l,m]^{-1}) = u_v^{(\delta_{j+1})}([l,m]^{-1}) = u_v^{(\de_j)}(m).
\]
If $\delta_{j+1} = \de_j+1$ we have
\[
r_3(\delta_j,\be,u_v^{(\beta)}) = r_1\bigl(r_3(\delta_{j+1},\be,u^{(\be)}_v)\bigr) = r_1\bigl(u_v^{(\delta_{j+1})}\bigr)  = r_1\bigl((u_v^{(\delta_{j})})'_v\bigr) = u_v^{(\de_j)}.
\]
Thus for $j=0$ we get $r_3(\ga,\be,u_v^{(\beta)}) = u_v^{(\gamma)}$.
\end{proof}

The next lemma shows relationships between jumps of the form $x^{(\beta)}_x$ and $x^{(\beta)}$ which simplifies some formulation later on.

\begin{lemma}\label{L:jump-in-x-and-in-0}
There exists a partial  function $r\colon [0,\alpha] \times \cN \rightharpoonup \cN$ which is $\Sigma_1^0(\ep)$-recursive on its domain
such that $x_x^{(\beta)} = r(\beta,x^{(\beta)})$ whenever $x \in \cN$ and $\beta \in [0,\alpha]$.

Conversely, there is a partial function $\wt{r}\colon [0,\alpha] \times \cN \rightharpoonup \cN$
which is  $\Sigma_1^0(\ep)$-recursive on its domain
such that $x^{(\beta)} = \wt{r}(\beta,x_x^{(\beta)})$ whenever $x \in \cN$ and $\beta \in [0,\alpha]$.
\end{lemma}

\begin{proof}
First we prove that there is a $\Sigma_1^0(\varepsilon)$-recursive function $r_+\colon [0,\al)\times\cN \to \cN$ such that $x^{(\ga+1)}_x = r_+(\ga,(x^{(\ga)}_x)' )$ whenever $x \in \cN$ and $\ga<\al$.
We define
\[
W = \bigl\{(n,\ga,l,y,w) \in \omega\times [0,\al] \times\om \times \cN^2 \set
(n,n,y,r_3(0,\ga,y),\ep) \in G^{\omega \times \cN^3}_{\Si_1^0}\bigr\}.
\]
The partial function $r_3$ comes from Lemma~\ref{L:jumps}(3). Thus $r_3$ is $\Sigma_1^0(\ep)$-recursive on its domain and, in particular, we have $x = r_3(0,\ga,x^{(\ga)}_x)$
for every $x \in \cN$ and $\gamma < \alpha$.
Notice that the fact $(n,\ga,l,y,w) \in W$ does not depend on $l$ and $w$.
The set $W$ is $\Sigma_1^0(\varepsilon)$ in the extended product space $\omega\times [0,\al]\times\om \times \cN^2$.
By Definition~\ref{D:relativizations} and Remark~\ref{R:extended-facts}, there is a $\Sigma_1^0$ set $V\su \omega\times [0,\al] \times\om \times \cN^3$ such that $W=V^{\ep}$.
By Fact~\ref{F:recursionthms}(a) and Remark~\ref{R:extended-facts}, there exists $e_W \in \omega$ such that
$V=(G^{\omega\times [0,\al]\times \om\times \cN^3}_{\Si_1^0})_{e_W}$. Thus we have that
$W = (G^{\omega\times [0,\al]\times \om\times \cN^3}_{\Si_1^0})_{e_W}^{\ep}$.
We set $S = S^{\omega\times [0,\al],\omega \times \cN^3}_{\Si_1^0}$.
Then we have for $x \in \cN$, $\ga<\al$, and $n \in \omega$
\[
\begin{split}
x_x^{(\ga+1)}(n) = 1
&\Leftrightarrow (n,n,x^{(\ga)}_x,r_3(0,\ga,x^{(\ga)}_x),\ep) \in  G^{\omega \times \cN^3}_{\Si_1^0}
\Leftrightarrow (n,\ga,S(e_W,n,\ga),x^{(\ga)}_x,\overline{0}) \in W \\
&\Leftrightarrow (e_W,n,\ga,S(e_W,n,\ga),x^{(\ga)}_x,\overline{0},\ep) \in G^{\omega\times [0,\al]\times\om \times \cN^3}_{\Si_1^0} \\
&\Leftrightarrow (S(e_W,n,\ga),S(e_W,n,\ga),x^{(\ga)}_x,\overline{0},\ep) \in G^{\omega \times \cN^3}_{\Si_1^0} \\
&\Leftrightarrow (x^{(\ga)}_x)'(S(e_W,n,\ga)) = 1.
\end{split}
\]
We define the desired auxiliary function $r_+$ by $r_+(\ga,z)(n) =  z(S(e_W,n,\ga))$ whenever $z \in \cN$, $\ga\in [0,\al)$, and $n \in \omega$.

We denote $U = U^{[0,\alpha]\times\cN,\cN}_{\Si_1^0(\varepsilon)}$ and we use the mapping $r_U$ from Lemma~\ref{L:jumps}(2) with the type 0 extended product space
$\om\times [0,\al]$ and $q = U$.
Now we define a partial function $f\colon \omega \times [0,\alpha] \times \cN \rightharpoonup \cN$ by
\[
f(e,\beta,z) =
\begin{cases}
z                                                                     &\text{if } \beta = 0, \\
r_+\bigl(\ga,r_U(e,\gamma,z)\bigr)                                    &\text{if } \beta = \gamma + 1, \\
\bigvee_{k \in \omega}U\bigl(e,a(\beta,k),r_3(a(\be,k),\be,z)\bigr)   &\text{if } \beta \text{ is limit}.
\end{cases}
\]
Using previously mentioned facts, in particular Lemma~\ref{L:bigvee}, we get that the function $f$ is $\Sigma_1^0(\ep)$-recursive on its domain.
Thus by Fact~\ref{F:recursionthms}(b) and Remark~\ref{R:extended-facts} there exists $e^* \in \omega$ such that $f(e^*,\beta,z) = U(e^*,\beta,z)$ whenever
$f(e^*,\beta,z)$ is defined. We set $r = f_{e^*}$. The function $r$ is $\Sigma_1^0(\ep)$-recursive on its domain
and we verify the equality $x_x^{(\beta)} = r(\beta,x^{(\beta)})$ for every $x\in\cN$ and $\beta \leq \alpha$
by transfinite induction on $\be \leq \al$. Let $x \in \cN$.
If $\beta = 0$, then  $x_x^{(0)} = x = f(e^*,0,x) = r(0,x^{(0)})$. If $\beta = \gamma + 1 < \alpha$, then
$x^{(\gamma)}_x = r(\gamma,x^{(\gamma)})$ by induction hypothesis and we can write
\[
\begin{split}
x_x^{(\beta)}
&= x_x^{(\gamma +1)} = r_+(\ga,(x_x^{(\gamma)})') = r_+(\ga,r(\gamma,x^{(\gamma)})')
= r_+(\ga,f(e^*,\gamma,x^{(\gamma)})') \\
&= r_+(\ga,U(e^*,\gamma,x^{(\gamma)})')
= r_+(\ga,r_U(e^*,\gamma,x^{(\gamma+1)}))
= f(e^*,\beta,x^{(\beta)}) = r(\beta,x^{(\beta)}).
\end{split}
\]
Let us point out that we apply $r_U$ to express Turing jumps with respect to $v=\cl{0}$.

If $\beta \leq \alpha$ is a limit ordinal, then we have
\[
\begin{split}
x_x^{(\beta)} &= \bigvee_{k \in \omega} x_x^{(a(\beta,k))}
= \bigvee_{k \in \omega} r(a(\beta,k),x^{(a(\beta,k))})
= \bigvee_{k \in \omega} U(e^*,a(\beta,k),x^{(a(\beta,k))})\\
&= \bigvee_{k \in \omega} U(e^*,a(\beta,k),r_3(a(\beta,k),\beta,x^{(\beta)}))
= f(e^*,\beta,x^{(\beta)}) = r(\beta,x^{(\beta)}).
\end{split}
\]
This concludes the proof of the first statement.

\medskip\noindent
As for the second statement, we first prove that there is a $\Si^0_1(\ep)$-recursive function $w_+\colon \cN \to \cN$ such that $x^{(\ga+1)} = w_+((x^{(\ga)})'_x )$ whenever $x \in \cN$ and $\ga \in [0,\alpha)$.
We define
\[
O = \bigl\{(n,l,y,w,z) \in \omega^2 \times \cN^3 \set (n,n,y,\overline{0},z) \in G^{\omega \times \cN^3}\bigr\}.
\]
The set $O$ is in $\Sigma_1^0$, thus there exists $e_O \in \omega$ such that
$O = (G^{\omega^2 \times \cN^3}_{\Si_1^0})_{e_O}$. We set $S = S^{\omega,\omega \times \cN^3}_{\Si_1^0}$.
Then we have for $x \in \cN$, $\ga<\al$, and $n \in \omega$
\[
\begin{split}
x^{(\ga+1)}(n) = 1
&\Leftrightarrow (n,n,x^{(\ga)},\overline{0},\ep) \in  G^{\omega \times \cN^3}_{\Si_1^0}
\Leftrightarrow (n,S(e_O,n),x^{(\ga)},x,\ep) \in O \\
&\Leftrightarrow (e_O,n,S(e_O,n),x^{(\ga)},x,\ep) \in G^{\omega^2 \times \cN^3}_{\Si_1^0}\\
&\Leftrightarrow (S(e_O,n),S(e_O,n),x^{(\ga)},x,\varepsilon) \in G^{\omega \times \cN^3}_{\Si_1^0}
\Leftrightarrow (x^{(\ga)})'_x(S(e_O,n)) = 1.
\end{split}
\]
We define the desired auxiliary function $w_+$ by $w_+(z)(n) =  z(S(e_O,n))$ whenever $z \in \cN$ and $n \in \omega$.

Let us point out that $U$ and $r_U$ have the same meaning as in the proof of the first part of the statement.
Now we define a partial function $\wt{f}\colon \omega \times [0,\alpha] \times \cN \rightharpoonup \cN$ by
\[
\wt{f}(e,\beta,z) =
\begin{cases}
z                                                            &\text{if } \beta = 0, \\
w_+(r_U(e,\gamma,z))                                         &\text{if } \beta = \gamma + 1, \\
\bigvee_{k \in \omega}U(e,a(\beta,k),r_3(a(\be,k),\be,z))    &\text{if } \beta \text{ is limit}.
\end{cases}
\]
Using Lemma~\ref{L:bigvee}, we get that the function $\wt{f}$ is $\Sigma_1^0(\ep)$-recursive on its domain.
Thus by Fact~\ref{F:recursionthms}(b) and Remark~\ref{R:extended-facts} there exists $\wt{e} \in \omega$ such that $\wt{f}(\wt{e},\beta,z) = U(\wt{e},\beta,z)$ whenever
$\wt{f}(\wt{e},\beta,z)$ is defined. We set $\wt r = \wt{f}_{\wt{e}}$. The function $\wt r$ is $\Sigma_1^0(\ep)$-recursive on its domain and
we verify the equality $x^{(\beta)} = \wt r(\beta,x_x^{(\beta)})$ for every $x \in \cN$ by transfinite induction on $\be \leq \al$.

Let $x \in \cN$.
If $\beta = 0$, then  $x^{(0)} = x = \wt{f}(\wt{e},0,x) = \wt r(0,x_x^{(0)})$. If $\beta = \gamma + 1 < \alpha$, then
$x^{(\gamma)} = \wt r(\gamma,x_x^{(\gamma)})$ by induction hypothesis and we can write
\[
\begin{split}
x^{(\beta)}
&= x^{(\gamma +1)}
= w_+((x^{(\gamma)})'_x) = w_+(\wt r(\gamma,x^{(\gamma)}_x)'_x)
= w_+(\wt{f}(\wt{e},\gamma,x^{(\gamma)}_x)'_x) \\
&= w_+(U(\wt{e},\gamma,x^{(\gamma)}_x)'_x)
= w_+(r_U(\wt{e},\gamma,x^{(\gamma+1)}_x))
= \wt{f}(\wt{e},\beta,x^{(\beta)}_x) = \wt r(\beta,x^{(\beta)}_x).
\end{split}
\]
Notice that we applied Lemma~\ref{L:jumps}(2) for $v=x$ here.

If $\beta \leq \alpha$ is a limit ordinal, then we have
\[
\begin{split}
x^{(\beta)} &= \bigvee_{k \in \omega} x^{(a(\beta,k))}
= \bigvee_{k \in \omega} \wt r \bigl(a(\beta,k),x^{(a(\beta,k))}_x\bigr)
= \bigvee_{k \in \omega} U\bigl(\wt{e},a(\beta,k),x^{(a(\beta,k))}_x\bigr)\\
&= \bigvee_{k \in \omega} U\bigl(\wt{e},a(\beta,k),r_3(a(\beta,k),\beta,x^{(\beta)}_x)\bigr)
= \wt{f}(\wt{e},\beta,x^{(\beta)}_x) = \wt r(\beta,x^{(\beta)}_x).
\end{split}
\]
This concludes the proof of the second statement.
\end{proof}

The relation $J$ for $(n,s,x,k)$ formalizes what could be understood by ``$s'_x(n)=1$ in at most $k+1$ steps'',
in other words by ``the Turing jump of $s$ relative to $x$ and $\ep$ is $1$ at $n$ in $k+1$ steps''.

\begin{definition}\label{D:J-jump}
Let $G = (G^{\om\times\cN^3}_{\Si_1^0})^\ep \su \om^2\times\cN^2$, $G^* \subset \omega^2 \times \Seq^2$ be from Lemma~\ref{L:natural-base}(a), and $R = R_{G^*} \subset \omega^2 \times \Seq^2 \times \omega$ be the corresponding
$\Si^0_1(\ep)$-recursive set which exists by Fact~\ref{F:Gstar}(a) and Remark~\ref{R:extended-facts}.
The set $R$ is $\Si_1^0(\ep)$-recursive.
We use the notation
\[
\begin{split}
J = \{(n,s,x,k)\in\om &\times (\cS\cup\cN) \times\cN\times \om \set \\
&\exists j_1 \le  \min\{k,|s|\} \ \exists j_2 \le k \ \exists j_3 \le k\colon (n,n,s|j_1,x|j_2,j_3) \in R\}.
\end{split}
\]
We also define the sets
\begin{align*}
J_0      &= \{(n,s,k)\in \omega \times \cS \times \omega \set (n,s,\cl{0},k)\in J\}, \\
J_\infty &= \{(n,s,k)\in\om\times\cN\times\om\set (n,s,\cl{0},k)\in J\}, \\
J_*      &= \{(n,s,x)\in \omega \times \cS \times \cN \set (n,s,x,|s|)\in J\}.
\end{align*}
\end{definition}

\begin{remark}\label{R:property-of-J}
Let $x, z \in \cN$, $n, k \in \om$, and $s, t \in \cS \cup \cN$.
Then we have
\begin{itemize}
\item[(a)] $(n,s,x,k) \in J$ if and only if $(n,t,x,k) \in J$ and $s|k\preceq t$,

\item[(b)] if $(n,s,x,k) \in J$ and $l \ge k$, then $(n,s,x,l) \in J$,

\item[(c)] if $(n,z,x,k) \in J$, then $z'_x(n)=1$,

\item[(d)] if $z'_x(n) = 1$, then there is $k \in \om$ such that $(n,z,x,k) \in J$.
\end{itemize}

We refer to the above observations also if applying their obvious consequences for the particular cases of $J$ which we defined above.
We point out two more observations by the following lemma, which we use in the proof of Lemma~\ref{L:recursivity-of-Theta}.
\end{remark}

\begin{lemma}\label{L:J-star-jump}
Let $\si\in\cN$, $x\in\cN$, and $n\in\om$.
\begin{itemize}
  \item[\upshape (a)] If $\si'_x(n)=1$, then we have
  \[
  \exists L\in\om\ \forall L'\ge L\colon (n,\si|L',x)\in J_*.
  \]

  \item[\upshape (b)] If there exists $s \in \cS$ such that $s \preceq \si$ and $(n,s,x) \in J_*$, then $\si'_x(n)=1$.
\end{itemize}
\end{lemma}

\begin{proof}
(a) Let $G$, $G^*$, and $R$ be as in Definition~\ref{D:J-jump}, which gives under our assumption that $(n,n,\si,x)\in G$.
It follows that there are $j_1,j_2,j_3\in\om$ such that $(n,n,\si|j_1,x|j_2)\in G^*$ and $(n,n,\si|j_1,x|j_2,j_3)\in R$.
Choosing $L=\max\{j_1,j_2,j_3\}$ we get that $(n,\si|L',x)\in J_*$ for $L'\ge L$ by the definition of $J_*$.

\medskip\noindent
(b) The assumption gives the existence of $j_1,j_2,j_3\le |s|$ such that $(n,n,s|j_1,x|j_2,j_3)\in R$. Then $(n,n,\si,x)\in G$ and so $\si'_x(n)=1$ by Definition~\ref{D:Turing-jump}.
\end{proof}

\begin{lemma}\label{L:recursivity-of-Js}
The sets $J_0$, $J_\infty$, and $J_*$ are $\Sigma_1^0(\ep)$-recursive.
\end{lemma}

\begin{proof}
We prove the statement for $J_*$. Since $R$ is $\Sigma_1^0(\ep)$-recursive (see Definition~\ref{D:J-jump}),
using the $\Si^0_1(\ep)$-recursivity of restrictions from Lemma~\ref{L:volba-epsilon}(e),(i) and Fact~\ref{F:partialrecursivefunctions}\eqref{I:substitution-property} we get that the set
\[
\wt R = \{(n,s,x,j_1,j_2,j_3) \in \omega \times \Seq \times \cN \times\om^3\set (n,n,s|j_1,x|j_2,j_3) \in R\}
\]
and its complement are in $\Sigma_1^0(\ep)$.
Thus $\wt R$ is $\Sigma_1^0(\ep)$-recursive.
Using in particular $\Si^0_1(\ep)$-recursivity of $s \mapsto |s|$ on $\cS$, we get that $J_*$ is $\Sigma_1^0(\ep)$-recursive
due to Lemma~\ref{L:bounded-quantifiers} (used for the quantifications coming in play through the definition of $J$, which appears in the definition of $J_*$).

The proofs for  $J_0$ and $J_\infty$ are analogous. Applying Lemma~\ref{L:bounded-quantifiers}, we use $\Sigma_1^0(\ep)$-recursivity of the upper bounds
$k$ for $j_1$ in the case of $J_\infty$, $\min\{k,|s|\}$ for $j_1$ in the case of $J_0$, and the upper bound $k$ for $j_2$ and $j_3$ in both cases.
\end{proof}

We represent the Turing jumps $y^{(\be)}$ ($= y^{(\be)}_{\cl{0}}$), $y \in \cN, \be \in [0,\al]$, as the infinite branches of a recursive family of trees as follows.
\begin{proposition}[$x^{(\be)}$'s as branches of trees]\label{P:jump-trees-old}
There exists a $\Si^0_1(\ep)$-recursive set $S \subset [0,\al] \times \cS \times{\cN}$ such that
\begin{itemize}
\item[(a)] $S_\be^{x} \ (= \{s \in \cS \set (\be,s,x) \in S\})$ is a tree for every $\beta \in [0,\alpha]$ and $x\in\cN$,

\item[(b)] for every $\beta \in [0,\alpha]$ and $x\in\cN$ there exists $\sigma_{\be}^{x} \in \cN$ such that $[S_\be^{x}] = \{\sigma_{\be}^{x}\}$, and

\item[(c)] there are $\Si^0_1(\ep)$-recursive on their domains partial functions $\rright \colon [0,\al]\times\cN\rightharpoonup\cN$ and $\rleft \colon [0,\al]\times\cN\rightharpoonup\cN$ such that $x^{(\be)} = \rright(\be,\sigma_{\be}^{x})$ and $\sigma_{\be}^{x} = \rleft(\be,x^{(\be)})$,
whenever $\beta \in [0,\alpha]$ and $x \in \cN$.
\end{itemize}
\end{proposition}

Before continuing with the proof of Proposition~\ref{P:jump-trees-old}, we informally explain its main idea.
First we define a set $Z \subset \cS$ such that $s$ is in $Z$ if and only if each coordinate $s_n$, $n < |s|$, is either too big in the sense that $[s_n]_1 > |s|$ or the number $[s_n]_1 \in \bigl[0,|s|\bigr]$  says whether $[s]_0$ ensures that $z'(n) = 1$ for any $z \in \cN$ with $[s]_0 \prec z$ or not. (We use Notation~\ref{N:identification} to interpret the coordinates of sequences from $\cS$ as pairs of elements of $\om$.)
In the former case $[s_n]_1$ is the smallest number such that $(n,[s]_0,[s_n]_1-1) \in J_0$, in the latter $[s_n]_1 = 0$.
Then we construct the desired set $S$ by constructing sections $S_{\beta}^x$ for $\beta \in [0,\alpha], \ x \in \cN$, using transfinite induction on $\beta$.

We start with $S_0^x = \{s\in\cS\set s\preceq x\}$. Having $S^x_{\gamma}$ for $\ga \in [0,\al)$, we consider elements $s$ of  $S^x_{\ga}$ as approximations of $x^{(\gamma)}$ or, because of technical reasons, of some $\sigma_{\ga}^x\in\cN$ Turing equivalent to $x^{(\ga)}$.
We put $s$ in $S_{\gamma + 1}^x$ if $[s]_0 \in S_{\ga}^x$ and $s \in Z$. Thus $s$ will be an approximation of the jump of $z \in \cN$ with $[s]_0 \prec z$. If $\beta \leq \alpha$ is a limit ordinal, then $s$ in $S_{\be}^x$ glues together approximations from
the trees $S_{a(\be,j)}^x, j \in \omega$. We use Kleene recursion theorem  to get $\Si_1^0(\ep)$-recursivity of $S$.

\begin{proof}[{\bf Proof of Proposition~\ref{P:jump-trees-old}}]

We introduce an auxiliary set $Z \subset \cS$ as follows.
A sequence $s \in \cS$ is in $Z$ if and only if for every $n < |s|$ we have
\begin{equation}\label{E:set-Z}
\begin{split}
&\bullet \ [s_n]_1 > |s| \text{ or } \\
&\bullet \ [s_n]_1=0 \land (n,[s]_0,|s|)\notin J_0 \text{ or } \\
&\bullet \ [s_n]_1 \in \bigl(0,|s|\bigr] \land (n,[s]_0,[s_n]_1-1)\in J_0 \land (n,[s]_0,[s_n]_1-2) \notin  J_0.
\end{split}
\end{equation}

\begin{claim}
The set $Z$ is a tree.
\end{claim}

\begin{proof}[Proof of Claim]
Suppose that we have sequences $s,t \in \cS$ with $t \preceq s$ and $s \in Z$. We show that $t \in Z$. Let $n < |t|$. If $[s_n]_1 > |s|$, then $[t_n]_1 = [s_n]_1 > |s| \geq |t|$.
If $[s_n]_1 = 0$ and $(n,[s]_0,|s|) \notin J_0$, then $[t_n]_1 = [s_n]_1 = 0$ and $(n,[t]_0,|t|) = (n,[s]_0||t|,|t|) \notin J_0$ by Remark~\ref{R:property-of-J}(a),(b).
Now suppose that $[s_n]_1 \in (0,|s|]$, $(n,[s]_0,[s_n]_1-1) \in J_0$, and $(n,[s]_0,[s_n]_1-2) \notin J_0$.
Then either $[t_n]_1 > |t|$ or $[t_n]_1 \leq |t|$. The first case verifies the first condition of the three for $t$ and $n$.
So assume that $[t_n]_1 \leq |t|$.
Since $[t_n]_1=[s_n]_1$, we have $(n,[s]_0,[t_n]_1-1) \in J_0$ and this implies $(n,[s]_0||t|,[t_n]_1-1) \in J_0$ by  Remark~\ref{R:property-of-J}(a). Consequently, $(n,[t]_0,[t_n]_1-1) \in J_0$.
Similarly, since $[t_n]_1-2 = [s_n]_1-2 \le |t|$,
$(n,[s]_0,[s_n]_1-2) \notin J_0$ implies $(n,[s]_0||t|,[s_n]_1-2) \notin J_0$ by Remark~\ref{R:property-of-J}(a).
Consequently, we have $(n,[t]_0,[t_n]_1-2) = (n,[s]_0||t|,[s_n]_1-2) \notin J_0$.
\end{proof}

Let us use the notation
\[
W = U_{\Si^0_1(\ep)}^{[0,\al]\times\cS\times\cN,\om}\colon \om\times [0,\al]\times\cS\times\cN\rightharpoonup \om,
\]
see Fact~\ref{F:recursionthms}(b) and Remark~\ref{R:extended-facts}.
Recall that if $s \in \Seq$ and $k \in \omega$, then $j \mapsto s([k,j]^{-1})$ defines a finite sequence
(see Notation~\ref{N:identification}.

We define a partial mapping $f\colon \omega \times [0,\al] \times \cS \times \cN \rightharpoonup \omega$ as follows.
We set $f(e,\beta,s,x) = 1$ whenever $(e,\beta,s,x)$ satisfies one of the following conditions:
\begin{itemize}
\item[(i)]   $\beta = 0 \land s \preceq x$,

\item[(ii)]  $\be = \ga + 1 \land W(e,\ga,[s]_0,x) = 1 \land  s \in Z$,

\item[(iii)] $\be\in \LIM(\al) \land  \forall n < |s|\colon W \bigl(e,a(\be,n),s([n,\cdot]^{-1}),x\bigr) = 1$.
\end{itemize}
We set $f(e,\beta,s,x) = 0$ whenever $(e,\beta,s,x)$ satisfies one of the following conditions:
\begin{itemize}
\item[(i')]   $\beta = 0 \land \neg(s \preceq x)$,

\item[(ii')]  $\be = \ga+1 \land \bigl((W(e,\ga,[s]_0,x) = 1 \land s \notin Z) \vee W(e,\ga,[s]_0,x) = 0 \bigr)$,

\item[(iii')] $\be \in \LIM(\al) \land \exists n < |s|\colon  W\bigl(e,a(\be,n),s([n,\cdot]^{-1}),x\bigr) = 0$.
\end{itemize}
The value of $f$ is not defined otherwise.

\begin{claim}
The function $f$ is $\Si^0_1(\ep)$-recursive on its domain.
\end{claim}

\begin{proof}[Proof of Claim]
Since $|\cdot|$, $[\, \cdot \,]_i$, the coordinate functions, and the function $a$ are $\Si^0_1(\ep)$-recursive functions by Lemma~\ref{L:volba-epsilon} and $J_0$ is $\Si^0_1(\ep)$-recursive by Lemma~\ref{L:recursivity-of-Js}, the set $Z$ is $\Si^0_1(\ep)$-recursive.

The function $W$ is $\Si^0_1(\ep)$-recursive on its domain. Thus there is a $\Si^0_1(\ep)$ set $\De^W$ computing $W$ on its domain.
We notice that replacing the occurrences of
$W(e,\de,u,x) = i$ for some $e\in\om$, $\de\in [0,\al]$, $u\in\cS$, $x \in \cN$, and $i=0,1$, in the definition of $f$
by the condition $\bigl(e,\de,u,x,i\bigr)\in\De^W$
defines a set $\De^f \subset \om \times [0,\alpha] \times \cS \times \cN \times \om$ which computes $f$ on its domain.
The set $\De^f$ is in $\Si^0_1(\ep)$. We get it from $\Si^0_1(\ep)$-recursivity of $Z$, $\LIM(\al)$, of the relation $<$ on $[0,\al]$, and further of functions and relations mentioned already above and in Lemma~\ref{L:volba-epsilon}.
We also apply Lemma~\ref{L:bounded-quantifiers} with the $\Si^0_1(\ep)$-recursive bound $|s|$ for $n$ several times.
\end{proof}

Using the recursion theorem, we get $e^*\in\om$ such that $f(e^*,\be,s,x) = W(e^*,\be,s,x)$, whenever $f(e^*,\be,s,x)$ is defined.

\begin{claim}
The value $f(e^*,\be,s,x)$ is defined for all $\be\in [0,\al]$, $s \in \cS$, $x\in\cN$.
\end{claim}

\begin{proof}[Proof of Claim]
Assume by contradiction that there is the smallest $\be \le \al$ such that for some $s$ and $x$ the value $f(e^*,\be,s,x)$ is not defined.
The function $f$ is defined for any $(e,0,s,x)$ and therefore $\beta > 0$.
If $\beta = \gamma+1$, then $f(e^*,\gamma,t,x)$ is defined for every $t \in \cS$ by our assumption.
Therefore either $W(e^*,\gamma,[s]_0,x) = 1$ or $W(e^*,\gamma,[s]_0,x) = 0$.
Since $s \in Z$ or $s \notin Z$, the value $f(e^*,\gamma+1,s,x) = f(e^*,\beta,s,x)$ is well defined, a contradiction.
If $\beta \in \LIM(\alpha)$, then $f(e^*,\gamma,t,x)$ is defined for every $\gamma < \beta$ and $t \in \Seq$. In particular,
$f(e^*,a(\beta,n),s([n,\cdot]^{-1}),x)$ is defined for every $n < |s|$.
Thus for every $n < |s|$ we have $W(e^*,a(\beta,n),s([n,\cdot]^{-1}),x) = 1$ or $W(e^*,a(\beta,n),s([n,\cdot]^{-1}),x) = 0$.
This shows that $f(e^*,\beta,s,x)$ is well defined, a contradiction again.
\end{proof}

Thus $f_{e^*}$ is the characteristic function of the set
\[
S = \{(\beta,s,x) \in [0,\alpha] \times \cS \times \cN \set f(e^*,\beta,s,x) = 1\}.
\]
The set $S$ is $\Si^0_1(\ep)$-recursive since $f$ is $\Si^0_1(\ep)$-recursive. We verify (a)--(c) for fixed $x \in \cN$.

\medskip\noindent
(a) By the definition of $S$ and (i), (i') above we have that $S_0^x$ is a tree since $S_0^x = \{s \in \Seq\set s \preceq x\}$.
Now assume that $1 \leq \beta \leq \alpha$ and $S_{\delta}^x$ is a tree for every $\delta < \beta$.
Suppose that $s \in S_{\beta}^x$ and $t \in \cS$ satisfy $t \preceq s$.

Assume that $\beta = \gamma + 1$. Then $W(e^*,\gamma,[s]_0,x) = 1$ and $s \in Z$ by (ii). Therefore $[s]_0 \in S_{\gamma}^x$. By induction hypothesis we have $[t]_0 \in S_{\gamma}^x$ since $[t]_0 \preceq [s]_0$. Since $Z$ is a tree by Claim above we have $t \in Z$ and we get $f(e^*,\beta,t,x) = 1$. Consequently, $t \in S_{\beta}^x$ by (ii).

Assume that $\beta \in \LIM(\alpha)$.  Then, using (iii) for $s$, we get for every $n < |t| \ (\leq |s|)$ that
\[
W\bigl(e^*,a(\beta,n),s([n,\cdot]^{-1}),x\bigr) = 1.
\]
Therefore $s([n,\cdot]^{-1}) \in S_{a(\beta,n)}^x$.
Using the induction hypothesis, we get $t([n,\cdot]^{-1}) \in S_{a(\beta,n)}^x$.
Now as before we easily infer $f(e^*,\beta,t,x) = 1$ and $t \in S_{\beta}^x$ using (iii) for $t$.

\medskip\noindent
(b) We define $\si^x_{\be}, \be \leq \al$, inductively. We set $\sigma_0^x = x$, for $\be=\ga+1 \leq \alpha$ we define
$\si^x_{\beta}$ by $[\si^x_{\beta}]_0 = \si^x_{\ga}$ and
\[
[\si^x_{\be}]_1(n)=
\begin{cases}
0                                                                &\text{if } (\si^x_{\ga})'(n) = 0, \\
\min\{l\ge 1\set (n,\si^x_{\ga},l-1) \in J_{\infty}\}            &\text{if } (\si^x_{\ga})'(n) = 1,
\end{cases}
\]
and for $\be \in \LIM(\al)$ we define $\si_{\be}^{x} = \bigvee_{n\in\om} \si^x_{a(\be,n)}$.

Using transfinite induction we verify that $[S^x_{\beta}] = \{\si^x_{\beta}\}$ for every $\beta \leq \alpha$.
Due to the definition of $f$, we see that $[S_0^{x}]=\{x\}$.
Let $1 \leq \beta \leq \alpha$ and assume that the assertion holds for every $\delta < \beta$.

Assume that $\beta = \ga + 1 \leq \alpha$. First we show that $\sigma^x_{\be} \in [S_{\be}^x]$.
Let $k \in \omega$. Then $[\si_{\be}^x] _0|k = \si_{\ga}^x|k \in S_{\ga}^x$ by the above definition and the induction hypothesis.
It remains to show $\si_{\be}^x|k \in Z$. Choose $n < k$. We verify \eqref{E:set-Z} for $n$ and $\si^x_{\be}|k$.
If $[\si^x_{\be}|k]_1(n) > k$, then we are done. If $[\si^x_{\be}|k]_1(n) = 0$, then $(\si_{\ga}^x)'(n) = 0$ by the definition of $\si^x_\be$
and therefore $(n,[\si^x_{\be}|k]_0,k) = (n,\si^x_{\ga}|k,k) \notin J_0$
by Remark~\ref{R:property-of-J}(a),(c). This verifies \eqref{E:set-Z} in this case.
Finally, if $[\si_{\beta}^x|k]_1(n) \in (0,k]$, then $(n,\si^x_{\ga},[\si^x_{\be}]_1(n)-1) \in J_{\infty}$ and
$(n,\si^x_{\ga},[\si^x_{\be}]_1(n)-2) \notin J_{\infty}$ by the definition of $\si^x_{\be}$.
Then we have $(n,[\si^x_{\be}|k]_0,[\si^x_{\be}|k]_1(n)-1) = (n,\si^x_{\ga}|k,[\si^x_{\be}|k]_1(n)-1)\in J_0$ and
$(n,[\si^x_{\be}|k]_0,[\si^x_{\be}]_1(n)-2) = (n,\si^x_{\ga}|k,[\si^x_{\be}|k]_1(n)-2) \notin J_0$.

Now assume that $z \in [S^x_{\beta}]$. For every $j \in \omega$ we have $z|j \in S^x_{\beta}$ and therefore $[z|j]_0 \in S_{\ga}^x$ for every $j \in \omega$.
This implies that $[z]_0 \in [S^x_{\ga}]$. The induction hypothesis gives $[z]_0 = \si^x_{\ga}$.
Fix $n \in \omega$. Then we have $z|k \in Z$ for every $k \in \omega$ by (ii).
If $k \geq [z(n)]_1$, then $z|k$ satisfies the second or the third condition in \eqref{E:set-Z}.
Thus if $[z(n)]_1 = 0$, we have $(n,\si^x_{\ga}|k,k) = (n,[z|k]_0,k) \notin J_0$ for every $k \geq [z(n)]_1$.
Using Remark~\ref{R:property-of-J}(a),(b),(d), this gives $([z]_0)'(n) = (\si^x_{\ga})'(n) = 0$.

If $[z(n)]_1 > 0$, then $(n,[z|k]_0,[z(n)]_1-1) \in J_0$ and $(n,[z|k]_0,[z(n)]_1-2) \notin J_0$ for every $k \geq [z(n)]_1$.
This implies that $(n,\si^x_{\ga},[z(n)]_1-1) \in J_{\infty}$ and $(n,\si^x_{\ga},[z(n)]_1-2) \notin J_{\infty}$.
This means that $(\sigma^x_{\gamma})'(n) = 1$ and $[z(n)]_1 = [\si^x_{\beta}]_1(n)$.
Consequently, $z = \si^x_{\be}$.
Thus we have $[S^x_{\beta}] = \{\si^x_{\beta}\}$.

Assume $\beta \in \LIM(\al)$. We show that $\sigma_{\be}^x \in [S_{\be}^x]$. Let $n,k \in \omega, n < k$.
According to induction hypothesis $\si_{a(\be,j)}^x \in [S_{a(\be,j)}^x]$ for every $j \in \omega$.
We have $(\si^x_{\be}|k)([n,\cdot]^{-1}) = \si^x_{a(\be,n)}|l$ for some $l \in \omega$ and so
\[
W(e^*,a(\be,n),(\si^x_{\be}|k)([n,\cdot]^{-1}),x) =
W(e^*,a(\be,n),\si^x_{a(\be,n)}|l,x) = 1.
\]
Thus $\sigma_{\be}^x|k \in S_{\be}^x$. Consequently, $\sigma_{\be}^x \in [S_{\be}^x]$.

Now assume that $z \in [S^x_{\beta}]$. Fix $n \in \omega$. Then $(z|j)([n,\cdot]^{-1}) \in S^x_{a(\beta,n)}$ for every $j \in \omega$.
This implies $z([n,\cdot]^{-1}) \in [S^x_{a(\beta,n)}]$.
Therefore by the induction hypothesis we get $z([n,\cdot]^{-1}) = \si^x_{a(\beta,n)}$.
Since $z = \bigvee_{n \in \omega} z([n,\cdot]^{-1})$, we get $z = \bigvee_{n \in \omega} \si^x_{a(\beta,n)} = \si^x_{\beta}$.

\medskip\noindent
(c) We define a $\Si_1^0(\ep)$-recursive function $w\colon \cN \to \cN$ by $w(z)(n) = \min\{1,[z]_1(n)\}$.
According to the previous part of the proof, we have $\sigma_0^x = x$, $w(\sigma_{\gamma+1}^x) = (\sigma_{\gamma}^x)'$ for every
$\gamma < \alpha$, $x \in \cN$, and $\sigma_{\beta}^x = \bigvee_{n \in \omega}\sigma^x_{a(\beta,n)}$
for every $\beta \in \LIM(\alpha)$ and $x \in \cN$.

Let $U^{[0,\alpha] \times \cN,\cN}_{\Si_1^0(\ep)}$ be the function guaranteed by Fact~\ref{F:recursionthms}(b) and Remark~\ref{R:extended-facts}.
Using Lemma~\ref{L:jumps}(2) for $W = \omega \times [0,\alpha]$ and $q = U^{[0,\alpha] \times \cN,\cN}_{\Si_1^0(\ep)}$,
we find a partial function $r_q\colon \omega \times [0,\alpha] \times \cN \rightharpoonup \cN$
which is $\Sigma_1^0(\ep)$-recursive on its domain such that $q(e,\beta,z)' = r_q(e,\beta,z')$
whenever $q(e,\beta,z)$ is defined. Finally, we define a partial function
$\fright\colon \omega \times [0,\alpha] \times \cN \rightharpoonup \cN$ by
\[
\fright(e,\beta,z) =
\begin{cases}
z                                                                      &\text{if } \beta = 0, \\
r_q(e,\gamma,w(z))                                                     &\text{if } \beta = \gamma + 1, \\
\bigvee_{n \in \omega} q\bigl(e,a(\beta,n),z([n,\cdot]^{-1})\bigr)     &\text{if } \beta \in \LIM(\alpha).
\end{cases}
\]
Then $\fright$ is $\Sigma_1^0(\ep)$-recursive on its domain by Lemma~\ref{L:volba-epsilon}(e),(g) and Lemma~\ref{L:bigvee}.
By Fact~\ref{F:recursionthms}(b) and Remark~\ref{R:extended-facts} there exists $\eright \in \omega$ such that
$\fright(\eright,\beta,z) = q(\eright,\beta,z)$ whenever $\fright(\eright,\beta,z)$ is defined.
The desired partial function $\rright$ is defined by $\rright(\beta,z) = \fright(\eright,\beta,z)$.
The function $\rright$ is clearly $\Sigma_1^0(\ep)$-recursive on its domain and we check by transfinite induction that
$\rright(\beta,\sigma_{\beta}^x) = x^{(\beta)}$.
If $\beta = 0$, then $\rright(0,\sigma_0^x) = \rright(0,x) = \fright(\eright,0,x) = x$
for every $x \in \cN$.
Now suppose that $\rright(\delta,\sigma_{\delta}^x) = x^{(\delta)}$ for every $\delta < \beta$ and $x \in \cN$.
If $\beta = \gamma + 1$ for some $\gamma$, then for every $x \in \cN$ we have
\[
\begin{split}
\rright(\beta,\sigma_{\beta}^x)
&= \fright(\eright,\beta,\sigma_{\beta}^x)
= r_q(\eright,\gamma,w(\sigma_{\beta}^x)) = r_q(\eright,\gamma,(\sigma_{\gamma}^x)') \\
&= q(\eright,\gamma,\sigma_{\gamma}^x)' = \fright(\eright,\gamma,\sigma_{\gamma}^x)' =
\rright(\ga,\sigma_{\ga}^x)'  = (x^{(\gamma)})' = x^{(\gamma+1)} =  x^{(\beta)}.
\end{split}
\]
If $\beta \in \LIM(\alpha)$, then for every $x \in \cN$ we have
\[
\begin{split}
\rright(\beta,\sigma_{\beta}^x) &= \bigvee_{n \in \omega} q\bigl(\eright,a(\beta,n),\sigma_{\beta}^x([n,\cdot]^{-1})\bigr)  \\
&= \bigvee_{n \in \omega} q\bigl(\eright,a(\beta,n),\sigma_{a(\beta,n)}^x\bigr)
= \bigvee_{n \in \omega} x^{(a(\beta,n))} = x^{(\beta)}.
\end{split}
\]

It remains to define the partial function $\rleft$. We put $\fleft(e,0,z) = z$ for $e\in\om$ and $z\in\cN$.
Now we define $\fleft(e,\be,z)$ for $e \in \omega$, $\be=\ga+1\le\al$, and $z \in \cN$ by determining $[\fleft(e,\be,z)]_i$ for $i=0,1$.
We recall that $r_1$ is the function from Lemma~\ref{L:jumps}(1).
We set $[\fleft(e,\be,z)]_0 = r_1(r_q(e,\gamma,z))$ and
\[
[\fleft(e,\be,z)]_1(n) =
\begin{cases}
0                                                                            &\text{if } r_q(e,\gamma,z)(n) = 0, \\
\min\{l \ge 1 \set (n,r_1(r_q(e,\gamma,z)),l-1) \in J_{\infty}\}             &\text{if } r_q(e,\gamma,z)(n) > 0.
\end{cases}
\]
If $\beta \in \LIM(\alpha)$, then we define
\[
\fleft(e,\beta,z) = \bigvee_{k \in \omega} q\bigl(e,a(\beta,k),r_3(a(\beta,k),\beta,z)\bigr),
\]
where $r_3$ is the function from Lemma~\ref{L:jumps}(3).
This finishes the definition of the function $\fleft$.

Using $\Sigma_1^0(\ep)$-recursivity of the functions used to define $\fleft$, we may check that the function $\fleft$ is $\Sigma_1^0(\ep)$-recursive on its domain.
By Fact~\ref{F:recursionthms}(b) and Remark~\ref{R:extended-facts} there is $\eleft \in \om$ such that $\fleft(\eleft,\be,z) = q(\eleft,\be,z)$ whenever the left-hand side is defined.
We set $\rleft(\beta,z) = \fleft(\eleft,\beta,z)$ whenever $\fleft(\eleft,\beta,z)$ is defined.
We check by induction that $\rleft(\be,x^{(\be)}) = \si_{\be}^{x}$ for all $\be \in [0,\al]$ and $x \in \cN$.
We will proceed by transfinite induction.

If $\beta = 0$, then we have $\rleft(0,x) = \fleft(\eleft,0,x) = x$ for every $x \in \cN$.
If $\beta = \gamma +1 \leq \alpha$, then we have
\[
[\rleft(\beta,x^{(\beta)})]_0
= \bigl[\fleft(\eleft,\gamma+1,x^{(\beta)})\bigr]_0  = \bigl[r_1\bigl(r_q(\eleft,\gamma,x^{(\beta)})\bigr)\bigr]_0.
\]
We have $r_q(\eleft,\gamma,x^{(\beta)}) = q(\eleft,\gamma, x^{(\gamma)})' = (\sigma_{\gamma}^x)'$.
Therefore we get
\[
[\rleft(\beta,x^{(\beta)})]_0 = [r_1((\sigma_{\gamma}^x)')]_0 = [\sigma_{\gamma}^x]_0.
\]
Now consider $n \in \omega$. If $r_q(\eleft,\gamma,x^{(\beta)})(n) = 0$, we have  $(\sigma_{\gamma}^x)'(n) = 0$. This implies that $[\sigma_{\beta}^x]_1(n) = 0$. Consequently,
$[\fleft(\eleft,\beta,x^{(\beta)})]_1(n) = [\sigma_{\beta}^x]_1(n)$.
If $r_q(\eleft,\gamma,x^{(\beta)})(n) > 0$, we have  $(\sigma_{\gamma}^x)'(n) > 0$ and
\[
\begin{split}
[\fleft(\eleft,\be,x^{(\beta)})]_1(n)
&= \min\{l \ge 1 \set (n,r_1(r_q(\eleft,\gamma,x^{(\beta)})),l-1) \in J_{\infty}\} \\
&= \min\{l \ge 1 \set (n,r_1((\sigma_{\gamma}^x)'),l-1) \in J_{\infty}\} \\
&= \min\{l\ge 1 \set (n,\sigma_{\gamma}^x,l-1) \in J_{\infty}\} = [\sigma_{\beta}^x]_1(n).
\end{split}
\]
This shows that $\rleft(\beta,x^{(\beta)}) = \sigma_{\beta}^x$ in this case.

Finally suppose that $\beta \in \LIM(\alpha)$ and $x \in \cN$. Then we have
\[
\begin{split}
\rleft(\beta,x^{(\beta)})
&= \fleft(\eleft,\beta,x^{(\beta)})  = \bigvee_{k \in \omega} q\bigl(\eleft,a(\beta,k),r_3(a(\beta,k),\beta,x^{(\beta)})\bigr) \\
&= \bigvee_{k \in \omega} \rleft\bigl(a(\beta,k),x^{(a(\beta,k))})\bigr) = \bigvee_{k \in \omega} \sigma_{a(\beta,k)}^x = \sigma_{\beta}^x
\end{split}
\]
and the proof is finished.
\end{proof}

\begin{proposition}\label{P:sigma-beta-x-jump}
There is a partial $\Sigma_1^0(\ep)$-recursive function $\Rleft\colon [0,\al]\times\cN\rightharpoonup\cN$
such that $(\si^x_\be)'_x = \Rleft(\be,x^{(\be+1)})$ for every $\beta \in [0,\alpha]$ and $x \in \cN$.
\end{proposition}

\begin{proof}
By (c) of Proposition~\ref{P:jump-trees-old} we have $\sigma_{\be}^{x} = \rleft(\be,x^{(\be)})$ and therefore
$(\si^x_\be)'_x=(\rleft(\be,x^{(\be)}))'_x$.
By the second part of Lemma~\ref{L:jump-in-x-and-in-0} we have a partial function $\wt{r}$ which is $\Si^0_1(\ep)$-recursive on its domain
such that $x^{(\be)}=\wt{r}(\be,x^{(\be)}_x)$ and therefore
$(\si^x_\be)'_x=(\rleft(\be,\wt{r}(\be,x^{(\be)}_x))'_x$.
Using Lemma~\ref{L:jumps}(2) for $W = [0,\alpha]$ and $q(\be,z)=\rleft(\be,\wt{r}(\be,z))$ we have
$(\si^x_\be)'_x=r_q(\be,x^{(\be+1)}_x)$.
By the first part of Lemma~\ref{L:jump-in-x-and-in-0} we have a partial function $r$ which is $\Si^0_1(\ep)$-recursive on its domain
such that $(\si^x_\be)'_x = r_q(\be,r(\be+1,x^{(\be+1)}))$.
So $\Rleft(\be,z)=r_q(\be,r(\be+1,z))$ defines the desired function.
\end{proof}

\section{Construction of \texorpdfstring{$T_{\alpha}$}{}}\label{S:construction-of-T-alpha}

We fix some notation to be used during this section.

\begin{notation}
Let $G = G^{\cN \times \om \times \cN}_{\Si_1^0(\ep)}$ and using Lemma~\ref{L:natural-base}(a)
let $G^*$ be the corresponding $\Si^0_1(\ep)$ subset of $\om \times \cS \times \om \times \cS$ with
\[
G = \bigcup\bigl\{\{e\}\times\cN(t)\times\{n\}\times\cN(s)\set (e,t,n,s)\in G^*\bigr\}.
\]
Using Fact~\ref{F:Gstar}(a) for extended product spaces, we find a $\Si_1^0(\ep)$-recursive subset $R = R_{G^*}$ of $\om \times\cS \times \om \times \cS \times \om$ such that its projection along the last coordinate is $G^*$.
\end{notation}

Further we will use the following notation used already in Section~\ref{S:proofs-of-main-results}.
For $k \in \om$ and $M \su \cN \cup \cS$, $(M)_k$ or $M_k$ stand for the set $\{z \in \cN \cup \cS \set k\w z\in M\}$.
The symbol $\overline{n}$ for $n \in \omega$ denotes the infinite constant sequence of $n$'s (Notation~\ref{N:notation-1}).

Before the construction of $T_\al = T_\al^{LN}$ from Proposition~\ref{P:E^LN}, we informally indicate the idea of its construction.
We employ a diagonal method to ensure (v) from Proposition~\ref{P:E^LN}.
We are going to construct $T^x_\al$ so that, for every $k\in\om$, we get that
$k\w z\in [T^x_\al]\cap G^{n,x}_k$ implies
that there is other $z'\ne z$ in $\cN$ such that $k\w z, k\w z' \in [T^x_\al]\cap G^{n',x}_k$
for some $n' \in \om$.
If  $W \su \cN \times \om$ is a $\Si^0_1(x,\ep)$ set, there is $e \in \om$ such that $W^n = G^{n,x}_e$ for every $n \in \om$.
Using the aforementioned property for $k=e$, we get that (v) of Proposition~\ref{P:E^LN} is fulfilled for $W$.

To this end we proceed as follows. For every $x\in\cN$  we put all sequences $t \in \cS$ with $|t| \leq 1$ to $T^x_\al$.
Then for every $k \in \omega$ we consider successively pairs of sequences $t^0,t^1\in\cS$ of length $0,1,2,\dots$ such that $[t^0]_1=[t^1]_1\in S^x_\al$ and $[t^i]_0=\cl{i}| |t|$, $i=0,1$, and we decide when we add them to $(T^x_\al)_k$.
If for both $k\w t^i$, $i=0,1$, and all $a,b,c,n \leq |t^i|$ we have $(k,k\w (t^i|a),n,x|b,c) \notin R$, then we put both $k \w t^i$ to $T^x_\al$.
In the opposite case we find the smallest integer $1\le m\in\om$ such that there are $i^*\in\{0,1\}$ and $a,b,c,n \leq m$ such that $(k, k\w (t^{i^*}|a), n, x|b, c) \in R$.
In this case we add to $T^x_\al$ just one such $k\w t^{i^*}$ (with suitably fixed $i^*$).
Moreover, we put to $T^x_\al$ all its extensions $k\w (t^{i^*}\w w^j)$ such that
$[t^{i^*}\w w^j]_1\in S^x_\al$, $[w^j]_0 = \cl{j}| |w^j|$, where $j=0,1$. It is not difficult to observe that this procedure leads to the conclusion indicated in the preceding paragraph.

The construction of $T_{\alpha} = T_\al^{SR}$ from \ref{P:E^SR} is similar. Instead of ``$i, j \in \{0,1\}$'' we consider ``$i,j \in \omega$'' with appropriate changes.

\begin{proof}[{\bf Proof of Proposition~\ref{P:E^LN}}]
We define the set $T_{\alpha} = T_{\al}^{LN}$ using the set $S$ from Proposition~\ref{P:jump-trees-old} as follows.
Let $x \in \cN$ be fixed.
We put every $s\in\cS$ with $|s|\le 1$ into $T^x_{\alpha}$.
Let further $k \in \omega$ and we consider $t \in \Seq$ such that $|t|\ge 1$.
Then $(k \w t,x) \in T_\al$ if and only if $[t]_1 \in S_{\alpha}^x$ and one of the following conditions (A)--(C) is satisfied by $t$.

\begin{itemize}
\item[(A)] We have $[t]_0 = \overline{i}||t|$ for some $i\in\{0,1\}$,
\begin{align*}
\forall a,b,c,n \le |t|  &\colon \bigl(k,k\w (t|a),n,x|b,c\bigr) \notin R, \\
\shortintertext{and for $t' \in \Seq$ defined by $[t']_0 = \cl{i'}||t|$, $i' = 1-i$, $[t']_1=[t]_1$ we have}
\forall a,b,c,n \le |t'| &\colon \bigl(k,k\w (t'|a),n,x|b,c\bigr) \notin R.
\end{align*}

\item[(B)] There exist $l \in \omega$, $l \leq |t|$, and $i \in \{0,1\}$ such that
$[t]_0 = (\overline{0}|l) \w (\overline{i}|(|t|-l))$,
\begin{align*}
\exists a,b,c,n \le l &\colon  \bigl(k,k \w (t|a),n,x|b,c \bigr) \in R, \\
\shortintertext{and}
\forall a,b,c,n < l \ \forall j \in \{0,1\} \ &\forall t' \in \Seq, [t']_0 \preceq\cl{j}, [t']_1 \preceq [t]_1 \colon \bigl(k, k \w (t'|a),n,x|b,c\bigr) \notin R.
\end{align*}

\item[(C)] There exist $l \in \omega$, $l \leq |t|$, and $i \in \{0,1\}$ such that $[t]_0 = (\overline{1}|l) \w (\overline{i}|(|t|-l))$,
\begin{align*}
\exists a,b,c,n \le l          &\colon  \bigl(k,k\w (t|a),n,x|b,c \bigl) \in R, \\
\forall a,b,c,n \le l\ \forall &t'\in \Seq, [t']_0=\cl{0}|a, [t']_1\preceq [t]_1 \colon  \bigl(k, k \w (t'|a),n,x|b,c \bigr) \notin R,  \\
\shortintertext{and }
\forall a,b,c,n < l\ \forall   &t'\in \Seq, [t']_0=\cl{1}|a, [t']_1\preceq [t]_1 \colon  \bigl(k, k \w (t'|a),n,x|b,c \bigr) \notin R.
\end{align*}
\end{itemize}

\medskip\noindent
The set
\[
\begin{split}
T_\al = \bigl\{(s,x) \in \cS \times \cN &\set |s| \leq 1 \bigr\} \\
&\cup  \bigl\{(k\w t,x) \in \cS \times \cN \set k\in\om, |t| \geq 1, \textrm{ and $t$ satisfies (A) or (B) or (C)} \bigr\}
\end{split}
\]
is $\Si^0_1(\ep)$-recursive since the sets $S_\al$ and $R$ are $\Si_1^0(\ep)$-recursive,
the quantification of $t'$ in (A), (B), and (C) is formal since $t'$ is uniquely defined from $t$, $i$, $j$, and $a$ by $\Si^0_1(\ep)$-recursive functions defined on $\Si_1^0(\ep)$-recursive domains, the other quantifications of $i,j\in\{0,1\}$ and of $a, b, c, n, l \in \om$ can be reformulated using the $\Si_1^0(\ep)$-recursive bound by $1$ and $|t|$ respectively, and
Lemma~\ref{L:bounded-quantifiers} can be used for extended spaces (see Remark~\ref{R:extended-facts}).

\medskip\noindent
We prove (i)--(v) from Proposition~\ref{P:E^LN}.

\medskip\noindent
(i) Since each $t\in\cS$ is in $T^x_\al$ if $|t|\le 1$, it is sufficient to check that $(T^x_\al)_k$ is a tree for each choice of $k \in \omega$ and $x \in \cN$. Let $t \in (T_{\alpha}^x)_k$ and $s \preceq t$.
Since $S_{\alpha}^x$ is a tree, we have that $[s]_1 \in S_{\alpha}^x$. Now we look at the conditions (A)--(C).
If $t$ satisfies (A), then $s$ satisfies (A) and therefore $s \in (T_{\alpha}^x)_k$. If $t$ satisfies (B), respectively (C), then there exists the corresponding $l$.
If $|s| \geq l$, then $s$ satisfies (B), respectively (C). If $|s| < l$, then $s$ satisfies (A).
Thus we have $s \in (T_{\alpha}^x)_k$ again.

\medskip\noindent
(ii) Suppose that $x \in \cN$, $k \in \omega$, and $k \w z \in [T_{\alpha}^x]$. Then $[z]_1|j \in S_{\alpha}^x$ for every $j \in \omega$.
Using Proposition~\ref{P:jump-trees-old}(b), we get $[z]_1 = \sigma_{\al}^x$ and (ii) follows.

\medskip\noindent
(iii) Suppose that $x \in \cN$, $k \in \omega$, and $k \w z \in [T_{\alpha}^x]$.
In (ii) we saw that $[z]_1 = \sigma_{\alpha}^x$.
Using the partial function $\rright$ from Proposition~\ref{P:jump-trees-old}(c), we get (iii).

\medskip\noindent
(iv) Let $x \in \cN$ and $k \in \omega$. Now we show that $(E_x)_k = [(T_{\alpha}^x)_k]$ is a two-point set.
For $i \in \{0,1\}$ define $z^i \in \cN$ by $[z^i]_0 = \overline{i}$ and $[z^i]_1 = \sigma_{\alpha}^x$.
We distinguish the following possibilities. If
\begin{equation}\label{E:being-out}
\forall i \in \{0,1\}\ \forall a,b,c,n \in \om \colon \bigl(k, k \w (z^i|a),n,x|b,c\bigr) \notin R,
\end{equation}
then using condition (A) we get $z^i \in [(T_{\alpha}^x)_k], i \in \{0,1\}$.
Further, (B) and (C) are not satisfied by $z^i|a$ for any $a\in\om$ and therefore $[T_{\alpha}^x] = \{k\w z^x_{k,0},k\w z^x_{k,1}\}$, where $z^x_{k,i}=z^i$ for $i=0,1$ in this case.

Suppose \eqref{E:being-out} is not satisfied. Then there exists $i \in \{0,1\}$ and a unique
$l \in \omega$ such that

\begin{equation}\label{E:being-in}
\begin{aligned}
\exists a,b,c,n \le l &\colon \bigl(k,k \w (z^i|a),n,x|b,c\bigr) \in R, \\
\forall j \in \{0,1\}\ \forall a,b,c,n < l   &\colon \bigl(k,k \w (z^j|a),n,x|b,c\bigr) \notin R.
\end{aligned}
\end{equation}

If \eqref{E:being-in} is satisfied for $i=0$, then for every $t \in \cS$ with $|t| \geq l$ we have $t \in (T_{\alpha}^x)_k$ if and only if
$t$ satisfies (B). Thus $t \in (T_{\alpha}^x)_k$ if and only if $[t]_1 \in S_\al^x$ and $[t]_0 = \overline{0}|l \w (\overline{j}|(|t|-l))$, where $j \in \{0,1\}$.
This and (ii) imply that $z \in [(T_{\alpha}^x)_k]$ if and only if $[z]_1 = \sigma_{\alpha}^x$ and $[z]_0 = \overline{0}|l \w \overline{j}$, where $j \in \{0,1\}$.
In this case, we denote $z$ by $z_{k,j}^x$, if $j \in \{0,1\}$ is used in the definition.

If \eqref{E:being-in} is satisfied for $i=1$ and not for $i=0$, then for every
$t \in \cS$ with $|t| \geq l$ we have $t \in (T_{\alpha}^x)_k$ if and only if
$t$ satisfies (C). Thus $t \in (T_{\alpha}^x)_k$ if and only if $[t]_1 \in S_\al^x$ and $[t]_0 = \overline{1}|l \w (\overline{j}|(|t|-l))$, where $j \in \{0,1\}$.
This and (ii) imply that $z \in [(T_{\alpha}^x)_k]$ if and only if $[z]_1 = \sigma_{\alpha}^x$ and $[z]_0 = \overline{1}|l \w \overline{j}$, where $j \in \{0,1\}$.
We use the notation $z_{k,j}^x$ for these $z$ also in this last case.

\medskip\noindent
(v) Suppose that $x \in \cN$  and  $W \subset \cN \times \om$ is a $\Si^0_1(x,\ep)$ set such that $\bigcup_{n \in \omega} W^n \supset E_x$.
Then we can find $e \in \omega$ such that $W= G^x_e$. In the particular case of $k=e$, we have
\begin{equation}\label{E:inclusion}
\{e\w z_{e,0}^x, e\w z_{e,1}^x\} \subset \bigcup_{n \in \om} G^{n,x}_e.
\end{equation}

If for every  $i \in \{0,1\}$ and every $1\le l \in \omega$ the sequence $z_{e,i}^x|l$ satisfies (A), then we have $(e,e\w z_{e,i}^x, n, x) \notin G$ for every $i \in \{0,1\}$ and $n \in \om$ (by the definition of $R$) which is in contradiction with \eqref{E:inclusion}.
If for some $i \in \{0,1\}$ and $1\le l \in \omega$ we have that $z_{e,i}^x|l$ satisfies (B) or (C),
then there exists $n \in \omega$ such that $(e, e \w z_{e,j}^x,n,x) \in G$ for every $j \in \{0,1\}$.
Therefore $e \w z_{e,0}^x, e \w z_{e,1}^x \in W^n$ and so $(W^n)_e \cap [(T_{\alpha}^x)_e]$
contains more than one point. Thus (v) is proved.
\end{proof}

\begin{proof}[{\bf Proof of Proposition~\ref{P:E^SR}}]
We define the set $T_{\alpha} = T_{\al}^{SR}$ as follows.
Let $x \in \cN$ be fixed.
We put every $s\in\cS$ with $|s|\le 1$ into $T^x_{\alpha}$.
Let further $k \in \omega$ and we consider $t \in \Seq$ such that $|t|\ge 1$.
Then $k\w t \in T_\al^x$
if and only if $[t]_1 \in S_{\alpha}^x$ and one of the following conditions (A) and (B) is satisfied by $t$.
\begin{itemize}
\item[(A)] We have either
\begin{itemize}
\item[--] there is $i\in\om, i>|t|$, such that $[t]_0\preceq \cl{i}$, or

\item[--] there is $i\in\om, i\le |t|$, such that $[t]_0=\cl{i}||t|$, and
\[
\forall i'\le |t|\ \forall t' \in \Seq, [t']_0 \preceq \cl{i'}, [t']_1\preceq [t]_1\ \forall a,b,c,n \le |t'| \colon
\bigl(k,k\w (t'|a),n,x|b,c\bigr) \notin R.
\]
\end{itemize}

\item[(B)] There exist $l,i,j \in \omega, i \leq l \leq |t|$, such that $[t]_0 = (\overline{i}|l) \w (\overline{j}|(|t|-l))$ and
\begin{align}
\exists a,b,c,n \le l &\colon \bigl(k,k \w (t|a),n,x|b,c\bigr) \in R,    \label{E:B-1}\\
\forall a,b,c,n < l \ \forall i' < l \ &\forall t' \in \Seq, [t']_0 \preceq \cl{i'}, [t']_1 \preceq [t]_1 \colon
\bigl(k,k \w (t'|a),n,x|b,c\bigr) \notin R, \textrm{ and } \label{E:B-2}\\
\forall a,b,c,n \le l \ \forall i'<i \ &\forall t' \in \Seq, [t']_0 \preceq \cl{i'}, [t']_1 \preceq [t]_1 \colon
\bigl(k,k\w (t'|a),n,x|b,c\bigr) \notin R.\label{E:B-3}
\end{align}
\end{itemize}

\medskip\noindent
Thus
\[
\begin{split}
T_\al = \bigl\{(s,x) \in \cS \times \cN &\set |s| \leq 1 \bigr\} \\
&\cup  \bigl\{(k\w t,x) \in \cS \times \cN \set k\in\om, |t| \geq 1, \textrm{ and $t$ satisfies (A) or (B)} \bigr\}
\end{split}
\]
is $\Si^0_1(\ep)$-recursive since the sets $S_\al$ and $R$ are $\Si^0_1(\ep)$-recursive and the quantifications
of $a$, $b$, $c$, $n$, $l$, $t'$
can be reformulated using the $\Si_1^0(\ep)$-recursive bound by $|t|$ and Lemma~\ref{L:bounded-quantifiers} can be used.

\medskip\noindent
We prove (i)--(v) from Proposition~\ref{P:E^SR}.

\medskip\noindent
(i) Since each $s\in\cS$ is in $T^x_\al$ if $|s|\le 1$, it is sufficient to check that $(T^x_\al)_k$ is a tree for each choice of
$k \in \omega$ and $x \in \cN$.
Let $t \in (T_{\alpha}^x)_k$ and $s \preceq t$.
Since $S_{\alpha}^x$ is a tree, we have that $[s]_1 \in S_{\alpha}^x$.
Now we look at the conditions (A) and (B).
If $t$ satisfies (A), then $s$ satisfies (A),
and therefore $s \in (T_{\alpha}^x)_k$. If $t$ satisfies (B), then there exist the corresponding $l,i$, and $j$.
If $|s| \geq l$, then $s$ satisfies (B) with the same $l,i$, and $j$. If $|s| < l$, then $s$ satisfies (A).
Thus we have $s \in (T_{\alpha}^x)_k$ again.

\medskip\noindent
(ii) Suppose that $x \in \cN, k \in \omega$, and $z \in [(T_{\al}^x)_k]$. Then $[z]_1|l \in S_{\al}^x$ for every $l \in \omega$.
Using Proposition~\ref{P:jump-trees-old}(b), we get $[z]_1 = \sigma_{\al}^x$ and (ii) follows.

\medskip\noindent
(iii) Suppose that $x \in \cN$, $k \in \omega$, and $z \in [(T_{\al}^x)_k]$. In (ii) we verified that $[z]_1 = \sigma_{\alpha}^x$.
Using the partial function $\rright$ from Proposition~\ref{P:jump-trees-old}(c), we get (iii).

\medskip\noindent
(iv) Suppose that $x \in \cN$ and $k \in \omega$. Let us remark that it is obvious from the definition and (ii) that each element $z$ of $[(T_{\al}^x)_k]$ is of the form $[z]_0=\cl{i}|l\w \cl{j}$ for some $i,j,l\in\om$ and $[z]_1=\si^x_\al$.
Let us consider $z^i\in\cN$, $i\in\om$, defined by $[z^i]_0=\cl{i}$
and $[z^i]_1=\si^x_\al$. We distinguish the following possibilities. If
\begin{equation}\label{E:being-outSR}
\forall l \in \omega \ \forall i,a,b,c,n \leq l \colon \bigl(k,k\w (z^i|a),n,x|b,c\bigr) \notin R,
\end{equation}
then, using condition (A), we get $z^i \in [(T_{\alpha}^x)_k]$ for every $i \in \om$.
Further, (B) is never satisfied and therefore $[(T_{\alpha}^x)_k] = \{z^i \set i \in \om\}$.
Thus $(E_x^{SR})_k$ is closed and discrete.

Suppose \eqref{E:being-outSR} is not satisfied. Let us denote by $l_0$ the smallest $l \in \om$ such that
\begin{equation}\label{E:being-inSR}
\exists i,a,b,c,n \le l  \colon \bigl(k, k \w (z^i|a),n,x|b,c\bigr) \in R.
\end{equation}
Find the smallest $i \in \omega$ such that
\begin{equation}\label{E:condition}
\exists a,b,c,n \le l_0   \colon \bigl(k,k \w (z^i|a),n,x|b,c\bigr) \in R
\end{equation}
and denote it by $i_0$. Notice that $i_0\le l_0$ by the choice of $l_0$.
Find $a_0, b_0, c_0, n_0 \leq l_0$ such that we have $\bigl(k,k\w (z^{i_0}|a_0),n_0,x|b_0,c_0\bigr) \in R$.
The sequence $z^{i_0}|a_0$ is in $(T_{\alpha}^x)_k$ since $z^{i_0}|a_0$ satisfies (B) by the choice of $a_0$ and $i_0$.

Suppose that $z \in [(T_{\al}^x)_k]$. Then $[z]_1 = \sigma_{\al}^x$ and $[z]_0 = \cl{i}|m \w \cl{j}$ for some $i,j,m \in \omega$.
We show that if $i = j$, then $i = i_0$ and if $i \neq j$, then $i = i_0$ and $m = l_0$.
First assume that $i = j$. Then $z = z^i$. Set $d = \max\{i,l_0\}$. Suppose that $z|d$ satisfies (A).
Since $i \leq d$, we have that $z|d$ satisfies the second condition in (A).
As $i_0\le l_0\le d$, this implies, in particular, $\bigl(k,k \w (z^{i_0}|a_0),n_0,x|b_0,c_0\bigr) \notin R$, a contradiction.
Thus $z|d$ satisfies (B). Find the corresponding $l$. Using  \eqref{E:B-1} and minimality of $l_0$ we get $l_0 \leq l$.
Assume that $l_0 < l$. Then  \eqref{E:B-2} implies $\bigl(k,k \w (z^{i_0}|a_0),n_0,x|b_0,c_0\bigr) \notin R$, a contradiction.
Thus $l = l_0$. Using \eqref{E:B-3} we infer $i = i_0$.

Suppose that $i \neq j$. Then $\cl{i}|m \w j$ does not satisfy (A). Thus it satisfies (B) with $l = m$.
The condition \eqref{E:B-2} gives $m = l_0$  and then \eqref{E:B-3} gives $i_0 = i$.

It follows that $[(T^x_\al)_k] = \{z^x_{k,j}\set j \in \om\}$,
where $[z^x_{k,j}]_1 = \si^x_\al$ and $[z^x_{k,j}]_0 = \cl{i_0}|l_0\w \cl{j}$.
Consequently, $(E_x^{SR})_k$ is closed and discrete and therefore $E_x^{SR}$ is closed and discrete.

\medskip\noindent
(v) Let $x \in \cN$. Suppose towards contradiction that  $W \subset \cN \times \om$ is a $\Si^0_1(x,\ep)$ set such that
$\bigcup_{n \in \omega} W^n \supset E^{SR}_x$ and the set $W^n \cap E^{SR}_x$ is finite for every $n \in \omega$.
Then we can find $e \in \omega$ such that $W = G^x_e$. Thus we have
\begin{equation}\label{E:inclusionSR}
E^{SR}_x \subset \bigcup_{n\in\om} (G_e^x)^n.
\end{equation}
Let $e\w z \in E^{SR}_x$. If for every $l \in \om$ we have that $z|l$ satisfies (A), then $\bigl(e,e \w z,n,x\bigr) \notin G$
for every $n \in \om$ (by the definition of $R$) which is in contradiction with \eqref{E:inclusionSR}.
If for some $l' \in \om$ we have that $z|l'$ satisfies (B), then there exist $l,i \in \omega$ such that
any $w^j \in \cN$, $j\in\om$, with $[w^j]_0 = (\cl{i}|l) \w \cl{j}$ and  $[w^j]_1 = \sigma_{\alpha}^x$
satisfies $e\w w^j \in G_e^{n,x} \cap E^{SR}_x$.
Therefore $(W^n)_e \cap (E^{SR}_x)_e$ is infinite and $W^n \cap E^{SR}_x$ is infinite as well. This proves (v).
\end{proof}

\section{Reduction of effective Borel classes to effectively open sets}\label{S:reduction-of-Borel-classes}

Let $T \subset \cS$ be a nonempty well-founded tree on $\om$.
The tree $\{s \in \cS \set t^{\wedge}s\in T\}$, $t \in T$, is denoted by $T_t$.
Let $\rho^T\colon \cS \to\om_1$ be the corresponding {\it rank function} (see \cite[2D.1]{Moschovakis}).
We use also an alternative notation $h^T_t=\rho^T(t)$ for $t\in T$.
The symbol $T_{\max}$ denotes the set $\{t\in T\set h^T_t=0\}$ of all {\it terminal elements} of $T$.

\begin{definition}\label{D:B-z-T,H}
Let $X$ be a  metric  space.
Given a nonempty well-founded tree $T$ on $\om$ and an open set $\Th \subset X \times T_{\max}$,
we define sets
$B^{T,\Th}(t)$ for all $t\in T$ by the scheme
\[
B^{T,\Th}(t) =
\begin{cases}
\Th^t                                                                       &\text{if } t \in T_{\max}, \\
\bigcup \{X \sm B^{T,\Th}(t \w k) \set k \in \om, t \w k \in T\}           &\text{if } t \in T \sm T_{\max}.
\end{cases}
\]
\end{definition}

\begin{remark}
We may easily observe that $B^{T,\Th}(t)$ is a $\bSi^0_{1+h^T_t}$ set in $X$ for every $t\in T$.
\end{remark}

\begin{lemma}[{cf. \cite[Chapter 33]{Miller}}]\label{L:classrepresentation}
Let $X$ be a  metric  space,  $\be$ be a countable ordinal, and $B$ be a $\bSi^0_{1+\be}$ subset of $X$.
Then there exists a nonempty well-founded tree $T$ on $\om$ and an open set
$\Th \subset X \times T_{\max}$ such that $h^T_\emptyset\le\be$ and $B = B^{T,\Theta}(\emptyset)$.

Moreover, the tree $T$ can be chosen such that
\begin{equation}\label{full-tree}
t\w l\in T\textrm{ for every }l\in\om\textrm{ whenever } t \in T \setminus T_{\max}.
\end{equation}
\end{lemma}

\begin{proof}
We proceed by transfinite induction over $\beta$. If $\beta = 0$, then $T = \{\emptyset\}$ and $\Th = B \times \{\emptyset\}$ satisfy the desired conclusions.

Suppose that $\beta > 0$ and the assertion holds for every ordinal $\ga < \be$. Let $B \in \bSi^0_{1+\be}\restriction X$ be given.
Then there are $0\le\ga_k < \be$ and $B(k) \in \bSi^0_{1+\ga_k}\restriction X$ for all $k\in\om$ such that
$B = \bigcup_{k\in\om} (X\sm B(k))$.
By the induction hypothesis there are trees $T(k)$ satisfying (\ref{full-tree}), and
open sets $\Th(k)\su X\times T(k)_{\max}$ with the corresponding properties $h^{T(k)}_\emptyset \le \gamma_k$
and $B(k) = B^{T(k),\Th(k)}(\emptyset)$. Thus setting
\[
T = \bigcup_{k \in \om}\{k \w t \set t \in T(k)\} \cup \{\emptyset\}
\]
and
\[
\Th = \bigl\{(z,k \w s) \in X \times \Seq \set k \in \omega,  k \w s \in T_{\max}, (z,s) \in \Th(k)\bigr\},
\]
we get that $T$ is a nonempty well-founded tree with $h^T_\emptyset\le\be$ satisfying (\ref{full-tree})
and $\Th$ is an open subset of $X \times T_{\max}$ such that $B^{T,\Th}\bigl((k)\bigr) = B(k)$ for every $k\in\om$, and therefore $B^{T,\Th}(\emptyset)=B$.
\end{proof}

\begin{lemma}\label{L:reduction}
Let $a \in \cN$, $T \su \cS$, $\Th \su \cN \times \cS$ be such that
\begin{itemize}
  \item $T$ is a nonempty well-founded $\Si^0_1(a,\ep)$ tree  satisfying (\ref{full-tree}),

  \item $\Th \su \cN \times T_{\max}$ is a $\Si^0_1(a,\ep)$ set, and

  \item the restricted rank function $\rho^T_\al \colon t \in T \rightharpoonup h^T_t\in [0,\al]$ is $\Si^0_1(a,\ep)$-recursive on its $\Si^0_1(a,\ep)$-recursive domain
      $D(\rho^T_\al) = \{t\in T\set h^T_t\in [0,\al]\}$.
\end{itemize}
Then there is a $\Sigma_1^0(a,\ep)\restriction (\cN \times \cS)$-recursive set $\wh{\Th}\su \cN \times T$ such that for every
$t \in D(\rho^T_\al)$ we have
\begin{equation}\label{E:jump-representation}
y \in B^{T,\Theta}(t)  \Leftrightarrow   y^{(h^T_t)}_{a} \in \wh{\Th}^t.
\end{equation}
\end{lemma}

\begin{proof}
To prove the statement we shall use the following observations. Notice that Claim A is related to Lemma~\ref{L:jumps}(1).

\begin{claima}
There is a $\Si_1^0(\ep)$-recursive $r\colon \om \times T \to \om$ such that for every $z \in \cN$, $t\in T$, and $e \in \omega$ we have
\[
z \in (G^{\cN \times \cS \times \cN^2}_{\Si_1^0})^{t,a,\ep}_e  \Leftrightarrow  z'_{a}\bigl(r(e,t)\bigr) = 1.
\]
\end{claima}

\begin{proof}[Proof of Claim A]
Set
\[
W = \bigl\{(e,t,w,z) \in \om \times \cS \times \om \times \cN \set (e,z,t,a,\ep) \in G^{\cN \times \cS \times \cN^2}_{\Si_1^0}\bigr\}.
\]
The set $W$ is in $\Sigma_1^0(a,\ep)$. Notice that the choice of $w$ has no influence on the validity of $(e,t,w,z)\in W$.
We find $e_W \in \omega$ such that $W = (G^{\omega\times \cS \times \om \times \cN^3}_{\Si_1^0})^{a, \ep}_{e_W}$.
We define $r\colon \omega\times T \to \omega$ by $r(e,t) = S^{\omega\times\cS,\omega \times \cN^3}(e_W,e,t)$.
For $(e,t,z)\in\om\times\cS\times\cN$ we get
\[
\begin{split}
z \in (G^{\cN \times \cS \times \cN^2}_{\Si_1^0})^{t, a, \ep}_{e}
&\Leftrightarrow \bigl(e,t,S^{\omega\times\cS,\omega \times \cN^3}(e_W,e,t),z\bigr) \in W \\
&\Leftrightarrow
\bigl(e_W,e,t,S^{\omega\times\cS,\omega \times \cN^3}(e_W,e,t), z, a, \ep \bigr) \in G^{\omega \times \cS \times \omega \times \cN^3}_{\Si_1^0} \\
&\Leftrightarrow  \bigl(S^{\omega\times\cS,\omega \times \cN^3}(e_W,e,t),S^{\omega\times\cS,\omega \times \cN^3}(e_W,e,t),z,a, \ep \bigr) \in G^{\omega \times \cN^3}_{\Si_1^0} \\
&\Leftrightarrow  \bigl(r(e,t),r(e,t),z,a, \ep \bigr) \in G^{\omega \times \cN^3}_{\Si_1^0} \ \Leftrightarrow \ z'_{a}(r(e,t)) = 1.
\end{split}
\]
The third equivalence follows from Fact~\ref{F:recursionthms}(a) and Remark~\ref{R:extended-facts}.
The last equivalence uses Definition~\ref{D:Turing-jump}.
\end{proof}

Let us use the notation $h_t = h_t^T$ in the remaining part of the proof.

\begin{claimb}
Let $r_3$ be the function from Lemma~\ref{L:jumps}(3) and $r$ be the function from Claim A.
We define a partial function $q\colon \omega \times \cN \times \Seq \rightharpoonup \cN$ by
\[
q(e,z,t)(k) = \bigl(r_3(h_{t \w k}+1,h_t,z)\bigr)(r(e,t \w k))
\]
on the domain
\[
D(q) = \bigl\{(e,z,t) \in \om \times \cN \times T \set 0 < h_t \le \al, z = y^{(h_t)}_a \textrm{ for some } y \in \cN\bigr\}.
\]
Then there is a $\Si^0_1(a,\ep)$ set $Q\su \om \times \cN \times \cS$ such that
\[
\bigl\{(e,z,t) \in D(q)\set \exists k\in\om \colon q(e,t,z)(k) = 0 \bigr\} = Q \cap D(q).
\]
\end{claimb}

\begin{proof}[Proof of Claim B]
The partial function $q$ is $\Sigma_1^0(a,\ep)$-recursive on its domain $D(q)$.
The set $E := \{v \in \cN \set \exists k \in \omega \colon v(k) = 0\}$ is a $\Sigma_1^0(a,\ep)$  set in $\cN$ as well.
Using substitution property of the class $\Sigma_1^0(a,\ep)$ on extended product spaces, we find $Q \in \Sigma_1^0(a,\ep)\restriction \omega \times \cN \times \Seq$ such that
\[
(e,z,t) \in D(q) \Rightarrow [(e,z,t) \in Q \Leftrightarrow q(e,z,t) \in E].
\]
\end{proof}

Now we get from the assumptions of the lemma that the set
\[
\wh{Q} = \Bigl(\om \times \Th \Bigr) \cup \Bigl(Q \cap \bigr(\om \times \cN  \times \{t\in T\set h_t>0\}\bigr)\Bigr)
\]
is a $\Si^0_1(a,\ep)$ subset of $\om \times \cN \times \cS$ and using Fact~\ref{F:recursionthms}(a) for extended product spaces we find $e^*\in\om$ such that $(G^{\cN \times \cS \times \cN^2}_{\Si_1^0})^{a,\ep}_{e^*} = \wh{Q}_{e^*}$.
Finally, we define $\wh{\Th} = \wh{Q}_{e^*}$.

We verify \eqref{E:jump-representation} by induction on the rank $h_t$.
For $t \in T$ with $h_t = 0$, i.e., $t \in T_{\max}$, we have $B^{T,\Theta}_t = \Theta^t$ by Definition~\ref{D:B-z-T,H}
and $\Th^t = (\wh{Q}_{e^*})^t = \wh{\Th}^t$. Thus we get \eqref{E:jump-representation}.

For $t \in T$ such that $0 < h_t \le \al$, we have
\[
\begin{split}
y \in B^{T,\Theta}(t)
&\Leftrightarrow \exists k \in \omega\colon y \notin B^{T,\Th}_{t \w k} \ \Leftrightarrow \
\exists k \in \omega\colon y_{a}^{(h_{t \w k})} \notin \wh{\Th}^{t\w k} = (G^{\cN \times \cS \times \cN^2}_{\Si_1^0})^{t \w k,a,\ep}_{e^*} \\
&\Leftrightarrow \exists k \in \omega\colon y_{a}^{(h_{t \w k}+1)}(r(e^*,t\w k)) = 0 \
\Leftrightarrow \ \exists k \in \omega\colon q(e^*,y^{(h_t)}_{a},t)(k) = 0  \\
&\Leftrightarrow (y^{(h_t)}_{a},t) \in \wh{Q}_{e^*} \ \Leftrightarrow \ y^{(h_t)}_{a} \in \wh{\Th}^t.
\end{split}
\]
We used successively
the definition of $B^{T,\Th}(t)$, $t \in T$;
the induction hypothesis, the fact that the rank $h_{t\w k}$ is smaller than $h_t$, and the choice of $e^*$;
Claim~A for $z=y^{(h_{t\w k})}_a$;
Lemma~\ref{L:jumps}(2) and definition of $q$;
Claim~B for $z=y^{(h_t)}_a$; definition of $\wh{\Th}$.
\end{proof}

\begin{proof}[{\bf Proof of Proposition~\ref{P:jump-and-classes-1}}]

Let $A(n) \in \bSi^0_{\al}\restriction \cN^2$ be given for every $n \in \omega$. We will need the following claim.

\begin{claim}\label{L:recursive-class-representation}
There exist $a \in \cC$, a nonempty well-founded tree $T$ satisfying \eqref{full-tree} on $\omega$,
and a $\Si^0_1(a,\ep)\restriction (\cN \times \cS)$ set $\Th \su \cN \times T_{\max}$  such that
\begin{itemize}
\item $h^T_\emptyset\le\alpha + 1$,

\item the set $T$ is $\Si^0_1(a,\ep)$-recursive,

\item the mapping $t \mapsto h^T_t$ is a partial $\Si^0_1(a,\ep)$-recursive function on $\{t\in T\set h^T_t\le\al\}$,

\item $h^T_{(n)}\le\al, n \in \omega$, and

\item $A(n)_{a} = B^{T,\Th}\bigl((n)\bigr), n \in \omega$.
\end{itemize}
\end{claim}

\begin{proof}[Proof of Claim]
Since $\alpha$ is an infinite ordinal number, we have $1+\alpha=\alpha$.
Thus we may use Lemma~\ref{L:classrepresentation} to find for every $n \in \omega$ a nonempty well-founded tree $T(n)$ satisfying (\ref{full-tree}) with $h^{T(n)}_\emptyset\le\alpha$ and an open set $\Xi(n) \subset \cN^2 \times T(n)_{\max}$ such that $A(n) = B^{T(n),\Xi(n)}_{\emptyset}$. Then we define
\[
T = \bigl\{(n) \w t \set n \in \omega, t \in T(n)\bigr\} \cup \{\emptyset\}.
\]
The set $T$ is obviously a nonempty well-founded tree satisfying \eqref{full-tree} and  $h^T_\emptyset \le \alpha + 1$.
So $h^T_{(n)} = h^{T(n)}_\emptyset\le\al$.
Further we define $\Xi \subset \cN^2 \times \Seq$ by
\[
\Xi = \bigl\{(x,y,n \w s) \in \cN^2 \times \Seq \set n \in \omega, s \in T(n)_{\max}, (x,y,s) \in \Xi(n)\bigr\}.
\]
The set $\Xi$ is an open subset of $\cN^2 \times \Seq$. Using Fact~\ref{F:relativization}, we find $a\in\cC$ such that $\Xi$ is a $\Si^0_1(a,\ep)$
subset of $\cN^2 \times \Seq$, $T$ is $\Si^0_1(a,\ep)$-recursive, and the function $t \mapsto h_t^T$ is
$\Si^0_1(a,\ep)$-recursive on the $\Si^0_1(a,\ep)$-recursive set $\{t\in T\set h^T_t\le\al\}$.
The set $\Th=\Xi_a$ is $\Si^0_1(a,\ep)$ in $\cN \times \cS$.
It remains to prove the equalities from the last item of Claim.
Suppose that $t \in T_{\max}$. Then there exists $n \in \omega$ and $s \in \Seq$ such that
$n \w s = t$ and $s \in T(n)_{\max}$. We have
\[
B^{T,\Th}(n \w s) = B^{T,\Xi_a}(t) = \Xi^t_a = \Xi(n)^s_a = B^{T(n),\Xi(n)}(s)_a.
\]
The sets $B^{T,\Xi_a}((n))$ and $B^{T(n),\Xi(n)}(\emptyset)_a=(A(n))_a$ are for each $n\in\om$ defined by the same operations from the families of sets $\Xi(n)^s_a$, $n\w s\in T_{\max}$.
Therefore the sets  $B^{T,\Th}\bigl((n)\bigr)$ and $B^{T(n),\Xi(n)}(\emptyset)_a = A(n)_a$ coincide for every $n \in \omega$.
This proves the equalities.
\end{proof}

Now we find $\wh{\Th_a}$, for $T$ and $\Th_a$ from the above Claim, using Lemma~\ref{L:reduction}.
For the particular case of $t=(n)\in T$ we have
\[
y\in A(n)_a \Leftrightarrow y^{(h^T_{(n)})}_{a} \in \wh{\Th_a}^{(n)}.
\]
Using Lemma~\ref{L:jumps}(3), we have that $y^{(h^T_{(n)})}_{a}=r_3(h^T_{(n)},\al,y^{(\al)}_a)$.
Thus
\[
H = \bigl\{(z,n) \in \cN \times \omega \set \bigl(r_3(h_{(n)}^T,\al,z),(n)\bigr) \in \wh{\Th_a}\bigr\}
\]
is the desired set from Proposition~\ref{P:jump-and-classes-1}.
It is $\Si^0_1(a,\ep)$ since $t\mapsto h_t^T$ is $\Si^0_1(\ep)$-recursive and defined for all $t=(n)$, $n\in\om$, and the function $r_3$ is
$\Si^0_1(\ep)$-recursive and $(h_{(n)}^T,\al,y^{(\al)}_a)\in D(r_3)$.
\end{proof}

\begin{lemma}\label{P:jump-and-classes-2}
Let $\Om$ be a type zero extended product space, $a\in\cN$, and $C\in\Sigma_2^0(a,\ep)\restriction \cN \times \Om$.
Then there is $W\in \Sigma_1^0(a,\ep)\restriction \cN \times \Om$ such that for all
$(z,s) \in \cN \times \Om$ we have
\[
(z,s) \in C \Leftrightarrow (z'_a,s) \in W.
\]
\end{lemma}

\begin{proof}
Due to Remark~\ref{R:extended-facts} we may assume that $\Om=\om^k$ for some positive $k\in\om$.
There is $\Theta \in \Sigma_1^0(a,\ep) \restriction (\cN \times \omega^{k+1})$ such that for every
$(z,s) \in \cN \times \omega^k$ we have
\[
(z,s) \in C \Leftrightarrow  \exists l \in \omega \colon (z,s\w l) \notin \Theta.
\]
We consider the tree $T = \bigcup_{i=0}^{k+1} \om^i$. Then $B^{T,\Theta}(s) = C^s$ for $s \in \om^k$ and
$B^{T,\Theta}(t) = \Theta^t$ for $t \in \om^{k+1} = T_{\max}$.
Using Lemma~\ref{L:reduction}, we find a  $\Si^0_1(a,\ep)\restriction \cN \times \cS$ set
$\wh{\Theta} \subset \cN \times T$ such that for every $s \in \om^k \subset T$ we have
\[
z \in B^{T,\Theta}(s) = C^s \Leftrightarrow z'_a \in \wh{\Theta}^s,
\]
since $h^T_s=1$ for such $s$.
Thus $W= \{(w,s) \in \wh{\Theta}\set s \in \om^k\}$ is the desired set since intersections of $\Sigma_1^0(a,\ep)\restriction \cN \times \cS$ sets with $\cN \times \om^k$ are $\Sigma_1^0(a,\ep)$ by Lemma~\ref{L:volba-epsilon}(e).
\end{proof}

\section{Harrington's construction}\label{S:Harrington-method}

We start with some notation needed later on. The main ideas of the construction are informally presented in Remark~\ref{R:basic-idea}.

\subsection{The set \texorpdfstring{$\bbA$}{} of admissible triples \texorpdfstring{$(\be,p,x)$}{}}

\begin{definition**}\label{D:Q}\hfil
\begin{enumerate}[(a)]
\item The set $\cQ \subset \cP$ contains all nonempty sequences of the form
\[
p = (\tau,n_0,\eta_0,\dots,n_{r-1},\eta_{r-1}),
\]
where $r \in \omega$, $\tau \in \cS$, $n_i \in \omega$, and $\eta_i \in \cS^*$ for $i = 0,\dots,r-1$.
Here we identify $n \in \omega$ and the sequence $(n)$.
See Definition~\ref{D:sequences}(b) for the definition of $\cP$.

\item Let $p = (\tau,n_0,\eta_0,\dots,n_{r-1},\eta_{r-1}) \in \cQ$ and $i < r$.
We set
\begin{gather*}
\tau(p) = \tau \in \cS, \qquad n_i(p) = n_i \in \omega, \qquad \eta_i(p) = \eta_i \in \Seqm,  \qquad r(p) = r \in \omega, \\
\nu(p) = (n_0,\dots,n_{r-1}) \in \Seq, \qquad \xi(p) = [n_0 \w \eta_0 \w \dots \w n_{r-1} \w \eta_{r-1}]_1 \in \Seq, \\
\zeta(p) = \tau(p) \w [\nu(p)]_0 \in \Seq.
\end{gather*}

\item For $p \in (\Seqm)^\om$ such that $p|(1+2k) \in \cQ$ for every $k \in \omega$ we define $\zeta(p)$ as the unique infinite sequence
with initial segments $\zeta\bigl(p|(1+2k)\bigr)$ for $k \in \om$.
\end{enumerate}
\end{definition**}

We may notice that $|p|\le 2|\wh{p}|+1$ for $p\in\cQ$. We use this observation in the proof of Lemma~\ref{L:recursivity-of-f}.
Throughout Section~\ref{S:Harrington-method} the letter $r$ will be used just in the sense of the above definition.

\begin{lemma**}\label{L:recursivity-of-Q}\hfil
\begin{enumerate}[\upshape (a)]
\item The set $\cQ$ is $\Si^0_1(\ep)$-recursive in $\cP$.

\item The mappings $(i,p)\mapsto n_i(p)\in\om$ and $(i,p)\mapsto \eta_i(p)\in\cS^*$ are $\Si^0_1(\ep)$-recursive on the $\Si^0_1(\ep)$-recursive domain $\{(i,p)\in\om\times\cQ\set 2i+1<|p|\}$.
      The mappings $\tau, r, \nu, \xi$, and $\zeta$ are $\Si^0_1(\ep)$-recursive on $\cQ$.
\end{enumerate}
\end{lemma**}

\begin{proof}
(a) Using Lemma~\ref{L:volba-epsilon}(e), we see that the functions $p\in\cP \mapsto |p|$, $(i,p) \in \omega \times \cP \mapsto p_i$,
and $s \in \cS \mapsto |s|$ are $\Si^0_1(\ep)$-recursive. Of course, $\cS$ is $\Si^0_1(\ep)$-recursive in $\cS^*$ by our definitions.
For $p \in \cP$ we have $p \in \cQ$ if and only if $p_0 \in \cS$, $|p|$ is odd, and
\[
\forall k \in \omega, 2k+1 <|p|\colon p_{2k+1} \in \cS \land |p_{2k+1}| = 1.
\]
Using Lemma~\ref{L:bounded-quantifiers} with the $\Si_1^0(\ep)$-recursive bound $k < \frac12 (|p|-1)$, the above recalled facts, and Fact~\ref{F:partialrecursivefunctions}\eqref{I:composition},
we get that $\cQ$ is $\Si_1^0(\ep)$-recursive.

\medskip\noindent
(b) We apply $\Si^0_1(\ep)$-recursivity of the mappings recalled in (a) and $\Si^0_1(\ep)$-recursivity of concatenations, of the mappings $[\,\,]_0$, $[\,\,]_1$, and of the equality relation on extended product spaces of type $0$. We need also the $\Si^0_1(\ep)$-recursivity of $(i,p)\mapsto (2i,p)$ or $(i,p)\mapsto (2i+1,p)$ to get $\Si^0_1(\ep)$-recursivity of their composition with $(n,p)\mapsto p_n$.
To this end we use Lemma~\ref{L:volba-epsilon} several times.
\end{proof}

The set $\bbJ$ defined in the next definition is a key notion for the construction.
See Definition~\ref{D:J-jump} for the definition of $J_*$.

\begin{definition**}\label{D:definition-of-J}
We set
\[
\bbJ = \{(p,x)\in \cQ\times\cN \set \forall i<r(p) \colon (i,\wh{p},x) \in J_*\Leftrightarrow \eta_i(p)\ne -\}.
\]
\end{definition**}

\begin{lemma**}\label{L:recursivity-of-F}
The set $\bbJ$ is $\Si^0_1(\ep)$-recursive.
\end{lemma**}

\begin{proof}
By Lemma~\ref{L:recursivity-of-Js} $J_*$ is $\Si_1^0(\ep)$-recursive. Using this, Lemmas~\ref{L:recursivity-of-Q}, \ref{L:volba-epsilon},
and \ref{L:bounded-quantifiers}, we get that $\bbJ$ is $\Si_1^0(\ep)$-recursive as well.
\end{proof}

\begin{definition**}\label{D:pripustnaA}
The set $\bbA \subset [0,\alpha] \times \cQ \times \cN$ contains all triples $(\beta,p,x)$ such that either $\beta = \alpha$ or
$\beta < \alpha$, $|\tau(p)| = \ell(\be)$, $(p,x) \in \bbJ$, and $\xi(p) \in S_{\beta}^x$. See Notation~\ref{N:diamant-el-a} for the definition of the mapping $\ell$.
\end{definition**}

\begin{lemma**}\label{L:recursivity-of-A}
The set $\bbA$ is $\Si^0_1(\ep)$-recursive.
\end{lemma**}

\begin{proof}
The $\Si^0_1(\ep)$-recursivity of $\bbA$ follows from $\Si^0_1(\ep)$-recursivity of the mappings $s\in\cS \mapsto |s|$ and $\ell\colon[0,\alpha] \to \omega$ (see Lemma~\ref{L:volba-epsilon}(e),(g)), $\Si^0_1(\ep)$-recursivity of the mappings $\tau$ and $\xi$ (see Lemma~\ref{L:recursivity-of-Q}(b)),
$\Si^0_1(\ep)$-recursivity of $S$ by Proposition~\ref{P:jump-trees-old}, and $\Si^0_1(\ep)$-recursivity of $\bbJ$ by Lemma~\ref{L:recursivity-of-F}.
\end{proof}

\subsection{Construction of \texorpdfstring{$P^x_{\be}$}{}}\label{SS:construction-of-P}

We are going to construct the crucial objects $P^x_{\be}$ for $x \in \cN$ and $\be\in [0,\al]$ with the properties summarized in Proposition~\ref{P:object-P}.
To this end we need a particular relation $\R$ on $\cQ$.

\begin{definition**}\label{D:definition-of-R}
Let $a,b \in \cQ$. Then we set $a \R b$ if there exists $m \in \omega$ such that we have either
\begin{itemize}
\item[(1)] $b = a \w (m,-)$, or

\item[(2)] $b = a \w (m,\emptyset)$, or

\item[(3)] $b = a|(|b|-1) \w \bigl(\eta_{r(b)-1}(a) \w  n_{r(b)}(a) \w \eta_{r(b)}(a) \w \dots \w \eta_{r(a)-1}(a) \w m\bigr)$ and moreover  $\eta_{r(b)-1}(a) = -$.
\end{itemize}
\end{definition**}

\begin{remark**}\label{R:properties-of-R}
(a) Observe that if $a \R b$ then $\widehat{b} = \widehat{a} \w m$ for some $m \in \omega$, in particular, $|\widehat{b}| = |\widehat{a}|+1$.

\medskip\noindent
(b) Assume that $a,b \in \cQ$ and that (3) from Definition~\ref{D:definition-of-R} is satisfied for some $m \in \om$.
The informal scheme at Figure~\ref{Fi:relation-R} can help to understand the relation between $a$ and $b$.
Rectangles represent elements of $\Seqm$, in particular, squares stand for sequences of length $1$.

\begin{figure}
\centering\includegraphics[scale=0.8,trim={200 0 30 590},clip]{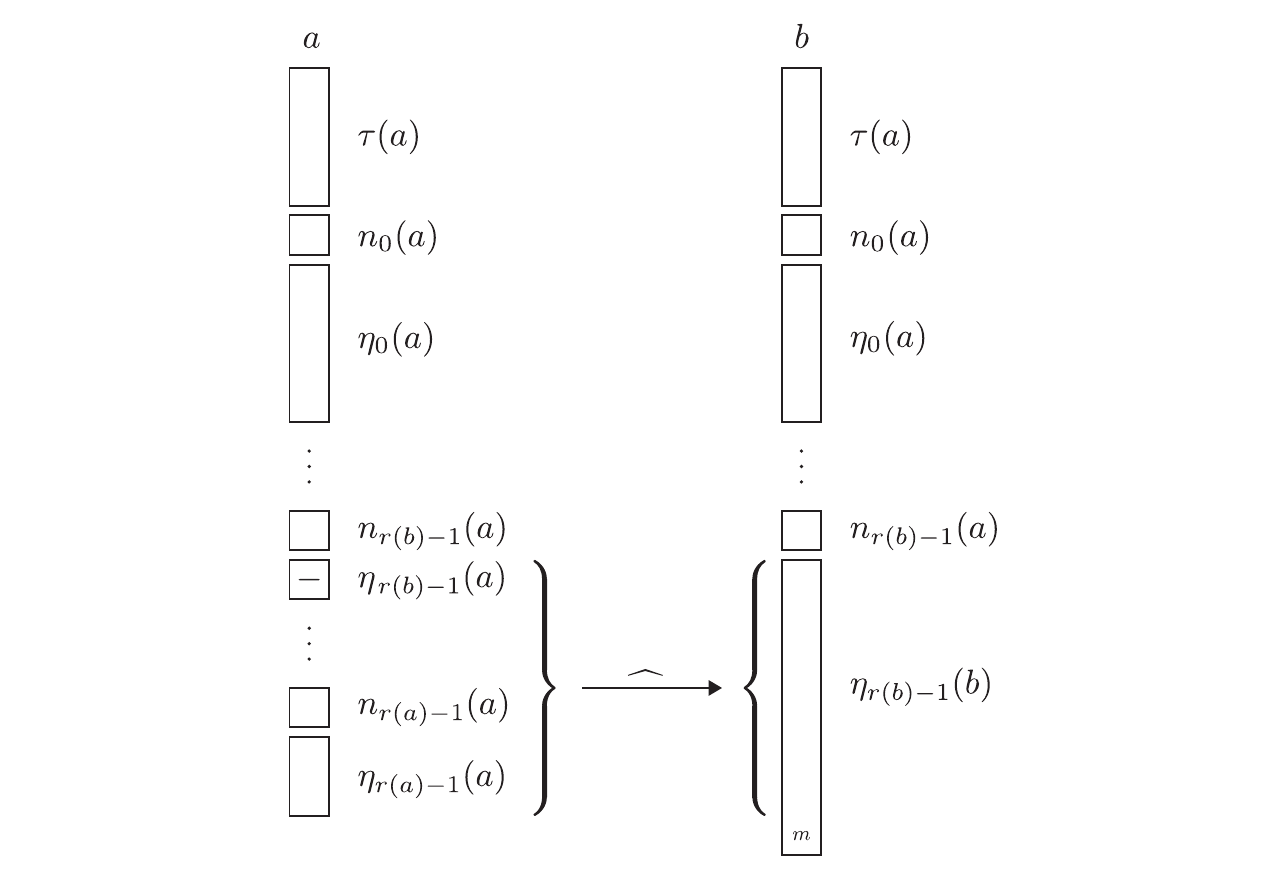}
\caption{}\label{Fi:relation-R}
\end{figure}

\medskip\noindent
(c) We point out that for every $q\in\bbJ^x$ and $m\in\om$ there is a  unique $p\in\bbJ^x$ such that $q\R p$ and $\wh{p}=\wh{q}\w(m)$.
Indeed, if there is no $j \le r(q)$ with $(j,\wh{q} \w m,x)\in J_*$, we have to define $p = q \w (m,-)$.
If $i = r(q)$ is the only $j \le r(q)$ with $(j,\wh{q}\w m,x)\in J_*$, we have to define $p = q \w (m,\emptyset)$.
If $i < r(q)$ is the smallest $j\le r(q)$ with $(j,\wh{q}\w m,x)\in J_*$, we have to define $p = q|(2i+2) \w (\eta)$, where $\eta = q(2i+2) \w \dots \w q(|q|-1) \w m$.
This operation is implicitly used in Harrington's construction later.
\end{remark**}

\begin{lemma**}\label{L:recursivity-of-R}
The relation $\R$ is $\Si^0_1(\ep)$-recursive in $\cP \times \cP$, i.e., the set $\{(a,b) \in  \cQ \times \cQ \set a \R b\}$ is $\Si^0_1(\ep)$-recursive in $\cP \times \cP$.
\end{lemma**}

\begin{proof}
It is sufficient to use Lemma~\ref{L:recursivity-of-Q}, Lemma~\ref{L:volba-epsilon}(a), and $\Si^0_1(\ep)$-recursivity of the mappings
\[
(i,a,b) \mapsto \eta_i(a) \quad \text{ and } \quad (i,a,b) \mapsto \eta_{r(b)-1}(b)| (|\eta_{r(b)-1}(b)|-1)
\]
on $\{(i,a,b) \in \omega \times \cQ^2 \set r(b)-1\le i\le r(a)-1\}$ which follows from Lemma~\ref{L:volba-epsilon}(e).
\end{proof}

\begin{remark**}[basic idea of the construction]\label{R:basic-idea}
Before we start the constructions leading to the definitions of $F$ and $\Xi^x\colon E_x \to F_x$, $x\in\cN$, with the properties stated in Proposition~\ref{P:hPsi-properties},
we point out roughly their main ideas.
We are going to define
\begin{itemize}
\item $T \su [0,\al] \times \cS \times \cN$ (in Subsection~\ref{SS:Tcka-a-uzavrenostF}) and

\item $\hPsi_{\ga,\be}^x \colon [T^x_{\be}] \to [T^x_{\ga}]$ for $x \in \cN$, $0 \le \ga<\be\le\al$ (in Subsection~\ref{SS:hPsi-beta-gamma})
\end{itemize}
such that $F_x = [T^x_0]$ and $\Xi^x = \hPsi^x_{0,\al}$, $x\in\cN$, satisfy the required conditions.

First we construct sets $P_{\be}^x \subset \bbA^x_{\be}$, $x \in \cN, \be \in [0,\al]$, cf. Proposition~\ref{P:object-P}.
Thus the sets $P_{\beta}^x$, $x\in\cN$, will be families of sequences of sequences.
One can also view the elements of each $P_{\be}^x$ as \emph{partitioned} sequences of natural numbers.
Applying concatenation to each element of $P_{\beta}^x$, we get the desired tree $T_{\beta}^x$.

The construction of $P_{\be}^x$ can be considered as an inductive modification of the elements of $T_{\al}^x$.
The definition of the mappings $\hPsi_{\ga,\be}^x$ knowing $P^x_\be$ will need an extra effort.
We start by defining $P_{\al}^x = \{p\in\cP\set \tau(p)\in T_{\al}^x,|p|=1\}$ (cf. \ref{D:function-f}(I)).
Then we successively decide whether $p \in \bbA^x_{\be}$ belongs to $P^x_\be$ or not for $\be<\al$.
We assume here that, when considering $p \in \bbA^x_{\be}$, we already know the corresponding decisions required
by the following two principles.
The first principle (``copying'') determines which sequences of the length $1$ will be in $P^x_{\be}$ and which not.
The second principle (``modification'') does the same job for sequences of the length greater than $1$.
The first principle requires a decision about a sequence from $P^x_{\diamond(\be)}$ and the second one requires decision about
some sequences from $P^x_{\be}$ and $P^x_{\be+1}$.
In fact, the formal application of these principles is realized using recursion theorem.

Let us describe these principles a little bit closer.
The elements of $P_{\be}^x$ of length $1$ are constructed by copying and concatenating of already constructed elements of $P_{\diamond(\be)}^x$.
More precisely, using the mappings $\diamond$ and $\ell$, we put in $P_{\beta}^x$ exactly those elements $p$ of length $1$ which satisfy $p = (\wh {q_{\diamond}})$, where $q_{\diamond} \in P_{\diamond(\be)}^x$ and $|\wh q_{\diamond}| = \ell(\beta)$, cf. \ref{D:function-f}($\diamond$) and Figure~\ref{Fi:length-1}.

\begin{figure}
\centering\includegraphics[scale=0.8,trim={0 70 300 650},clip]{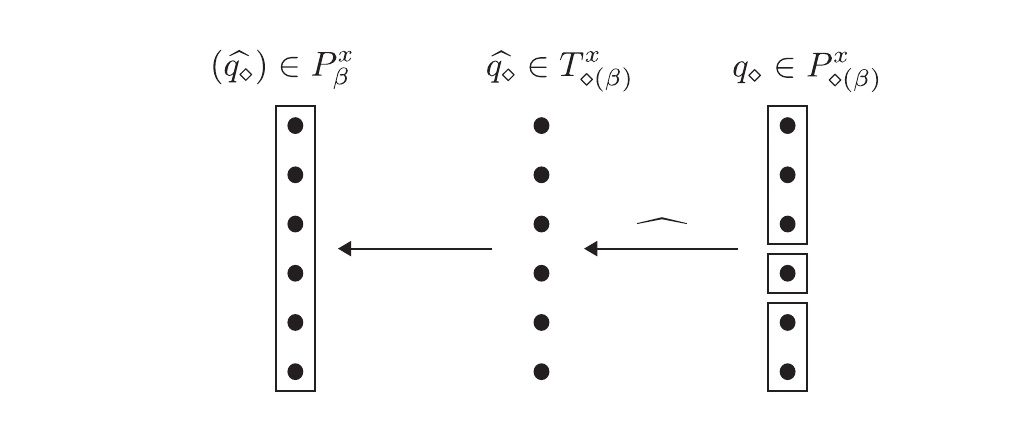}
\caption{}\label{Fi:length-1}
\end{figure}

The second principle of the construction consists of several steps.
We put to $P^x_\be$ exactly those $p \in \bbA_{\be}^x$ of length greater than
one which satisfy the following conditions.
There are $t \in T_{\be+1}^x$, i.e., $t = \wh{q_E}$ for some $q_E \in P^x_{\be+1}$, and $q_R \in P^x_\be$
such that
\[
\zeta(p) = t, \quad \xi(p) \in S^x_\be, \quad \text{ and } \quad q_R \R p,
\]
cf. \ref{D:function-f}(E),(R) and Remark~\ref{R:properties-of-R}(c).
Moreover, the following ``minimality condition'', cf. \ref{D:function-f}(M),(M'), has to be satisfied.
We need that the decision about $q\in P^x_\be$ has been already done for the elements of the set
\[
\begin{split}
M(p)=\{q\in\bbA^x_\be\set  \exists i<r(p)\colon |q|=2i+3, &\  q|(2i+2) = p|(2i+2), \\
&\eta_i(q)\ne -, \xi(q) \preceq \xi(p), q <_\cP p|(2i+3)\}
\end{split}
\]
and for every $q \in M(p)$ we have $q\notin P^x_\be$.

At Figure~\ref{Fi:principle-2} we illustrate how candidates for $p\in P^x_\be$ are constructed by (E) and (R) (ignoring the minimality condition (M)) from $q_E \in P^x_{\be+1}$ and $q_R\in T^x_\be$.
Each circle represents a natural number. Certain items are represented as pairs using the identification
$n = ([n]_0,[n]_1)$.
Rectangles represent elements of $\Seqm$.
The left column represents the sequence $p$.
Black circles used in the middle column and in the left column correspond to a sequence of natural numbers which is the same as the sequence $t = \wh{q_E}$
represented by the black circles in the right column. The middle column represents the sequence $q_R$ extended by the sequence $(m)$, where $m \in \omega$ is any integer such that
$[m]_0 = t(|t|-1)$ and $\xi(q_R) \w [m]_1 \in S^x_{\be}$. The sequence  $\xi(q_R) \w [m]_1 $ is represented by the gray column.
The left column shows the three possibilities how
sequences for given $q_R$ and $m$ to be considered by the minimality conditions
can be constructed: either $q_R$ is extended by $(m,-)$ or by $(m,\emptyset)$, or the sequence $(m)$ is added and certain sequences are concatenated, cf. Definition~\ref{D:definition-of-R}.

\begin{figure}
\centering\includegraphics[scale=0.8,trim={160 220 0 180},clip]{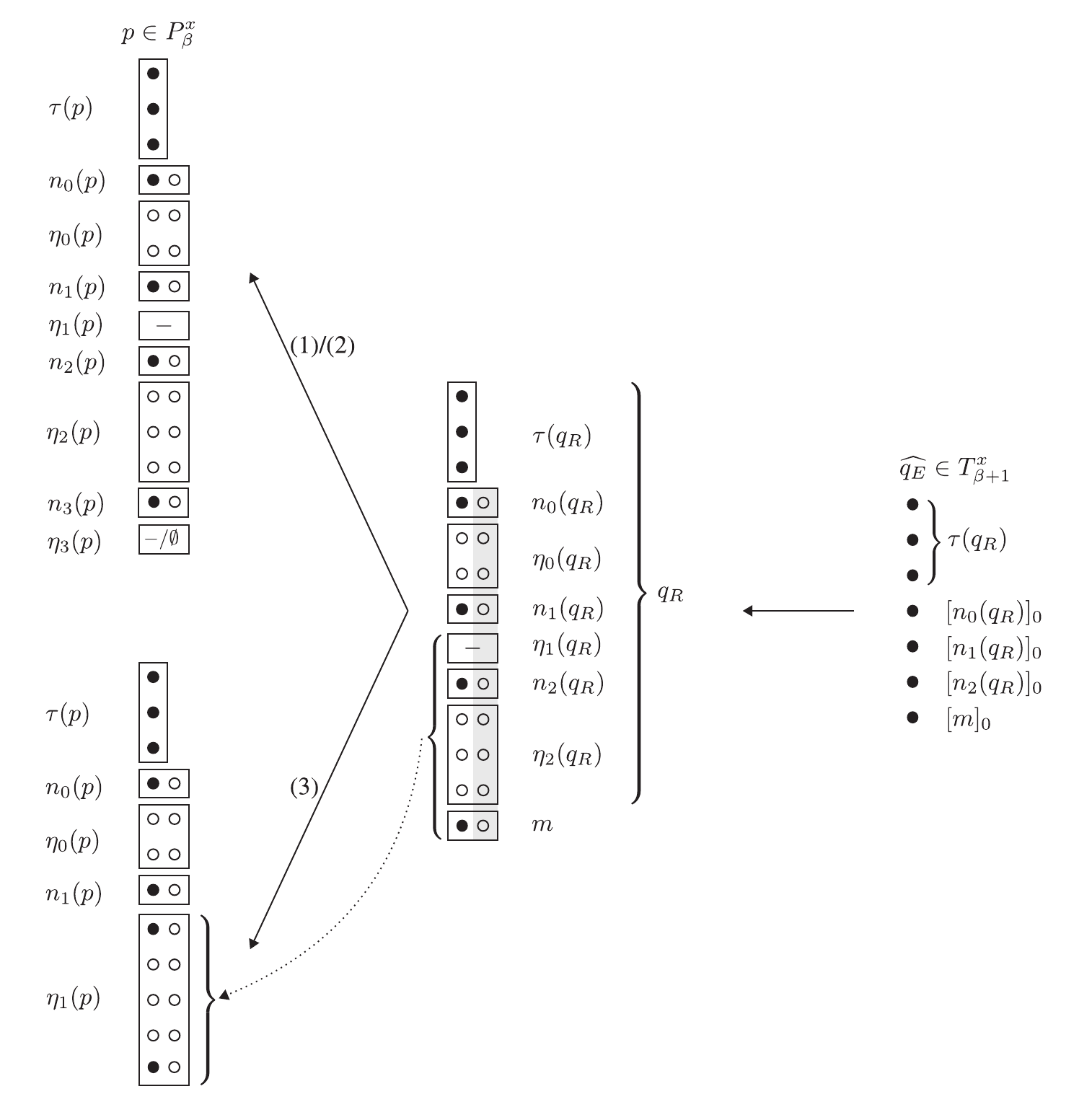}
\caption{}\label{Fi:principle-2}
\end{figure}

The construction of $P_{\be}^x$'s is, in fact, done uniformly, so we construct a $\Si_1^0(\ep)$-recursive
set $P \subset [0,\al] \times \cP \times \cN$.
Using $\Si_1^0(\ep)$-recursiveness of $P$, we get that the set $F=[T_0]$ is closed, cf. Proposition~\ref{P:uzavrenost-F}.

Now we indicate how the mappings $\Psi^x_{\be,\be+1}$ can be derived from $P$.
Suppose that $u \in [T^x_{\be+1}]$.
Due to the choice of sequences of length $1$ in $P^x_\be$, we have $(u|\ell(\beta)) \in P_{\be}^x$ and there are also $p \in P_{\be}^x$ of the form $(u|\ell(\beta), u(\ell(\be)), \eta)$ constructed as above.
Using (M) it is not difficult to observe that there exists a unique $\wt \eta$ such that $\wt p = (u|\ell(\beta), u(\ell(\be)), \wt \eta)$ has at least one infinite extension in $[P_{\be}^x]$. For the meaning of the symbol $[P_{\be}^x]$, see Definition~\ref{N:tree-like}(c). In fact, $\wt p$ is the $<_{\cP}$-minimal element among the mentioned $p$'s.
After finding the unique $\eta_0$ ($= \wt \eta$) corresponding to $u$, the same reasoning gives $\eta_1$, then $\eta_2$, and so on.
Then we define $\Psi^x_{\be,\be+1}(u)$ by
\[
\xi\bigl(\Psi^x_{\be,\be+1}(u)\bigr) = \sigma_{\be}^x, \qquad
\zeta\bigl(\Psi^x_{\be,\be+1}(u)\bigr) = u, \qquad
\eta_i\bigl(\Psi^x_{\be,\be+1}(u)\bigr) = \eta_i.
\]

Applying concatenation we define $\hPsi^x_{\beta,\beta+1}\colon [T_{\beta+1}^x] \to [T_{\beta}^x]$ by
$\hPsi^x_{\beta,\beta+1}(y) = \wh{\Psi^x_{\beta,\beta+1}(y)}$, cf. Definition~\ref{D:hPsi-pro-nasledniky}(2).

The crucial property is that the partitioned sequence $\Psi^x_{\be,\be+1}(u)$
can be $\Si_1^0(\ep)$-recursively computed from $u \in [T_{\be+1}^x]$. Since each element $v \in [P_{\be+1}^x]$ has its $\xi$-part $\xi(v) = \si^x_{\be+1}$,
the jump $x^{(\be+1)}$ can be $\Si_1^0(\ep)$-recursively computed from $u \in [T_{\be+1}^x]$.
The jump $x^{(\be+1)}$ is trivially used to $\Si_1^0(\ep)$-recursively compute the $\xi$-part of $\Psi_{\be,\be+1}^x(u)$
but also the remaining items of $\Psi_{\be,\be+1}^x(u)$ can be $\Si_1^0(\ep)$-recursively computed, cf. Lemma~\ref{L:properties-of-theta}.

We are going to show that $\hPsi^x_{\be,\be+1}(u)'_x(i)=1$ is equivalent to $\eta_i(\Psi^x_{\be,\be+1}(u))\ne -$.
Since $p=\Psi^x_{\be,\be+1}(u)|(2i+3)\in \bbA^x_{\be}$ ($\subset \bbJ^x$), the relation $\eta_i(p) \neq -$ is equivalent to that $z'_x(i) = 1$ for every infinite sequence $z$ extending $\wh{p}$. In particular, $\eta_i(p) \neq -$ implies that $\wh{\Psi}^x_{\be,\be+1}(u)'_x(i)=1$.
Assume now that $\hPsi^x_{\be,\be+1}(u)'_x(i)=1$. Using Remark~\ref{R:property-of-J}(d), for sufficiently large $I\in\om$ we have $\wh{q}\in (J_*)^x_i$, where $q=\Psi^x_{\be,\be+1}(u)|(2I+3)\in P^x_\be\su\bbJ^x$. Thus $\eta_i(\Psi^x_{\be,\be+1}(u))= \eta_i(q)=\eta_i(p)\ne -$.

This fact shows that the jump $\hPsi^x_{\beta,\beta+1}(u)'_x$, where $u \in [T_{\be+1}^x]$, is $\Si_1^0(\ep)$-recursively computable from $\Psi^x_{\beta,\beta+1}(u)$, since one can simply look at the element $\eta_i(\Psi^x_{\beta,\beta+1}(u))$ to evaluate $\hPsi^x_{\beta,\beta+1}(u)'_x(i)$.
Therefore also $\hPsi^x_{\beta,\beta+1}(u)'_x$ is $\Si_1^0(\ep)$-recursively computable from $u$.

Finally, using these mappings we build all the mappings $\wh{\Psi}^x_{\be,\ga}$ and in particular the mapping $\Xi^x=\hPsi_{0,\alpha}^x$ having the property that
$\Xi^x(y)^{(\al)}_x = \bigl(\hPsi^x_{0,\alpha}(y)\bigr)^{(\alpha)}_x$ can be $\Si_1^0(\ep)$-recursively computed from $y \in [T_{\alpha}^x]$.
\end{remark**}

Let us recall that we introduced the well ordering $<_{\cP}$ in Notation \ref{N:usporadanicP} and the mapping $\diamond$ in Lemma~\ref{L:diamond}.

\begin{proposition**}\label{P:object-P}
There is a $\Si^0_1(\ep)$-recursive set $P\su\bbA$ such that $(\be,p,x)\in \bbA$ is in $P$, i.e., $p\in P^x_\be$, if and only if one of the following three statements is satisfied:
\begin{enumerate}[\upshape (a)]
  \item $\be=\al$, $|p| = 1$, and $\wh{p} \in T_{\al}^x$,

  \item $\be<\al$, $|p|=1$, and there exists $q_{\diamond} \in \cQ$ such that $q_{\diamond}\in P^x_{\diamond(\be)}$ and $\wh{p} = \wh{q_{\diamond}}$,

  \item $\be < \al$, $|p| > 1$, and the following conditions are satisfied:
  \begin{itemize}
    \item[(i)]  $\exists q_E \in P^x_{\be+1} \colon \zeta(p) = \widehat{q}_E$,
    \item[(ii)] $\exists q_R \in P^x_\be \colon q_R \fR p$,
    \item[(iii)]
    \[
    \begin{split}
    \forall q \leq_{\cP} p,  &\, q \in P^x_\be, \xi(q) \preceq \xi(p) \ \forall i < r(p), \eta_i(q) \ne -, q|(2i+2) = p|(2i+2) \colon \\
    &q|(2i+3) \leq_{\cP} p|(2i+3)  \Rightarrow \eta_i(q) = \eta_i(p).
    \end{split}
    \]
\end{itemize}
\end{enumerate}
\end{proposition**}

To define $P$ from Proposition~\ref{P:object-P} we will rather look for the characteristic function of $P$ applying the recursion theorem
to the function $f$ which we define now.

\begin{notation**}
Further on in this subsection we use the notation $X = [0,\alpha] \times \cP \times \cN$
and $U = U^{X,\om}_{\Si_1^0(\ep)}$ (see Remark~\ref{R:extended-facts}).
In this subsection, let $T_\al$ be a set with properties (i)--(iii) from Proposition~\ref{P:hPsi-properties}.
Notice that $T_\al^{LN}$ defined in Proposition~\ref{P:E^LN} and $T_\al^{SR}$ in Proposition~\ref{P:E^SR} are particular cases of such $T_{\alpha}$'s.
These propositions have been already proved in Section~\ref{S:construction-of-T-alpha}.
\end{notation**}

\begin{definition**}\label{D:function-f}
We define a partial function $f \colon \om \times \bbA \rightharpoonup \{0,1\} \su \om$ as follows. Let $e\in\om$ be arbitrary.

\medskip\noindent
(a) For $\be=\al$ and $(\be,p,x)\in\bbA$ we define $f(e,\be,p,x)=1$ if
\begin{equation*}
|p|=1 \land \wh{p} \in T^x_\al, \tag{I}
\end{equation*}
and we set $f(e,\be,p,x)=0$ if
\begin{equation*}
|p| \ne 1  \vee  \wh{p} \notin T^x_\al. \tag{I'}
\end{equation*}

\medskip\noindent
(b) For $\be<\al$, $(\be,p,x)\in\bbA$, and $|p|=1$ we define $f(e,\be,p,x)=1$ if
\begin{equation*}
\exists q_{\diamond} \in \cQ \colon (\diamond(\be),q_{\diamond},x)\in\bbA \land \wh{p} = \wh{q_{\diamond}} \land U(e,\diamond(\be),q_{\diamond},x)=1, \tag{$\diamond$}
\end{equation*}
and we set $f(e,\be,p,x) = 0$ if
\begin{equation*}
\forall q \in \cQ, (\diamond(\be),q,x)\in\bbA, \wh{p} = \wh{q} \colon  U(e,\diamond(\be),q,x) = 0. \tag{$\diamond$'}
\end{equation*}

\medskip\noindent
(c) For $\be < \al$, $(\be,p,x)\in\bbA$, and $|p|>1$ we define $f(e,\be,p,x)=1$ if the following conditions are satisfied
\begin{align*}
\exists q_E \in \cQ, &(\be+1,q_E,x) \in \bbA, \zeta(p) = \wh{q}_E \colon U(e,\be+1,q_E,x) = 1,  \tag{E} \\
\exists q_R \in \cQ, &(\be,q_R,x)   \in \bbA, q_R \fR p           \colon U(e,\be,q_R,x)   = 1,  \tag{R} \\
\begin{split}
\forall q <_{\cP} p, &(\be,q,x) \in \bbA,  \xi(q)\preceq\xi(p) \ \forall i < r(p), \eta_i(q)\ne -, q|(2i+2)=p|(2i+2) \colon \\
&q|(2i+3) <_{\cP} p|(2i+3) \Rightarrow U(e,\be,q,x)=0,
\end{split}
\tag{M}
\end{align*}
and we set $f(e,\be,p,x)=0$ if at least one of the following conditions is satisfied
\begin{align*}
\forall q \in \cQ, &(\be+1,q,x) \in \bbA, \zeta(p) = \wh{q} \colon U(e,\be+1,q,x) = 0,   \tag{E'} \\
\forall q \in \cQ, &(\be,q,x)   \in \bbA, q \fR p           \colon U(e,\be,q,x)   = 0,   \tag{R'} \\
\begin{split}
\exists q <_\cP p, &(\be,q,x) \in \bbA, \xi(q)\preceq\xi(p) \ \exists i < r(p), \eta_i(q)\ne -, q|(2i+2)=p|(2i+2) \colon \\
&q|(2i+3) <_{\cP} p|(2i+3) \land U(e,\be,q,x) = 1.
\end{split}\tag{M'}
\end{align*}
\end{definition**}

Let us remark that the conditions ($\diamond$) and ($\diamond$') define disjoint subsets of $\omega \times \bbA$ which need not cover $\om \times \bbA$. The same is true for the pairs of conditions (E) and (E'), (R) and (R'), (M) and (M').
The pair (I) and (I') defines a decomposition of $\omega \times \bbA$. Thus the partial function $f$ is well defined.
Now we verify the assumption of the recursion theorem (see Fact~\ref{F:recursionthms}(b) and Remark~\ref{R:extended-facts}).

\begin{lemma**}\label{L:recursivity-of-f}
The function $f$ is a partial $\Sigma_1^0(\ep)$-recursive function on its domain.
\end{lemma**}

\begin{proof}
Let $\De^U \su X \times \omega$ be a $\Si^0_1(\ep)$ set which computes the function $U$ on its domain,
see Lemma~\ref{L:natural-base}(b).
We define a set $\De^f\su \om \times \bbA \times \omega$ which computes $f$ on its domain as follows.
Suppose that $(e,\beta,p,x,i) \in \omega \times \bbA \times \omega$.
We modify conditions in (a)--(c) from Definition \ref{D:function-f} to (a)$_m$--(c)$_m$ by replacing conditions of the form ``$U(e,\gamma,q,x) = i$'' with ``$(e,\gamma,q,x,i) \in \Delta^U$''.
More precisely, ``$U(e,\diamond(\be),q_{\diamond},x)=1$'' is replaced by ``$(e,\diamond(\be),q_{\diamond},x,1) \in \Delta^U$'' in ($\diamond$),
``$U(e,\diamond(\be),q,x)=0$'' is replaced by ``$(e,\diamond(\be),q,x,0) \in \Delta^U$'' in ($\diamond$'), and so on.
We denote the corresponding modified conditions by adding the index $m$, i.e., by  ($\diamond$)$_m$, ($\diamond$')$_m$, $\text{(E)}_m$, $\text{(E')}_m$, and so on ((I) and (I') are not changed by this replacement).

Then we define $\Delta^f$ as the set of all $(e,\beta,p,x,i) \in \omega \times \bbA \times \omega$ for which
either $i = 1$ and one of the following conditions holds:
\begin{itemize}
\item $\beta = \alpha$, $(\beta,p,x) \in \bbA$, and (I) holds, or

\item $\beta < \alpha$, $(\beta,p,x) \in \bbA$, $|p|=1$, and ($\diamond$)$_m$ holds, or

\item $\beta < \alpha$, $(\beta,p,x) \in \bbA$, $|p|>1$, and $\text{(E)}_m \land \ \text{(R)}_m \land \text{(M)}_m$ holds,
\end{itemize}

or $i = 0$ and one of the follwing conditions holds:

\begin{itemize}
\item $\beta = \alpha$, $(\beta,p,x) \in \bbA$, and (I') holds, or

\item $\beta < \alpha$, $(\beta,p,x) \in \bbA$, $|p|=1$, and ($\diamond$')$_m$ holds, or

\item $\beta < \alpha$, $(\beta,p,x) \in \bbA$, $|p|>1$, and $\text{(E')}_m \vee \ \text{(R')}_m \vee \text{(M')}_m$ holds.
\end{itemize}

To verify that the set $\Delta^f$ computes $f$,
it is sufficient to observe that for $(e,\be,p,x)\in D(f)$ the conditions
$(e,\be,p,x,1) \in \De^f$ and $(e,\be,p,x,0) \in \De^f$ are mutually exclusive and
that $f(e,\be,p,x) = i$ implies $(e,\be,p,x,i) \in \Delta^f$.

It remains to show that $\De^f$ is a $\Sigma_1^0(\ep)$ set. It is sufficient to prove that the modified conditions define $\Si^0_1(\ep)$ sets
since the relation $\be=\al$, $\be<\al$, $|p|=1$, and $|p|>1$ define $\Si_1^0(\varepsilon)$ sets.

\medskip\noindent
(a)$_m$ Using $\Si^0_1(\ep)$-recursiveness of $\bbA$ (Lemma~\ref{L:recursivity-of-A}) and of $T_{\al}$,
and Lemma~\ref{L:volba-epsilon}(e), we see that (I) and (I') define $\Si^0_1(\ep)$ sets. Thus the case (a)$_m$ = (a) is verified.

\medskip\noindent
(b)$_m$ The set defined by ($\diamond$)$_m$ is $\Si^0_1(\ep)$-recursive by Lemma~\ref{L:volba-epsilon}(e) and by Fact~\ref{F:properties-of-Gamma} with Remark~\ref{R:extended-facts}.
Using Lemma~\ref{L:volba-epsilon}(a),(e),(g), the set
\[
V := \bigl\{(e,(\be,p,x),i,q) \in \omega \times X \times \omega \times \cP \set (\diamond(\be),q,x) \in \bbA, \
\wh{p} = \wh{q}, \ (e,\diamond(\be),q,i) \in \Delta^U, \ i = 0 \bigr\}
\]
is also in $\Sigma_1^0(\ep)$.
We consider the $\Si^0_1(\ep)$-recursive function $M_\cP$ from Remark~\ref{R:upper-bounds}.
If $\wh{q} = \wh{p} =: s$, $q \in \cQ$, and $p \in \cP$, we have that $|q| \le 2|\wh{q}|+1$ and so $q \le_\cP M_\cP(\wh{p})$ by Remark~\ref{R:upper-bounds}.
Then applying Lemma~\ref{L:bounded-quantifiers} for $V$ and $g(p) = M_\cP(\wh{p})$ we obtain that the condition ($\diamond$')$_m$ defines  also a $\Sigma_1^0(\ep)$ set.

\medskip\noindent
(c)$_m$  The property (E)$_m$ defines $\Si^0_1(\ep)$ sets by Lemma~\ref{L:volba-epsilon}(a),(g),(e),
Fact~\ref{F:properties-of-Gamma}, and Lemmas~\ref{L:recursivity-of-Q}, \ref{L:recursivity-of-A}.
The property (R)$_m$ defines $\Si^0_1(\ep)$ sets by Lemma~\ref{L:volba-epsilon}(a), Fact~\ref{F:properties-of-Gamma}, and Lemmas~\ref{L:recursivity-of-Q}, \ref{L:recursivity-of-A}, \ref{L:recursivity-of-R}.
The property (M')$_m$ defines $\Si^0_1(\ep)$ sets by Lemma~\ref{L:volba-epsilon}(a),(d),(e), Remark~\ref{R:recursivity-of-cP},
Fact~\ref{F:properties-of-Gamma}, and Lemmas~\ref{L:recursivity-of-Q}, \ref{L:recursivity-of-A}.

The remaining properties (E')$_m$, (R')$_m$, and (M)$_m$ lead to $\Si^0_1(\ep)$ sets by the arguments used for ($\diamond$)$_m$.
We only point out the bounds used in the applications of Lemma~\ref{L:bounded-quantifiers} in these three remaining cases.
In (M)$_m$ the quantifiers $\forall q <_\cP p$ and $\forall i<r(p)$ satisfy the assumptions of Lemma~\ref{L:bounded-quantifiers}, see Remark~\ref{R:upper-bounds}.
In (E')$_m$ we use the fact that $\wh{q}=\zeta(p)$ implies $q\le_\cP M_\cP(\zeta(p))$ and $\Si^0_1(\ep)$-recursivity of the functions $M_\cP$ and $\zeta$.
Here we may use Remark~\ref{R:upper-bounds} since $\wh{q}=\zeta(p)=:s$ gives $\wh{q}\preceq s$ and $|q|\le 2|\wh{q}|+1$ for $q\in\cQ$ as above.
In (R')$_m$ we may use Remark~\ref{R:upper-bounds} since $q\fR p$ for $p,q\in\cP$ implies that $\wh{q}\preceq\wh{p}=:s$ and we have that $|q|\le 2|\wh{q}|+1$ for $q\in\cQ$. Therefore $q\le_{\cP} M_\cP(\wh{p})$.
\end{proof}

\begin{definition**}
Using recursion theorem (Fact~\ref{F:recursionthms}(b) and Remark~\ref{R:extended-facts}), we find $e^* \in \omega$ such that for all $(\beta,p,x) \in X$ we have
$f(e^*,\beta,p,x) = U(e^*,\beta,p,x)$ whenever $f(e^*,\beta,p,x)$ is defined.
Thus we have defined the function $f_{e^*}\colon (\beta,p,x) \mapsto f(e^*,\beta,p,x)$, cf. Notation~\ref{N:sections}.
\end{definition**}

\begin{lemma**}\label{L:Precursive}
The function $f_{e^*}$ is $\Sigma_1^0(\ep)$-recursive on $\bbA$.
\end{lemma**}

\begin{proof}
We know from the recursion theorem that $f_{e^*}$ is $\Si_1^0(\ep)$-recursive on its domain.
Therefore it is sufficient to prove that $f_{e^*}(\be,q,x)$ is defined for every $(\be,q,x)\in\bbA$.
Fix $x \in \cN$. Towards contradiction assume that the set
\[
N := \{(\be,q) \in [0,\alpha] \times \cP \set (\be,q,x) \in \bbA \land f_{e^*}(\be,q,x) \text{ is not defined}\}
\]
is nonempty.

Let $L$ be the smallest integer such that there is $(\be,q) \in N$ with $|\wh{q}| = L$.
Since $(\alpha,q) \notin N$ for every $q \in \bbA^x_{\alpha}$ we may define $\be^* \in (0,\al]$ as the smallest ordinal such that $f_{e^*}(\be,q,x)$ is defined for
all $(\be,q,x) \in \bbA$ such that $\be \ge \be^*$ and $|\wh{q}| \le L$.

If $\be^* = \be+1$ for some $\beta$ then there is $p \in \cP$ such that $(\be,p) \in N$ and $|\wh{p}|=L$. We may assume that $p$ is moreover the smallest one with respect to $<_{\cP}$ with these properties.
It is not difficult to check that all needed evaluations of $U_{e^*}$ in the definition of $f(e^*,\be,p,x)$ are defined since we need them only for triples $(\be',p',x)\in\bbA$ such that $|\wh{p'}|\le L$ and either $|\wh{p'}|< L$, or $\be'\ge\be^*$, or ($|\wh{p'}|=L$ and $\be'=\be$ and $p'<_{\cP} p$). Thus the value $f_{e^*}(\beta,p,x)$ is defined, a contradiction with $(\be,p)\in N$.

If $\be^*$ is a limit ordinal then by Lemma~\ref{L:diamond} there is $k\in\om$ such that $\ell(\be) > L$ for every $\be\in [a(\be^*,k),\be^*)$. Denote $\be_0 = a(\be^*,k)$.
Now for every $\beta \in [\beta_0,\beta^*)$ and $p \in \cP$ with $|\wh{p}| \leq L$ we have $(\be,p,x) \notin \bbA$ by definition of $\bbA$ (Definition~\ref{D:pripustnaA}).
So $\beta^* \leq \beta_0 < \beta^*$, a contradiction.
\end{proof}

\begin{definition**}\label{D:object-P}
We set $P = \{(\be,p,x) \in \bbA \set f_{e^*}(\beta,p,x) = 1\}$. Let $\be \in [0,\al]$ and $x \in \cN$.
We often use the notation $P^x_\be$ (see Notation~\ref{N:sections}).
\end{definition**}

\begin{proof}[{\bf Proof of Proposition~\ref{P:object-P}}]
The set $P \subset \bbA$ is $\Si^0_1(\ep)$-recursive due to Lemmas~\ref{L:recursivity-of-A} and \ref{L:Precursive}.
By Lemma~\ref{L:Precursive} we know that for each symbol $F \in \{I,\diamond, E, R, M\}$ the corresponding formulae (F) and (F') are complementary in $\bbA$.
Let $(\beta,p,x) \in \bbA$. If $\beta = \alpha$ then we may easily check that the condition (I) from Definition~\ref{D:function-f}(a) is satisfied for $(\be,p,x)$ if and only if
$(\beta,p,x)$ satisfies (a) in Proposition~\ref{P:object-P}. The case $\beta < \alpha$ and $|p|=1$ can be treated similarly.

Assume that $(\beta, p, x) \in P$, $\beta < \alpha$, and $|p| > 1$.
Then $(\be, p, x)$ clearly satisfies (c-i) and (c-ii). To show (c-iii) take $q$ and $i$ such that $q \leq_{\cP} p,  q \in P^x_\be, \xi(q) \preceq \xi(p)$, $i < r(p)$, $\eta_i(q) \ne -$, $q|(2i+2) = p|(2i+2)$, and $q|(2i+3) \leq_{\cP} p|(2i+3)$.
If $p = q$ or just $p|(2i+3) = q|(2i+3)$, then clearly $\eta_i(q) = \eta_i(p)$. If $q <_{\cP} p$ and  $p|(2i+3)  <_{\cP} q|(2i+3)$,
then (M) gives $q \notin P_{\be}^x$, a contradiction.

Now assume that $(\beta, p, x) \in \bbA$ satisfies (c) in Proposition~\ref{P:object-P}. Then we have either $f_{e^*}(\be,p,x) = 1$
or  $f_{e^*}(\be,p,x) = 0$ by Lemma~\ref{L:Precursive}. Towards contradiction assume that the latter case occurs.
Since $(\beta, p, x) \in \bbA$ satisfies (c-i) and (c-ii), we get that $(e^*,\beta, p, x)$ satisfies neither (E') nor (R').
Thus it satisfies (M'). This condition gives  $q$ and $i$ such that $q <_{\cP} p, \xi(q) \preceq \xi(p)$, $i < r(p)$, $\eta_i(q) \ne -$, $q|(2i+2) = p|(2i+2)$, $q|(2i+3) <_{\cP} p|(2i+3)$, and $q \in P_{\be}^x$. This implies $\eta_i(q) \neq \eta_i(p)$,  a contradiction with (c-iii).
\end{proof}

\subsection{\texorpdfstring{$T_\be^x$}{} and \texorpdfstring{$[T_\be^x]$}{}}\label{SS:Tcka-a-uzavrenostF}

\begin{definition**}
We set $T_{\beta}^x = \{\wh p\set p \in P_{\beta}^x\}$.
\end{definition**}

\begin{notation**}\label{N:tree-like}\hfil
\begin{enumerate}[(a)]
\item The set $T_{\beta}^x$ is not a tree since any $t \in T_{\beta}^x$ satisfies $|t| \geq l(\beta)$.
However, if $s \in T_{\beta}^x, t \preceq s$, and $|t| \geq l(\beta)$, then $t \in T_{\beta}^x$ by (c-ii) in Proposition~\ref{P:object-P} and Remark~\ref{R:properties-of-R}(a).
We set $[T^x_\be] = \{z \in \cN \set \forall n \geq \ell(\beta)\colon z|n \in T_{\beta}^x\}$.

\item We define a mapping $\pi\colon [0,\alpha] \times \cN \to \cN$ by
\[
\pi(\beta,z) = ([z(\ell(\beta))]_1,[z(\ell(\beta)+1)]_1, \dots).
\]
Using Definition~\ref{D:object-P}, Definition~\ref{D:pripustnaA}, and Proposition~\ref{P:jump-trees-old},
we get that $\pi(\be,z)=\si^x_\be$ for $z\in [T^x_\be]$.

The mapping $\pi$ is $\Si^0_1(\ep)$-recursive.
To this end consider the mapping $\pi^*(\be,z,n)=\pi(\be,z)(n)$ (cf. Fact~\ref{F:partialrecursivefunctions}\eqref{I:recursivity-to-N}).
It is the composition of the mappings $(\be,z,n)\mapsto (\ell(\be)+n,z)$, $(k,z)\mapsto z_k$, and $[\,\,]_1$
which are $\Si^0_1(\ep)$-recursive due to Lemma~\ref{L:volba-epsilon}(g),(i),(e).
Therefore using Fact~\ref{F:partialrecursivefunctions}(e) the function $\pi^*$ is $\Si^0_1(\ep)$-recursive.
Thus also $\pi$ is recursive by Fact~\ref{F:partialrecursivefunctions}\eqref{I:recursivity-to-N}.

\item The set $P_{\beta}^x$ has the property that if $q\preceq p$ for $p \in P^x_\be$ and $q \in \cP$ is of odd length, then $q \in P^x_\be$.
We set $[P_{\beta}^x] = \{z \in \cP_{\infty} \set \forall l \in \omega \colon z|(2l+1) \in P_{\beta}^x\}$.
\end{enumerate}
\end{notation**}

\begin{proposition**}\label{P:uzavrenost-F}
The set $F := \bigl\{(x,y)\in \cN \times \cN\set y \in [T_0^x]\bigr\}$ is $\Pi^0_1(\ep)$ in $\cN \times \cN$.
\end{proposition**}

\begin{proof}
We define
\[
Z = \{(x,y,n) \in \cN \times \cN \times \omega\set \exists p \in P_{0}^x\colon y|n = \wh p\}.
\]
To apply Lemma~\ref{L:bounded-quantifiers} we define
\[
Z^* = \{(x,y,n,p) \in \cN \times \cN  \times \om \times \cP \set (0,p,x)\in P,\ y|n=\wh{p}\}
\]
and we use that
\[
Z = \{(x,y,n)\in\cN \times \cN \times \omega\set \exists p\le_{\cP} M_{\cP}(y|n)\colon (x,y,n,p)\in Z^*\}.
\]
The set $Z^*$ is $\Si^0_1(\ep)$-recursive since $P$ is $\Si^0_1(\ep)$-recursive by Proposition~\ref{P:object-P}
and the relation $y|n = \wh{p}$ is $\Si^0_1(\ep)$-recursive by Lemma~\ref{L:volba-epsilon}(a),(e),(i).
The mapping $(x,y,n)\mapsto M_{\cP}(y|n)$ is $\Si^0_1(\ep)$-recursive by Remark~\ref{R:upper-bounds}, Lemma~\ref{L:volba-epsilon}(i), and Fact~\ref{F:partialrecursivefunctions}(e).
Using Remark~\ref{R:recursivity-of-cP} we apply Lemma~\ref{L:bounded-quantifiers} for the corresponding extended product spaces to get $\Si^0_1(\ep)$-recursivity of $Z$.
Then we have
\[
\bigl\{(x,y)\in \cN \times \cN\set y \in [T_0^x]\bigr\} = \{(x,y) \in \cN \times \cN\set \forall n \in \omega, n \geq 1\colon (x,y,n) \in Z\}
\]
and therefore the set $F$ is $\Pi_1^0(\ep)$ (see Subsection~\ref{SS:higher-classes}).
\end{proof}

\begin{lemma**}\label{L:hatsurjection}
Let $x \in \cN$ and $\beta \in [0,\alpha)$.
\begin{enumerate}[\upshape (a)]
\item The mapping $\widehat{\phantom{a}} \colon P_{\beta}^x \to T_{\beta}^x$ is a bijection.

\item The mapping $\widehat{\phantom{a}} \colon [P_{\beta}^x] \to [T_{\beta}^x]$ is a bijection.

\item The mapping $\zeta$ on $[P^x_\be]$ is an injective mapping to $[T^x_{\be+1}]$.
\end{enumerate}
\end{lemma**}

\begin{proof}
(a) The mapping $\widehat{\phantom{a}}$ is onto by definition. It remains to prove injectivity.
Suppose that we have $p^0,p^1 \in P_{\beta}^x$ with $\widehat{p^0} = \widehat{p^1}$. We will proceed by induction over the length of $\widehat{p^0}$.
If $|\widehat {p^0}| = \ell(\beta)$, then $|p^0| = 1$ by Definition~\ref{D:pripustnaA} and Notation~\ref{N:tree-like}(a). Then we have $p^0 = (\tau(p^0)) = (\tau(p^1)) = p^1$ by Definition~\ref{D:Q} of $\cQ$.
If $|\widehat {p^0}| > \ell(\beta)$ then there exist $q^0,q^1 \in P_{\beta}^x$ such that $q^0 \R p^0$ and $q^1 \R p^1$ by Proposition~\ref{P:object-P}(c-ii).
We have $\widehat{q^0} = \widehat{q^1}$ by Remark~\ref{R:properties-of-R}(a).
Using induction hypothesis we infer $q^0 = q^1$. Due to this and to the fact that $\widehat{p^0} = \widehat{p^1}$ we have $p^0 = p^1$ by Remark~\ref{R:properties-of-R}(c).

\medskip\noindent
(b) Injectivity follows from (a). We prove surjectivity.
Let $w\in [T^x_\be]$. Using the definition of $T_{\beta}^x$ and (a), we find for every $k \in \omega$ a unique $r_k \in P^x_\be$ with $w|(\ell(\be)+k) = \wh{r_k}$.
Observe that $r_k \R r_{k+1}$ for every $k \in \omega$. Indeed, by the property Proposition~\ref{P:object-P}(c-ii) there exists $q \in P_{\beta}^x$ with $q \R r_{k+1}$.
By Remark~\ref{R:properties-of-R}(a) we have that $\widehat{q} = \widehat{r_{k+1}}|_{|\widehat{r_{k+1}}|-1} = w|_{|\widehat{r_{k+1}}|-1} = \widehat{r_k}$.
By (a) we have $q = r_k$.

We show that the sequence $\{\tau(r_k)\}_{k \geq 0}$ is constant and for every $j \in \omega$ the sequences
$\{n_j(r_k)\}_{k \geq j}$, $\{\eta_j(r_k)\}_{k \geq j}$ are eventually constant.
All $r_k$'s begin by the same finite sequence $\tau(r_k) = w|\ell(\be)$ since $|\tau(r_k)| = \ell(\be)$ by the definition of $\bbA$ (Definition~\ref{D:pripustnaA}).
Thus the sequence $\{\tau(r_k)\}_{k \geq 0}$ is constant. Now assume that the sequences
$\{n_j(r_k)\}_{k \geq j}$, $\{\eta_j(r_k)\}_{k \geq j}$ are eventually constant for every $j < j_0$.
Find $k_0 \in \omega$ such that for every $j < j_0$ the sequences $\{n_j(r_k)\}_{k \geq k_0}$, $\{\eta_j(r_k)\}_{k \geq k_0}$
are constant. The sequence $\{n_{j_0}(r_k)\}_{k > k_0}$ is constant, since $n_{j_0}(r_k) = w_m$, where
\[
m = \ell(\beta) + (j_0-1) + |\eta_0(r_{k_0})| + \dots + |\eta_{j_0-1}(r_{k_0})|.
\]
Now we distinguish two possibilities. If we have that $\eta_{j_0}(r_k) = -$ for every $k > k_0$, then we are done.
If this is not the case, there is $k' > k_0$ such that $\eta_{j_0}(r_{k'}) \neq -$. We have $r_{k'} \R r_{k'+1}$, $r_{k'+1} \R r_{k'+2},
\dots$ and $r_{k'}|(2j_0+2) = r_{k'+1}|(2j_0+2) = r_{k'+2}|(2j_0+2) = \dots$ by induction assumption.
Then by the definition of $\R$ we have $\eta_{j_0}(r_{k'}) = \eta_{j_0}(r_{k'+1}) = \eta_{j_0}(r_{k'+2}) = \dots$
It follows that the sequence $\{\eta_{j_0}(r_k)\}_{k \geq j_0}$ is eventually constant.

If we set $p_0 = \lim_{k \to \infty} \tau(r_k)$, $p_{2l+1} = \lim_{k \to \infty} n_l(r_k)$ and $p_{2l+2} = \lim_{k \to \infty} \eta_l(r_k)$ for $l \in \omega$,
then $p=\{p_i\}_{i\in\om} \in [P_{\be}^x]$ and $\wh{p} = w$.

\medskip\noindent
(c) By Proposition~\ref{P:object-P}(c-i) we have $\zeta(p) \in [T^x_{\be+1}]$ for every $p \in [P_{\beta}^x]$.
Now we suppose on the contrary that there are sequences $u,p\in [P^x_\be]$ such that $u \ne p$ and $\zeta(u) = \zeta(p)$.
We have $\xi(u) = \xi(p) = \sigma_{\beta}^x$ and there exists $m \in \omega$ with $u|m = p|m$ and $u|(m+1) \neq p|(m+1)$.
Since $\tau(u) = \zeta(u)|1 = \zeta(p)|1 = \tau(p)$, we get $u|1 = p|1$ and therefore $m \geq 1$.

First suppose that $m = 2j+2$ for some $j \in \omega$. Since $\eta_j(u) \neq \eta_j(p)$ we may assume without any loss of generality
that  $\eta_j(p) \neq -$. We find $k \in \omega$ such that $\xi(p|(2j+3)) \preceq \xi(u|(2k+3))$ and $p|(2j+3) <_{\cP} u|(2k+3)$. Then by Proposition~\ref{P:object-P}(c-iii) we get
$\eta_j(u) = \eta_j(u|(2j+3)) = \eta_j(p|(2k+3)) =\eta_j(p)$, a contradiction.

Now suppose that $m = 2j + 1$ for some $j \in \omega$. Then
\[
[u_m]_0 = \zeta(u)_{\ell(\beta)-1+j} = \zeta(p)_{\ell(\beta)-1+j} = [p_m]_0
\]
and $[u_m]_1 = [p_m]_1 = \si^x_\be(k)$, where $k = |\xi(q|m)| = |\xi(p|m)|$. This implies $u_m = p_m$. Consequently, $u|(m+1) = p|(m+1)$, a contradiction.
\end{proof}

\subsection{The mapping \texorpdfstring{$\psi_{\beta,\be+1}^x$}{}}\label{SS:psi-be-be+1}

We employ the following notation in the sequel.

\begin{notation**}
For $t \in \Seq$ with $|t| > 0$ we denote $t^- = t|(|t|-1)$. Thus we have $t = t^- \w t_{|t|-1}$.
\end{notation**}

\begin{definition**}\label{D:psi}
Let $x \in \cN$, $\beta \in [0,\alpha)$.
We define $\psi(t)=\psi_{\be,\be+1}^x(t)\in\cQ$ for every $t\in T^x_{\be+1}$ by the following inductive procedure.
We put first $\psi(t^-)=(t^-)$ if $|t|=\ell(\be+1)$ ($=\ell(\be)+1$). Assuming that $\psi(t^-)\in\cQ$ has already been defined, we define $\psi(t)$ as
the element of the form $\psi(t)=\psi(t^-)\w (n,\eta)\in\cQ$ described uniquely by the following properties.
\begin{itemize}
\item[(a)] $\zeta(\psi(t)) = t$, in particular $[n]_0=t_{|t|-1}$ and $\tau(\psi(t))=t|\ell(\be)$,

\item[(b)] $\xi(\psi(t)) \preceq \sigma_{\beta}^x$, in particular $[n]_1$ is uniquely defined,

\item[(c)] $\psi(t)$ is the $<_{\cP}$-minimal element of the set
\[
E(t)=\{\psi(t^-)\w (n,s)\in P^x_\be\set s\in\cS^*, s\ne -,\ \xi\bigl(\psi(t^-)\w (n,s)\bigr) \preceq \si^x_\be\}
\]
if $E(t)\ne\emptyset$ and $\eta=-$ otherwise.
\end{itemize}
\end{definition**}

\begin{remark**}\label{R:monotony-of-theta}
We may easily notice that from Definition~\ref{D:psi} it follows
\begin{itemize}
\item  $\psi^x_{\be,\be+1}(s) \preceq \psi^x_{\be,\be+1}(t)$ for $s \preceq t$ in $T^x_{\be+1}$ since $\psi^x_{\be,\be+1}(t^-) \preceq \psi^x_{\be,\be+1}(t)$,
\item $|\psi^x_{\be,\be+1}(t)|=|\psi^x_{\be,\be+1}(t^-)|+2$,
\item $\psi^x_{\be,\be+1}$ is injective by (a) in Definition~\ref{D:psi}.
\end{itemize}
\end{remark**}

\begin{lemma**}\label{L:range-of-psi}
Let $x \in \cN$ and $\be \in [0,\al)$. The mapping $\psi^x_{\be,\be+1}\colon T^x_{\be+1}\to\cQ$ takes values in $P^x_\be$.
\end{lemma**}

\begin{proof}
We fix $x\in\cN$ and $\be\in [0,\al)$ and use the abbreviation $\psi$ for $\psi^x_{\be,\be+1}$ again.
By Lemma~\ref{L:diamond}(d) we have $\ell(\be+1)=\ell(\be)+1$.

By Lemma~\ref{L:diamond}(a) there exists $k \in \omega$ such that
\[
\beta < \diamond^{k-1}(\beta+1) < \diamond(\beta) \leq \diamond^k(\beta+1).
\]
By Lemma~\ref{L:diamond}(c) we get $\diamond(\diamond^{k-1}(\beta+1)) =  \diamond^k(\beta+1) \leq \diamond(\beta)$.
Thus we have $\diamond^k(\beta+1) = \diamond(\beta)$.
For every $i \in \{0,\dots,k-1\}$ we have $a(\diamond^{i+1}(\beta+1),0) = \diamond^i(\beta+1)$ by Lemma~\ref{L:diamond}(c), consequently,
\begin{equation}\label{E:rovnost-ell}
\forall i \in \{0,\dots,k-1\} \colon \ell(\diamond^{i+1}(\be+1)) = \ell(\diamond^i(\be+1))
\end{equation}
by Lemma~\ref{L:diamond}(b).

\bigskip\noindent
\emph{Case 1.} Let $t \in T_{\beta +1}^x$ and $|t| = \ell(\beta+1)$.
Thus using the definition of $T^x_{\be+1}$, Proposition~\ref{P:object-P}(b), and \eqref{E:rovnost-ell} we get that
there exist $q^i \in P_{\diamond^i(\be+1)}^x$ such that  $t = \wh{q^0}$ and $\wh{q^i} = \wh{q^{i+1}}$ for every $i \in \{0,\dots,k-1\}$.
Then we have $\wh{q^k} \in T^x_{\diamond(\be)}$, $|\wh{q^k}| = |t| = \ell(\be) + 1$, and $\ell(\diamond(\be)) \leq \ell(\be)$.
Thus we have  $\wh{q^k}^- \in T_{\diamond(\be)}^x$. Since $\wh{q^k}^- = t^-$ we get $t^- \in T^x_{\be}$.
This shows that $\psi(t^-) = (t^-) \in P_{\beta}^x$.

\bigskip\noindent
\emph{Case 2a}.  Assume that $t\in T^x_{\be+1}$ with $\psi(t^-)\in P^x_\be$. Moreover, we assume that the set $E(t)$ introduced in Definition~\ref{D:psi} is nonempty.
Then immediately from the definition of $\psi$ we have $\psi(t) \in P_{\beta}^x$.

\bigskip\noindent
\emph{Case 2b.} Assume that $t\in T^x_{\be+1}$ with $\psi(t^-)\in P^x_\be$ and $E(t) = \emptyset$. In particular, $\psi(t^-)\in\bbA$ and satisfies one of the properties (a)--(c) of Proposition~\ref{P:object-P}. Recall that $\psi(t)$ is of the form $\psi(t^-) \w (n,-)$ in this case.
To this end we are going to verify that $(\be,\psi(t),x)\in\bbA$ and the properties (i)--(iii) of Proposition~\ref{P:object-P}(c)
(since $\be<\al$ and $|\psi(t)|>1$, the properties (a) and (b) of Proposition~\ref{P:object-P} are not satisfied by $(\be,\psi(t),x)$).

\medskip\noindent
We begin with verification of (c) from Proposition~\ref{P:object-P}.

\smallskip\noindent
(c-i) By the definition of $T^x_{\be+1}$ there is $q \in P^x_{\be+1}$ such that $t = \wh{q}$.
Setting $q_E = q$ the equality (a) from Definition~\ref{D:psi} gives (c-i).

\smallskip\noindent
(c-ii) This condition is satisfied since $\psi(t^-) \in P^x_\be$ and $\psi(t^-) \R \psi(t)$ by the definitions of $\psi(t)$ and $\fR$.

\smallskip\noindent
(c-iii) Suppose that we have $q\in P^x_\be$ and $i < r(\psi(t))$ such that
\[
\begin{split}
q \leq_{\cP} \psi(t) \ \land \ \xi(q) & \preceq \xi(\psi(t)) \ \land \ \eta_i(q) \neq - \\
&\land \ q|(2i+2) = \psi(t)|(2i+2) \ \land \ q|(2i+3) \le_{\cP} \psi(t)|(2i+3).
\end{split}
\]
If $i = r(\psi(t))-1 = r(\psi(t^-))$, then $\psi(t^-) = q|(2i+1)$. Consequently,  $q|(2i+3) \in E(t)$, a contradiction with $E(t) = \emptyset$.
Thus we may assume that $i <r(\psi(t))-1=r(\psi(t^-))$.
Then we have
\[
q|(2i+3)\le_{\cP} \psi(t)|(2i+3)=\psi(t^-)|(2i+3)\le_{\cP}\psi(t^-).
\]
Now (c-iii) for $\psi(t^-)$ gives $\eta_i(q) = \eta_i(q|(2i+3)) = \eta_i(\psi(t^-)) = \eta_i(\psi(t))$.
This verifies (c-iii) for $\psi(t)$.

\medskip\noindent
It remains to prove that $(\be,\psi(t),x)\in\bbA$.
By Definition~\ref{D:psi}(b) we have $\xi(\psi(t))\preceq\si^x_\be$ which implies that $\xi(\psi(t))\in S^x_\be$.
It follows easily from the definition of $\psi(t)$ that $\tau(\psi(t)) = t|\ell(\be)$ and so $|\tau(\psi(t))|=\ell(\be)$ for $t\in T^x_{\be+1}$.

It remains to prove that $(\psi(t),x) \in \bbJ$. If $i < r(\psi(t))$ and $\eta_i(\psi(t)) \neq -$, then $i < r(\psi(t^-))$ and $(i,\wh{\psi(t^-)},x) \in J_*$, since $\psi(t^-) \in P^x_\be$.
Since $\wh{\psi(t^-)} \preceq \wh{\psi(t)}$, we get $(i,\wh{\psi(t)},x) \in J_*$.
Now assume that $i < r(\psi(t))$ and $\eta_i(\psi(t)) = -$. We want to show that $(i,\wh{\psi(t)}, x) \notin J_*$.
Towards contradiction suppose that $(i,\wh {\psi(t)},x) \in J_*$. We may moreover assume that $i$ is the smallest number with the properties

\begin{equation}\label{CL:claim}
i<r(\psi(t)), \quad \eta_i(\psi(t)) = - , \text{ and } \quad \bigl(i,\wh {\psi(t)},x\bigr) \in J_*.
\end{equation}
If $i < r(\psi(t))-1$ we set
\begin{align*}
w &= \bigl(\tau(\psi(t)),n_0(\psi(t)),\eta_0(\psi(t)), \dots,n_i(\psi(t)),\eta_i(\psi(t)) \w n_{i+1}(\psi(t)) \w \dots \w \eta_{r(\psi(t))-1}(\psi(t))\bigr) \\
  \Bigl(&= \psi(t)|(2i+2) \w \bigl(\eta_i(\psi(t)) \w n_{i+1}(\psi(t)) \w \dots \w \eta_{r(\psi(t))-1}(\psi(t))\bigr) = \psi(t)|(2i+2)\w (\eta_i(w))\Bigr).
\end{align*}
If $i = r(\psi(t))-1$ we set $w = \psi(t)|(2i+2) \w (\emptyset)$. In both cases we have $w|(2i+2) = \psi(t)|(2i+2)$, $\eta_i(w) \neq -$, and $r(w) = i+1$.
We observe the following claim.

\begin{claima}
The property $\eta_i(\psi(t)) = -$ implies $w \notin P^x_\be$.
\end{claima}

\begin{proof}
If $i < r(\psi(t^-))$ and $w \in P^x_{\be}$, then the relations $\eta_i(\psi(t^-)) = \eta_i(\psi(t)) = -$, $\xi(w) \preceq \si^x_{\be}$, $\psi(t^-)|(2i+2)\preceq w$, and $\eta_i(w)\ne -$
give that $w\in E(t^i)$, where $t^i=\zeta(\psi(t^-)|(2i+3))$, which is a contradiction with Definition~\ref{D:psi}(c) for $\psi(t^-)|(2i+3)$.

If $i=r(\psi(t^-))$ and $w \in P^x_{\be}$, then the relations $\psi(t^-)\w (n) \preceq w$, $\xi(w) \preceq \si^x_{\be}$, and $\eta_i(w) \ne -$ give
that $w\in E(t)$ which is a contradiction with $E(t)=\emptyset$. This concludes the proof of Claim~A.
\end{proof}

Finally, we prove that (\ref{CL:claim}) implies also that $w \in P^x_\be$ which, together with the above Claim A, gives the desired contradiction.

\begin{claimb}
The property $(i,\wh {\psi(t)},x) \in J_*$ implies $w \in P_{\beta}^x$.
\end{claimb}

\begin{proof}
Clearly $w \in \cQ$ and $\xi(w)\in S^x_{\be}$, as an initial segment of $\si^x_\be$, by the definitions of $w$ and $\psi(t)$.
Since $\wh{w} = \wh{\psi(t)}$ and $i$ is the minimal integer witnessing $(\psi(t),x) \notin \bbJ$, we have for $j < i$ that $(j,\wh{w},x) \in J_*$ if and only if $\eta_j(w) \neq -$.
Further we have $(i,\wh{w},x) \in J_*$ by our assumption and $\eta_i(w) \ne -$. Thus $(w,x) \in \bbJ$ and $(\beta,w,x) \in \bbA$.
Since $\be <\al$ and $|w|>1$, it remains to prove (c-i)--(c-iii) of Proposition~\ref{P:object-P} for $w$.

We have $\zeta(w) \preceq t$ and $|\zeta(w)| \geq l(\beta+1)$. By Notation~\ref{N:tree-like}(a)
we have $\zeta(w)  \in T_{\beta+1}^x$.
This shows that $(\beta,w,x)$ satisfies the condition (c-i) from Proposition \ref{P:object-P}.
The condition (c-ii) is satisfied because clearly $\psi(t^-) \R w$.

Now we check (c-iii).
Let $q \in P^x_\be$ satisfy $q \leq_{\cP} w$ and $j < r(w)$.
Towards contradiction suppose that
\[
\begin{split}
\xi(q) \preceq \xi(w) \ \land \ \eta_{j}(q) \neq - \land \ q|(2j+2) = w|(2j+2) \\
q|(2j+3) \leq_{\cP} &w|(2j+3) \ \land \ \eta_{j}(q) \neq \eta_{j}(w).
\end{split}
\]
First suppose that moreover $j < i$.
Then we have $w|(2j+3) = \psi(t^-)|(2j+3)$. This implies that $q|(2j+2) = \psi(t^-)|(2j+2)$ and $q|(2j+3) \leq_{\cP} \psi(t^-)|(2j+3)$.
Since $\psi(t^-)|(2j+3) \in P_{\beta}^x$ we have $\eta_j(\psi(t^-)|(2j+3)) = \eta_j(q)$. This implies $\eta_j(q) = \eta_j(w)$, a contradiction.

Now suppose that $j = i < r(\psi(t))-1$. Then we have $\psi(t^-)|(2j+2) = w|(2j+2)$ and $\eta_j(\psi(t^-))=\eta_j(\psi(t))=-$. Using induction hypothesis for $\psi(t^-)$, we get that there is no $q \in P_{\beta}^x$ such that $\xi(q) \preceq \sigma_{\beta}^x$, $\psi(t^-)|(2j+2) = q|(2j+2)$, and $\eta_j(q) \neq -$, a contradiction.

Finally suppose that $j = i = r(\psi(t))-1$. Then $\eta_{j}(w) = \emptyset$ and $w$ is minimal with respect to $<_{\cP}$ extension of
$\psi(t^-)\w (n)$
(see Notation~\ref{N:usporadanicP}) such that $\eta_j(w)\ne -$.
Thus there is no $q \leq_{\cP} w$ such that $q|(2j+2)=w|(2j+2)$, $\xi(q)\preceq\xi(w)$, $\eta_j(q)\ne -$, and $\eta_j(q) \neq \eta_j(w)$,
a contradiction. Therefore (c-iii) for $w$ is satisfied and Claim~B is proved.
\end{proof}
\end{proof}

\begin{notation**}\label{N:psi-hat}
We define a function $\wh{\psi}^x_{\be,\be+1}\colon T_{\beta+1}^x \to T_{\beta}^x$ by $\wh{\psi}^x_{\be,\be+1}(t) = \widehat{\psi_{\be,\be+1}^x(t)}$.
\end{notation**}

\begin{remark**}\label{R:monotony-of-theta-2}
We may easily notice that for every $t \in T_{\beta+1}^x$ we have
\[
|\wh{\psi}_{\be,\be+1}^x(t)| \geq |\zeta\bigl(\psi_{\be,\be+1}^x(t)\bigr)| = |t|
\]
due to Definition~\ref{D:psi}.
\end{remark**}

\begin{lemma**}[recursivity of $\psi_{\be,\be+1}^x$]\label{L:properties-of-theta}
Let $x \in \cN$. There exists a partial function $u\colon [0,\alpha) \times \Seq \times \cN  \rightharpoonup \cP$
which is $\Sigma_1^0(\ep)$-recursive on its domain such that for $\be \in [0,\alpha], t \in T_{\be+1}^x$ we have $\psi_{\be,\be+1}^x(t) = u\bigl(\be,t,x^{(\be+1)}\bigr)$.
\end{lemma**}

\begin{proof}
First we denote $U = U^{[0,\alpha] \times \cS \times \cN,\cP}_{\Si_1^0(\ep)}$ and we fix a $\Si^0_1(\ep)$ set
$\De^U \su [0,\al] \times \cS \times \cN \times \cP$ computing $U$, see Remark~\ref{R:extended-facts} and Lemma~\ref{L:natural-base}(b).
We set $Z = [0,\alpha] \times \cP \times \Seq \times \cN$ and
\[
\begin{split}
F &= \Bigl\{(\beta,w,t,z) \in  Z \set \beta < \alpha \land  w \in P^x_\be \land \zeta(w) = t \land \xi(w) \preceq z \\
& \land \Bigl(r(w) \geq 1 \Rightarrow \bigr(\eta_{r(w)-1}(w) \neq - \vee (\forall q \in P_{\beta}^x, \xi(q) \preceq z, w|2r(w) \preceq q \colon \eta_{r(w)-1}(q) = -)\bigr)\Bigr) \Bigr\}.
\end{split}
\]

Using Definition~\ref{D:psi} and Lemma~\ref{L:range-of-psi}, we have the following claim.

\begin{claima}
For $t \in T_{\beta+1}^x$ we have that $\psi^x_{\be,\be+1}(t)$ is the $<_{\cP}$-minimal $w\in\cP$ such that $(\be,w,t,\si^x_\be)\in F$ and $\psi^x_{\be,\be+1}(t^-) \preceq w$.
\end{claima}

We show that $F$ is in $\Pi_1^0(x,\ep)$. First observe that $[0,\al)$ is $\Si_1^0(\ep)$-recursive in $[0,\al]$ by definition.
The set $P^x$ is $\Si^0_1(x,\ep)$-recursive by Proposition~\ref{P:object-P}. The sets $\{(w,t) \in \cQ \times \cS\set \zeta(w) = t\}$ and $\{(w,z)\in\cQ \times \cN \set \xi(w) \preceq z\}$ are $\Si^0_1(\ep)$-recursive  sets due to $\Si^0_1(\ep)$-recursivity of the mappings $\zeta$ and $\xi$ (see Lemma~\ref{L:recursivity-of-Q}(b)) and by $\Si^0_1(\ep)$-recursivity of the equality relation on $\cS^*$ and of $\preceq \ \subset \cS\times\cN$ due to Lemma~\ref{L:volba-epsilon}(a),(h).
Also the sets $\{w \in \cP \set r(w) \geq 1\}$ and $\{w \in \cP \set \eta_{r(w)-1}(w) \neq -\}$ are $\Si^0_1(\ep)$-recursive
by $\Si^0_1(\ep)$-recursivity of $r$ and of $(i,p)\mapsto \eta_i(p)$ due to Lemma~\ref{L:recursivity-of-Q}(b) and Lemma~\ref{L:volba-epsilon}(e).
Now one can easily infer that the set of all $(q,\beta,w,t,z) \in \cQ \times Z$ such that
\[
(q \in P_{\beta}^x \land \xi(q) \preceq z \land w|2r(w) \preceq q) \Rightarrow \eta_{r(w)-1}(q) = -
\]
is $\Pi_1^0(x,\ep)$.
The class $\Pi_1^0(x,\ep)$ is closed with respect to $\forall^{\cP}$ (see Subsection~\ref{SS:higher-classes}).
Thus the set of all $(\beta,w,t,z) \in Z$ such that
\[
\forall q \in P_{\beta}^x, \xi(q) \preceq z, w|2r(w) \preceq q \colon \eta_{r(w)-1}(q) = -
\]
is $\Pi_1^0(x,\ep)$. Now we have that $F \in \Pi_1^0(x,\ep)\restriction Z$.

Since the space $\Om=[0,\alpha] \times \cP \times \cS$ is of type $0$, we may use Lemma~\ref{P:jump-and-classes-2} to find sets
$W_0, W_1 \in \Sigma_1^0(x,\ep)\restriction Z$ such that
\begin{align*}
(\beta,w,t,z) \notin F   &\Leftrightarrow (\beta,w,t,z_x') \in W_0, \\
(\beta,w,t,z) \in F      &\Leftrightarrow (\beta,w,t,z_x') \in W_1.
\end{align*}
By Definition~\ref{D:relativizations} we find $\Sigma_1^0(\ep)$ sets $H_0,H_1 \subset Z \times \cN$ such that $W_0 = H_0^x$ and $W_1 = H_1^x$. Further we define
\[
\tilde C_i = \bigl\{(\beta,w,t,y) \in Z \set \beta < \alpha \land \bigl(\beta, w, t, \Rleft(\beta,y), r_3(0,\beta+1,y)\bigr) \in H_i\bigr\}, \qquad i = 0, 1,
\]
where $\Rleft$ is the function introduced in Proposition~\ref{P:sigma-beta-x-jump} and
$r_3$ is the function from Lemma~\ref{L:jumps}(3).
The functions $\Rleft$ and $r_3$ are partial functions which are $\Sigma_1^0(\ep)$-recursive on their domains.
Therefore there exist sets $C_0,C_1$ in $\Sigma_1^0(\ep)$ such that for $(\beta,w,t,y) \in Z$ with $(\beta,y) \in D(\Rleft)$ and $(0,\beta+1,y) \in D(r_3)$
we have $(\beta,w,t,y) \in C_i$ if and only if $(\beta,w,t,y) \in \tilde C_i$, $i \in \{0,1\}$.
One can simply define
\[
\begin{split}
C_i = \bigl\{(\beta,w,t,y) \in Z &\set \beta < \alpha \ \land \exists u,v\in \cS \ \exists a,b,c,d \in \omega\colon
(\beta,y,u|a) \in \Delta^{\hskip-2pt\Rleft}, \\
&(0,\be+1,y,v|b) \in \Delta^{r_3},(\be,w,t,u|c,v|d) \in (H_i)^* \bigr\}, \qquad i = 0, 1,
\end{split}
\]
where $\Si_1^0(\ep)$-sets $\Delta^{\hskip-2pt\Rleft}$ and $\Delta^{r_3}$ compute the functions $\Rleft$ and $r_3$ on their domains respectively and $(H_i)^* \subset [0,\alpha] \times \cP \times \cS^3$ is a $\Si_1^0(\ep)$-set corresponding to $H_i$ by Lemma~\ref{L:natural-base}(a).

\begin{claimb}\label{properties-of-theta:claimb}
For all $(\be,w,t)\in \Om$ we have that
\begin{itemize}
\item[\upshape (a)] $(\be,w,t,\si^x_\be)\in F$ if and only if $(\be,w,t,x^{(\be+1)})\in C_1$ and

\item[\upshape (b)] $(\be,w,t,\si^x_\be)\notin F$ if and only if $(\be,w,t,x^{(\be+1)})\in C_0$.
\end{itemize}
\end{claimb}

\begin{proof}
(a) Notice that substituting $z$ by $\si^x_\be$ we get for all $(\be,w,t)\in [0,\al)\times \cP \times T^x_{\be+1}$  that the following statements are equivalent due to the above definitions:
\[
\begin{split}
(\be,w,t,\si^x_\be) \in F &\Leftrightarrow (\be,w,t,(\si^x_\be)_x') \in W_1 \Leftrightarrow (\be,w,t,(\si^x_\be)_x',x)\in H_1 \\
&\Leftrightarrow (\be,w,t,x^{(\be+1)})\in C_1.
\end{split}
\]

The part (b) can be inferred analogously.
\end{proof}

We define a partial function $f\colon \omega \times [0,\alpha] \times \cS \times \cN \rightharpoonup \mathcal P$ by
\begin{enumerate}[(a)]
\item $f(e,\beta,t,y) = (t)$, whenever $t \in \cS, |t| = \ell(\beta)$,

\item $f(e,\beta,t,y) = w$, whenever $t \in T_{\beta+1}^x$, $y = x^{(\be+1)}$, and the following conditions are satisfied
\begin{enumerate}[(i)]
\item $\bigl(\beta,w,t,y\bigr)\in C_1\sm C_0$,

\item $U(e,\beta,t^-,y) \preceq w$,

\item
\begin{equation}\label{E:minimality-1}
\forall v \in \cP, v <_{\cP} w, \bigl(\beta,v,r_3(0,\beta+1,y)\bigr) \in P, U(e,\beta,t^-,y) \preceq v\colon
\bigl(\beta,v,t,y\bigr)\in C_0\sm C_1.
\end{equation}
\end{enumerate}
\end{enumerate}
To show that $f$ is $\Si_1^0(\ep)$-recursive on its domain we define $\De^f$ computing $f$ as follows.  For $(e,\beta,t,y,w) \in \omega \times [0,\alpha] \times \cS \times \cN \times \cP$
we define $\De^f$ by setting $(e,\beta,t,y,w) \in \De^f$ if and only if
\begin{enumerate}[(a)]
\item $t \in \cS$, $|t| = \ell(\beta)$, and $w = (t)$, or

\item $t \in T_{\beta+1}^x$ and the following conditions are satisfied
\begin{enumerate}[(i)]
\item $(\beta,w,t,y)\in C_1$,
\item there is $z\in\cP$ such that $(e,\beta,t^-,y,z)\in\De^U$, $z \preceq w$, and
\begin{equation}\label{E:minimality-2}
\begin{split}
\forall v \in \cP, v <_{\cP} w &,  z \preceq v \colon \\
& \bigl(\exists u \in \cS \colon (0,\be+1,y,u) \in \Delta^{r_3} \land (\beta,v,u) \in P^*\bigr) \Rightarrow (\beta,v,t,y) \in C_0.
\end{split}
\end{equation}
\end{enumerate}
\end{enumerate}
The set $P^* \subset [0,\alpha] \times \cP \times \cS$ corresponds to the set $P$ via Lemma~\ref{L:natural-base}.

Using Facts~\ref{F:zachovani-recursivita}, \ref{F:partialrecursivefunctions}, Lemma~\ref{L:bounded-quantifiers},
Lemma~\ref{L:volba-epsilon}(d), we get that $\De^f$ is a $\Si^0_1(\ep)$ set in $\om \times [0,\al] \times \cS \times \cN \times \cP$.

Now we show that $\Delta^f$ computes $f$. Consider $(e,\be,t,y) \in D(f)$ and put $f(e,\be,t,y) = w$.
Assume that $(e,\be,t, y,w') \in \Delta^f$ for some $w' \in \cP$.
If $|t| = \ell(\be)$, then obviously $w = w'$. So assume that $t \in T^x_{\be+1}$ and $y = x^{(\be+1)}$.
Then we have $(\be,w,t,x^{(\be+1)}) \in C_1\sm C_0$ by the condition (b-i) of the definition of $f$ and $(\be,w',t,x^{(\be+1)}) \in C_1$ by the condition (b-i) of the definition of $\De^f$. Therefore both $(\be,w,t,\sigma_{\be}^x)$ and $(\be,w',t,\sigma_{\be}^x)$ are elements of $F$ by Claim~B. In particular $w,w' \in P_{\be}^x$ by the definition of $F$.

By the definition of $f$, we get from $f(e,\be,t,x^{(\be+1)})=w$ that $(e,\be,t^-,x^{(\be+1)}) \in D(U)$ and so $z=U(e,\be,t^-,x^{(\be+1)})$ from the condition (b-ii) of the definition of $\De^f$ is uniquely determined and $z \preceq w'$.

If $w' <_{\cP} w$, we get using $w' \in P^x_\be$ and \eqref{E:minimality-1} that $(\be,w',t,x^{(\be+1)}) \notin C_1$. This is in contradiction with the corresponding above observation.

If $w <_{\cP} w'$, using that $w \in P_{\be}^x$, we find $u \in \cS$ such that $u \preceq r_3(0,\be+1,x^{(\be+1)})$ and $(\be,w,u) \in P^*$. Using \eqref{E:minimality-2} we get $(\be,w,t,x^{(\be+1)}) \in C_0$, which is in contradiction with the corresponding above observation.
Thus we have $w=w'$.

Applying Fact~\ref{F:recursionthms}(b), we find $e^* \in \omega$ such that $f(e^*,\beta,t,y) = U(e^*,\beta,t,y)$ whenever $f(e^*,\beta,t,y)$ is defined.
The desired function $u$ is defined by $u = f_{e^*}$.

By induction on length of $|t|$ we show that $\psi_{\be,\be+1}^x(t) = u(\beta,t,x^{(\beta+1)})$ for $\beta \in [0,\alpha)$ and $t \in \cS$ with $|t|=\ell(\be)$ or with $t \in T_{\beta+1}^x$.
If $|t|=\ell(\be)$, we defined both $f(e^*,\be,t,x^{(\be+1)})$ and $\psi_{\be,\be+1}^x(t)$ by $(t)$.

Let $t\in T^x_{\be+1}$ and assume that the equality was proved for $t^-$. We use the notation $\psi= \psi_{\beta,\beta+1}^x$.
Using Claim A, we know that $(\be,\psi(t),t,\si^x_\be)\in F$. Applying Claim B, we get the equality $(\be,\psi(t),t,x^{(\be+1)})\in C_1\sm C_0$.
This proves the condition (i) from the definition of $f(e^*,\be,t,x^{(\be+1)}) = \psi(t)$.

Using the induction hypothesis and the definition of $\psi(t)$, we have
\[
u(\be,t^-,x^{(\beta+1)}) = \psi(t^-) \preceq \psi(t).
\]
This verifies (ii) from the definition of $f(e^*,\be,t,x^{(\be+1)}) = \psi(t)$.

If $v \in P_{\beta}^x$, $v <_{\cP} \psi(t)$, and $\psi(t^-)=u(\be,t^-,x^{(\beta+1)})=U(e^*,\beta,t^-,x^{(\beta+1)}) \preceq v$, then
$(\beta,v,t,\sigma_{\beta}^x) \notin F$ as $\psi(t)$ is the minimal element of $F$ such that $U(e^*,\be,t^-,x^{(\beta+1)})=\psi(t^-)\prec\psi(t)$ using Claim~A.
This means that $\bigl(\beta,v,t,x^{(\beta+1)}\bigr)\in C_0\sm C_1$ due to Claim B.
Thus condition (iii) from the definition of $f(e^*,\be,t,x^{(\be+1)})=\psi(t)$ is also satisfied and this concludes the verification of $u(\be,t,x^{(\be+1)})=\psi(t)$.
\end{proof}

\subsection{The mapping \texorpdfstring{$\Psi^x_{\be,\be+1}$}{}}\label{SS:Psi-be-be+1}

\begin{definition**}\label{D:hPsi-pro-nasledniky}\hfil
\begin{enumerate}[(1)]
\item For any $x \in \cN$ and $\beta \in [0,\alpha)$ we define a mapping $\Psi_{\beta,\be+1}^x \colon [T_{\beta+1}^x] \to \cP_\infty=(\Seqm)^{\omega}$ as follows.
If $z \in [T_{\beta+1}^x]$, then by Remarks~\ref{R:monotony-of-theta} and \ref{R:monotony-of-theta-2}
we have $\psi_{\beta,\be+1}^x(z|l) \preceq \psi_{\beta,\be+1}^x(z|(l+1))$ and $|\psi_{\beta,\be+1}^x(z|l)| \geq l$ for every $l \geq \ell(\beta+1)$.
Therefore there exists a uniquely determined element $\Psi_{\beta,\be+1}^x(z)$ of $\cP_\infty$ such that
$\psi_{\beta,\be+1}^x(z|l) \preceq \Psi_{\beta,\be+1}^x(z)$ for every $l \geq \ell(\beta+1)$.

\item Using Lemma~\ref{L:range-of-psi}, we define a mapping $\hPsi_{\beta,\be+1}^x \colon [T_{\beta+1}^x] \to [T_{\beta}^x]$ using the formula $\hPsi_{\beta,\be+1}^x(z) = \wh{\Psi_{\beta,\be+1}^x(z)}$.
\end{enumerate}
\end{definition**}

\begin{lemma**}[bijectivity of $\Psi_{\beta,\be+1}^x$]\label{L:bijectivity-of-Theta} Let $x \in \cN$.
\begin{enumerate}[\upshape (a)]
\item The mapping $\Psi^x_{\be,\be+1}$ is a well defined bijection of $[T^x_{\be+1}]$ onto $[P^x_{\be}]$  for every $\beta \in [0,\alpha)$.
\item The mapping $\hPsi_{\be,\be+1}^x$ is a well defined bijection of  $[T_{\beta+1}^x]$ onto $[T_{\beta}^x]$ for every $\beta \in [0,\alpha)$.
\end{enumerate}
\end{lemma**}

\begin{proof}
(a) The value $\Psi_{\beta,\be+1}^x(z)$ of the mapping ${\Psi_{\beta,\be+1}^x}\colon [T_{\beta + 1}^x] \to (\Seqm)^{\omega}$ for each $z\in [T^x_{\be+1}]$
is well defined since the sequence of
$\bigl(\psi_{\beta,\be+1}^x(z|l)\bigr)_{l \geq \ell(\beta+1)}$ is ordered by $\preceq$ and the lengths tend to infinity by Remark~\ref{R:monotony-of-theta}.
The injectivity follows from Remark~\ref{R:monotony-of-theta}.
The mapping ${\Psi_{\beta,\be+1}^x}$ is into $[P^x_\be]$ by Lemma~\ref{L:range-of-psi}.

It remains to prove the surjectivity. Let $p\in [P^x_\be]$.
By Proposition~\ref{P:object-P}(c-i) we have $\zeta(p)\in [T^x_{\be+1}]$.
Now $q :=\Psi^x_{\be,\be+1}(\zeta(p))$ is a well defined element of $[P^x_\be]$.
The mapping $\psi^x_{\be,\be+1}$ fulfils $\zeta(\psi^x_{\be,\be+1}(t)) = t$, $t \in T^x_{\beta+1}$,
by Definition~\ref{D:psi}(a). Therefore $\zeta(q) = \zeta\bigl(\Psi^x_{\be,\be+1}(\zeta(p))\bigr) = \zeta(p)$.
By Lemma~\ref{L:hatsurjection}(c) $p=q$, and we proved the surjectivity.

\medskip\noindent
(b) This follows from (a) and Lemma~\ref{L:hatsurjection}(b).
\end{proof}

\begin{lemma**}[recursivity of $\Psi_{\beta,\be+1}^x$]\label{L:recursivity-of-Theta}
Let $x \in \cN$.
\begin{enumerate}[\upshape (a)]
\item The mapping $(\beta,z) \mapsto \Psi_{\beta,\be+1}^x(z)$ is a partial function from $[0,\alpha] \times \cN$ to $\cP_\infty$ which is $\Sigma_1^0(\ep)$-recursive on its domain
$[0,\alpha) \times [T^x_{\be + 1}]$.
\item The mapping $(\beta,z) \mapsto \bigl({\hPsi_{\beta,\beta+1}^x(z)}\bigr)'_x$ is a partial function from $[0,\alpha] \times \cN$ to $\cN$ which is $\Sigma_1^0(\ep)$-recursive on its domain
$[0,\alpha) \times [T^x_{\be + 1}]$.
\end{enumerate}
\end{lemma**}

\begin{proof}
(a) Let $u$ be the function from Lemma~\ref{L:properties-of-theta} and
$\De^u\su [0,\al]\times\cS\times\cN\times\cP$ be a $\Sigma_1^0(\ep)$-recursive set computing $u$ on its domain which exists by Lemma~\ref{L:natural-base}(b).
Let moreover $\rright\colon [0,\al]\times\cN\rightharpoonup\cN$ be the $\Si^0_1(\ep)$-recursive on its domain function from Lemma~\ref{P:jump-trees-old}(c) and
$\pi\colon [0,\al]\times\cN\to\cN$ be the $\Si^0_1(\ep)$-recursive mapping from Notation~\ref{N:tree-like}(b).
Then applying the corresponding definitions we get for $\be\in [0,\al)$, $z\in [T^x_{\be+1}]$, and $p\in\cP$ the following equivalences
\begin{align*}
&\Psi_{\beta,\be+1}^x(z) \in \mathcal P_{\infty}(p) \\
\Leftrightarrow &\exists L_1,L_2 \in \omega\colon \psi_{\beta,\be+1}^x(z|L_1)|L_2 = p \\
\Leftrightarrow &\exists L_1,L_2 \in \omega\colon u(\beta, z|L_1, x^{(\beta+1)})|L_2 = p \\
\Leftrightarrow &\exists t \in \Seq \ \exists q \in \mathcal P \ \exists L_1,L_2 \in \omega\colon u(\beta, t, x^{(\beta+1)}) = q \land  q|L_2 = p \land z|L_1 = t \\
\Leftrightarrow &\exists t \in \Seq \ \exists q \in \mathcal P \ \exists L_1,L_2 \in \omega\colon (\beta, t, x^{(\beta+1)},q) \in \De^u \land q|L_2 = p \land z|L_1 = t\\
\Leftrightarrow &\exists t \in \Seq \ \exists q \in \mathcal P \ \exists L_1,L_2 \in \omega\colon
(\beta, t,\rright(\beta+1,\pi(\beta+1,z)) ,q) \in \De^u \land q|L_2 = p \land z|L_1 = t.
\end{align*}
Using Lemma~\ref{L:natural-base}(b), we find $\Si^0_1(\ep)$-recursive set
$\De^{\rright}\su [0,\al]\times\cN\times\cS$ computing $\rright$ on its domain.
Using Lemma~\ref{L:natural-base}(a), we find a $\Si^0_1(\ep)$-recursive set $(\De^u)^*\su [0,\al]\times\cS\times\cS\times\cP$ corresponding to $\De^u$.
Then we get that the set $\De^{\Psi}$ defined by
\begin{align*}
\De^{\Psi}=\{&(\be,z,p) \in [0,\al) \times \cN \times \cP\set
\exists t \in \Seq \ \exists q \in \mathcal P \ \exists L_1,L_2 \in \omega\ \exists w\in \cS \ \exists a, b \in \omega\colon\\
&(\beta+1,\pi(\beta+1,z),w|a) \in \De^{\rright} \land
(\beta, t, w|b ,q) \in (\De^u)^* \land q|L_2 = p \land z|L_1 = t\}
\end{align*}
computes $(\beta,z) \mapsto \Psi_{\beta,\be+1}^x(z)$ on $[0,\al)\times [T^x_{\be+1}]$.
Applying the facts that projections of $\Si^0_1(\ep)\restriction (\cZ\times \cX)$ sets along type $0$ extended space
$\cZ = \Seq \times \mathcal P \times \omega^2 \times \cS \times \om^2$ are $\Si^0_1(\ep)$ in $\cX = [0,\alpha] \times \cN \times \mathcal P$ by Fact~\ref{F:properties-of-Gamma} and Remark~\ref{R:extended-facts},
the restrictions $(n,q)\in\om\times\cP\mapsto q|n\in\cP$ and $(n,z)\in\om\times\cN\mapsto z|n$ are $\Si^0_1(\ep)$-recursive on their domains by Lemma~\ref{L:volba-epsilon}(e),(i), and $\De^{\rright}$, $(\De^U)^*$, and $\pi$ are $\Si^0_1(\ep)$-recursive,
we get that $\De^{\Psi}$ is a $\Si^0_1(\ep)$ set in $[0,\alpha] \times \cN \times \mathcal P$ which
proves $\Sigma_1^0(\ep)$-recursivity of the mapping $(\beta,z) \mapsto \Psi_{\beta,\be+1}^x(z)$
on its domain.

\medskip\noindent
(b) Let $\be\in [0,\al)$, $z \in [T_{\beta+1}^x]$, and $i \in \omega$. Denote $w = \Psi_{\beta,\beta+1}^x(z)$.
Let further $J_*$, $G$, $G^*$, and $R$ be as in Definition~\ref{D:J-jump}.
We are going to prove first that $\wh{w}'_x(i)=1$ if and only if $w_{2i+2}\ne -$.

If $\wh{w}'_x(i)=1$, then using Lemma~\ref{L:J-star-jump}(a) we find $L\in\om$ such that $(i,\wh{w}|L',x)\in J_*$ for every $L'\ge L$.
Let $l\in\om$ be chosen so that $l\ge i$ and $\wh{w}|L\preceq \wh{w|2l+3}$.
Then $(i,\wh{w|2l+3},x)\in J_*$ and, since in the same time $w|(2l+3) \in P^x_\be \su \bbJ^x$ by Lemma~\ref{L:range-of-psi},
we get that $\eta_i(w|(2l+3)) = w_{2i+2}\ne -$.

Conversely, let $w_{2i+2}\ne -$. Using Lemma~\ref{L:range-of-psi}, we know that $w|(2i+3) \in P^x_{\be} \su \bbJ^x$
and since $i < r(w|(2i+3)) = i+1$ we have $(i,\wh{w|(2i+3)},x) \in J_*$.
By Lemma~\ref{L:J-star-jump}(b) we get $\wh{w}'_x(i) = 1$, which concludes the proof of the equivalence above.

Applying the just proved equivalence, we get that the equality $(\hPsi^x_{\be,\be+1}(z))'_x(i)=1$ holds if and only if the equality $(\Psi^x_{\be,\be+1}(z))_{2i+2}\ne -$ holds for $(i,\be,z)\in \om\times[0,\al)\times [T^x_{\be+1}]$.
Thus the mapping $(i,\be,z)\mapsto (\hPsi^x_{\be,\be+1}(z))'_x(i)$ on the domain $\om\times[0,\al)\times [T^x_{\be+1}]$ is the composition of
the mappings
\[
\begin{split}
(i,\be,z)\mapsto &(2i+2,\Psi^x_{\be,\be+1}(z)) \in \om\times\cP_{\infty}, \qquad
(j,w)\in \om \times\cP_{\infty}\mapsto w_j \in \cS^*,  \\
&\text{and } t \in \cS^* \mapsto \iota(t) \in \om, \text{ where $\iota(-)=0$ and $\iota(t) = 1$ otherwise.}
\end{split}
\]
The first mapping is $\Si^0_1(\ep)$-recursive on $\om\times [0,\al)\times [T^x_{\be+1}]$ due to (a) of this lemma,
the next mapping is $\Si^0_1(\ep)$-recursive on $\om\times\cP_\infty$ by Lemma~\ref{L:volba-epsilon}(i),
and the last mapping is a characteristic function of the $\Si^0_1(\ep)$-recursive set $\{p\in\cP\set p\ne -\}$.
The application of Fact~\ref{F:partialrecursivefunctions}(\ref{I:recursivity-to-N}),(\ref{I:composition}) and Remark~\ref{R:extended-facts} concludes the proof.
\end{proof}

\subsection{Construction of \texorpdfstring{$\hPsi_{\beta,\gamma}^x$}{}}\label{SS:hPsi-beta-gamma}

We fix $x \in \cN$.
See Proposition~\ref{P:diamond-path} and Notation~\ref{N:L} for the definition of $n(L)$, $\beta^i_L, i \leq n(L)$, and $L(\be,\ga)$.
We are going to define mappings
\[
\tpsi_{\beta,\gamma}\colon T^x_{\gamma} \cap \{t \in \Seq\set |t|\ge L(\be,\ga)\} \to
T^x_{\beta} \cap \{t \in \Seq\set |t|\ge L(\be,\ga)\}, \qquad 0 \leq \beta < \gamma \leq \alpha,
\]
using the already defined family of mappings $\hpsi^x_{\be,\be+1}$ (see Notation~\ref{N:psi-hat}).
Let $t \in T^x_\ga$ be such that $|t| \geq L(\beta,\gamma)$.
Thus there exist $i, j \in \omega$  such that $0 \leq i < j \leq n(|t|)$, $\beta_{|t|}^i = \gamma$, and $\beta_{|t|}^j = \beta$.

First suppose that $j = i+1$. Then we set
\[
\tpsi_{\be,\ga}^x(t) =
\begin{cases}
t                                   &\text{if } \diamond(\be)=\ga, \\
{\hpsi_{\beta,\be+1}^x(t)}||t|      &\text{if } \be+1=\ga.
\end{cases}
\]
In the first case, we have $|t| = \ell(\beta)$ by  Proposition~\ref{P:diamond-path}. By Proposition~\ref{P:object-P}(b) we have $(t) \in P_{\beta}^x$.
Consequently, we have $\tpsi_{\be,\ga}^x(t) = t \in T^x_{\beta} \cap \{s \in \Seq\set |s|\ge L(\be,\ga)\}$.
In the other case, we have $\ell(\beta) <  \ell(\be+1) \leq |t|$. Since  $|\hpsi_{\beta,\be+1}^x(t)| \geq |t|$ by Remark~\ref{R:monotony-of-theta-2}
we have $\hpsi_{\beta,\be+1}^x(t)||t| \in T_{\beta}^x \cap \{s \in \Seq\set |s|\ge L(\be,\ga)\}$.
Thus we have that $\tpsi^x_{\be,\ga}(t)$ is a well defined element of $T^x_\be \cap
\{s \in \Seq\set |s|\ge L(\be,\ga)\}$ in this particular case of $\be$ and $\ga$.

For general $\be<\ga$ we define
\[
\tpsi^x_{\be,\ga}(t)= \tpsi^x_{\be^j_{|t|},\be^{j-1}_{|t|}}\circ\dots\circ\tpsi^x_{\be^{i+1}_{|t|},\be^{i}_{|t|}}(t).
\]
Thus $\tpsi^x_{\be,\ga}(t)$ is a  uniquely defined element of $T_{\be}^x$.
Notice that the mapping $\tpsi^x_{\be,\ga}$ preserves the length of sequences. We prove monotonicity of $\tpsi^x_{\be,\ga}$.

\begin{lemma**}\label{L:consistency}
Let $x\in\cN$, $0\le\beta<\gamma\le\al$, $s,t \in T^x_\ga$, $s \preceq t$, and $|s| \ge L(\be,\ga)$. Then
$\tpsi^x_{\be,\ga}(s) \preceq \tpsi^x_{\be,\ga}(t)$.
\end{lemma**}

\begin{proof}
Due to the definition of $\tpsi^x_{\be,\ga}$ and due to Lemma~\ref{L:subsequence}
it is sufficient to prove the particular case, where $|t| = L+1$, $s = t|L$, $\be = \be^{n+1}_L < \be^n_L = \ga$, and $\be = \be^q_{L+1} < \dots < \be^p_{L+1}=\ga$
for some $n,p,q \in \omega, p < q$. Then either $\be+1=\ga$ or $\diamond(\be)=\ga$.

In the first case, we have necessarily $q=p+1$. Since ${\hpsi^x_{\be,\be+1}}$ is monotone (see Definition~\ref{D:psi}) we get
\[
\tpsi^x_{\be,\ga}(s) = \tpsi^x_{\be,\ga}(t|L) = {\hpsi^x_{\be,\be+1}(t|L)}|L \preceq {\hpsi^x_{\be,\be+1}(t)}||t| = \tpsi^x_{\be,\ga}(t).
\]

In the second case, we have $\ell(\be^r_{L+1})\ge\ell(\be)=L$
for $p < r \le q$ by Lemma~\ref{L:diamond}(c). We have $\tpsi^x_{\be,\ga}(t|L)=t|L$.
In the same time, following the definition of $\tpsi^x_{\be,\ga}$, we may see that
\[
\tpsi^x_{\be^{r+1}_{L+1},\be^{r}_{L+1}}(t')|L = t'|L
\] for every $p \leq r < q$ and $t'\in T^x_{\be^{r}_{L+1}}$ with $|t'| = L+1$.
Indeed, if $\diamond(\be^{r+1}_{L+1}) = \be^{r}_{L+1}$, then
\[
\tpsi^x_{\be^{r+1}_{L+1},\be^{r}_{L+1}}(t') = t'.
\]
If $\be^{r+1}_{L+1}+1=\be^{r}_{L+1}$, then
\[
\tpsi^x_{\be^{r+1}_{L+1},\be^r_{L+1}}(t')|L = \hpsi^x_{\be^{r+1}_{L+1},\be^{r+1}_{L+1}+1}(t')|L=t'|L
 \]
due to the property (a) from Definition~\ref{D:psi} of $\hpsi^x_{\be^{r+1}_{L+1},\be^{r+1}_{L+1}+1}$ since $\ell(\be^{r+1}_{L+1})\ge L$ as explained above.
Applying it successively to
\[
\tpsi^x_{\be^{r+1}_{L+1},\be^{r}_{L+1}}\circ\dots\circ\tpsi^x_{\be^{p+1}_{L+1},\be^p_{L+1}}(t)
\]
for $p \leq r < q$ we get $\tpsi_{\beta,\gamma}^x(t)|L = t|L$. This implies
\[
\tpsi_{\beta,\gamma}^x(s) = \tpsi_{\beta,\gamma}^x(t|L) = t|L = \tpsi_{\beta,\gamma}^x(t)|L \preceq \tpsi_{\beta,\gamma}^x(t).
\]
This finishes the proof.
\end{proof}

\begin{lemma**}\label{L:initial-segment-tpsi}
Let $x \in \cN$, $0 \leq \beta < \gamma \leq \alpha$, $l \in \omega$, $s,t \in T_{\gamma}^x$, $|t| \geq \max\{L(\beta,\gamma),l\}$, $|s| \geq \max\{L(\beta,\gamma),l\}$,
and $t|l = s|l$. Then $\tpsi^x_{\beta,\gamma}(t)|l = \tpsi^x_{\beta,\gamma}(s)|l$.
\end{lemma**}

\begin{proof}
It is sufficient to prove the assertion for two special cases, namely, for the case $\diamond(\beta) = \gamma$ and $\beta + 1 = \gamma$.
In the first case we have
\[
\tpsi^x_{\beta,\gamma}(t)|l = t|l = s|l = \tpsi^x_{\beta,\gamma}(s)|l.
\]
In the other case we distinguish two possibilities. If $l \leq \ell(\beta)$, then we have by Definition~\ref{D:psi}
\[
\tpsi^x_{\beta,\gamma}(t)|l = t|l = s|l = \tpsi^x_{\beta,\gamma}(s)|l.
\]
If $l \geq \ell(\gamma) \ (= \ell(\beta) + 1)$, then using Lemma~\ref{L:consistency} we get
\[
\tpsi^x_{\beta,\gamma}(t)|l = \tpsi^x_{\beta,\gamma}(t|l) = \tpsi^x_{\beta,\gamma}(s|l) = \tpsi^x_{\beta,\gamma}(s)|l.
\]
\end{proof}

\begin{definition**}\label{D:definition-of-hPsi}
By Lemma~\ref{L:consistency} we get that for $0 \le \be < \ga \le \al$ and $z \in [T_{\gamma}^x]$ there exists a unique sequence $\hPsi_{\beta,\gamma}^x(z)$
such that $\tpsi^x_{\be,\ga}(z|L) = \hPsi_{\beta,\gamma}^x(z)|L$ for every $L \geq L(\be,\ga)$.
\end{definition**}

\begin{remark**}\label{R:hTheta}\hfil
\begin{enumerate}[(a)]
\item  By definition we have $\hPsi^x_{\be,\ga}(z) \in [T_{\be}^x]$ for every $z \in [T^x_{\gamma}]$ since
$\hPsi^x_{\be,\ga}(z)|L = \tpsi^x_{\beta,\gamma}(z|L) \in T^x_{\beta}$ for every $L \geq L(\beta,\gamma)$.

\item It can be easily checked that the former definition of $\hPsi^x_{\be,\be+1}$ in Definition~\ref{D:hPsi-pro-nasledniky}(2) agrees with the preceding definition for this particular case. Indeed, by monotonicity we have
\[
\tpsi_{\be,\be+1}^x(t) \preceq \hpsi_{\be,\be+1}^x(t) \preceq \tpsi_{\be,\be+1}^x(t')
\]
for every $t,t'\in T^x_{\be+1}$ such that $t\preceq t'$ and $|t'|\ge |\hpsi_{\be,\be+1}^x(t)|$.
\end{enumerate}
\end{remark**}

\begin{lemma**}\label{L:composition}
Let $x \in \cN$, $0 \leq \beta < \gamma < \delta \leq \alpha$. Then $\hPsi_{\beta,\delta}^x = \hPsi_{\beta,\gamma}^x \circ \hPsi_{\gamma,\delta}^x$.
\end{lemma**}

\begin{proof}
Let $L \in \omega$ be such that $L \geq \max\{L(\beta,\gamma),L(\gamma,\delta),L(\beta,\delta)\}$.
For every $t \in T_{\delta}^x$ with $|t| \geq L$
we have $|\tpsi_{\gamma,\delta}^x(t)| = |t|$ and so $\{\beta_{|t|}^i\}_{i=0}^{n(|t|)}$ is used in the definition of the values
$\tpsi_{\beta,\delta}^x(t)$, $\tpsi_{\beta,\gamma}^x\bigl(\tpsi_{\gamma,\delta}^x(t)\bigr)$, and $\tpsi_{\gamma,\delta}^x(t)$.
Then we have by definition $\tpsi_{\beta,\delta}^x(t) = \tpsi_{\beta,\gamma}^x \circ \tpsi_{\gamma,\delta}^x(t)$.
For $z \in [T^x_{\delta}]$ we have
\[
\begin{split}
\hPsi_{\beta,\delta}^x(z)|L &= \tpsi_{\beta,\delta}^x(z|L) = \tpsi_{\beta,\gamma}^x \circ \tpsi_{\gamma,\delta}^x(z|L) \\
&=\tpsi_{\beta,\gamma}^x(\hPsi_{\gamma,\delta}^x(z)|L) = \hPsi_{\beta,\gamma}^x(\hPsi_{\gamma,\delta}^x(z))|L.
\end{split}
\]
This gives $\hPsi_{\beta,\delta}^x(z) = \hPsi_{\beta,\gamma}^x \circ \hPsi_{\gamma,\delta}^x(z)$.
\end{proof}

\begin{lemma**}\label{L:Psi-copy}
Let $x \in \cN$, $0 \leq \beta < \gamma \leq \alpha$, and $L \in \omega$. If $\ell(\delta) \geq L$ for every $\delta \in [\beta,\gamma)$,
then $\hPsi^x_{\beta,\gamma}(z)|L = z|L$ for every $z \in [T_{\gamma}^x]$.
\end{lemma**}

\begin{proof}
We find $k \in \omega$ such that $k \geq L(\beta,\gamma)$ and $k \geq L$. Consider the corresponding sequence $\{\beta_k^i\}_{i=0}^{n(k)}$ from Proposition~\ref{P:diamond-path}
It is sufficient to verify $\hPsi^x_{\beta_k^{i+1},\beta_k^i}(z)|L = z|L$ for $z \in [T_{\be^i_k}^x]$.
Set $t = z|k$. If $\beta_k^{i+1} + 1 = \beta_k^i$, then
\[
\hPsi^x_{\beta_k^{i+1},\beta_k^{i}}(z)|L = \tpsi^x_{\beta_k^{i+1},\beta_k^{i}}(z|k)|L  = \hpsi^x_{\beta_k^{i+1},\beta_k^{i}}(t)|L = t|L = z|L
\]
by Definition~\ref{D:psi} since $\ell(\beta_k^{i+1}) \geq L$.

If $\diamond(\beta_k^{i+1}) = \beta_k^i$, then
\[
\hPsi^x_{\beta_k^{i+1},\beta_k^{i}}(z)|L = \tpsi^x_{\beta_k^{i+1},\beta_k^{i}}(z|k)|L =  \tpsi^x_{\beta_k^{i+1},\beta_k^{i}}(t)|L = t|L = z|L
\]
by the definition of $\tpsi^x_{\beta,\gamma}$.
\end{proof}

\begin{lemma**}\label{L:initial-segment-hPsi}
Let $x \in \cN$, $0 \leq \beta < \gamma \leq \alpha$, $l \in \omega$, $z,w \in [T_{\gamma}^x]$ and $z|l = w|l$.
Then $\hPsi^x_{\beta,\gamma}(z)|l = \hPsi^x_{\beta,\gamma}(w)|l$.
\end{lemma**}

\begin{proof}
Set $L = \max\{L(\beta,\gamma),l\}$. Using Definition~\ref{D:definition-of-hPsi} and Lemma~\ref{L:initial-segment-tpsi}, we get
\[
\hPsi^x_{\beta,\gamma}(z)|l = \tpsi^x_{\beta,\gamma}(z|L)|l = \tpsi^x_{\beta,\gamma}(w|L)|l = \hPsi^x_{\beta,\gamma}(w)|l.
\]

\end{proof}

\begin{lemma**}\label{L:computation-of-small-psi}
Let $x \in \cN$. There exists a partial function $\rho\colon [0,\alpha] \times [0,\alpha] \times \mathcal S \times \cN \rightharpoonup \mathcal S$ which is $\Sigma_1^0(\ep)$-recursive on its domain and such that $\tpsi^x_{\beta,\gamma}(t) = \rho\bigl(\beta,\gamma,t,x^{(\gamma)}\bigr)$ whenever $\tpsi^x_{\beta,\gamma}(t)$ is defined.
\end{lemma**}

\begin{proof}
Let $u$ be the partial function from Lemma~\ref{L:properties-of-theta}.
Using $\Sigma_1^0(\ep)$-recursivity of the mapping $\operatorname{\wh{\phantom{a}}}\colon \cP \to \cS$
(Lemma~\ref{L:volba-epsilon}(e)), we see that the  partial function $\wh{u}\colon t \mapsto \wh{u(t)}$ is $\Sigma_1^0(\ep)$-recursive on its domain.
Notice that we have $\hpsi_{\beta,\be+1}^x(t) = \wh{u}(\beta,t,x^{(\beta+1)})$ for every $\beta \in [0,\alpha), t \in T_{\beta+1}^x$.
Further, let $r_3$ be the partial function from Lemma~\ref{L:jumps}(3) (with $u$ replaced by $x$ and $v$ replaced by $\cl{0}$), i.e.,
$r_3$ is a $\Sigma_1^0(\ep)$-recursive function on its domain and
$x^{(\beta)}  = r_3\bigl(\beta,\gamma,x^{(\gamma)}\bigr)$ for every $0 \le \beta \le \gamma \leq \alpha$.
We define a partial function
\[
f\colon \omega \times [0,\alpha] \times [0,\alpha] \times \Seq \times \cN \rightharpoonup \cS,
\]
using the notation $U = U^{[0,\al]^2\times \cS\times\cN, \cS}_{\Si_1^0(\ep)}$, by
\[
\begin{split}
f(e,&\beta,\gamma,t,y) \\
&=
\begin{cases}
t                                                                                          &\text{if } \be=\ga,  \\
U\bigl(e,\beta,\delta, \wh{u}(\delta,t,y)||t|,y\bigr)                                      &\text{if } \gamma = \delta+1, |t| \geq L(\be,\gamma), \\
U\bigl(e,\beta,a(\gamma,|t|-\ell(\gamma)),t,r_3\bigl(a(\gamma,|t|-\ell(\gamma)),\gamma,y\bigr)\bigr) &\text{if } \gamma \text{ is limit and } |t|\ge L(\be,\ga).
\end{cases}
\end{split}
\]
Using Lemma~\ref{L:volba-epsilon}(g),(e) for $a$, $\ell$, $t \mapsto |t|$ on $\cS^*$,  we get that the partial function $f$ is $\Sigma_1^0(\ep)$-recursive on its domain. By Fact~\ref{F:recursionthms}(b) and Remark~\ref{R:extended-facts} there exists $e^* \in \omega$
such that
$f(e^*,\beta,\gamma,t,y) = U(e^*,\beta,\gamma,t,y)$ whenever $f(e^*,\beta,\gamma,t,y)$ is defined. We set $\rho = f_{e^*}$.
By induction on $1 \leq \gamma \leq \alpha$ we verify that $\tpsi_{\beta,\gamma}^x(t) = f_{e^*}(\beta,\gamma,t,x^{(\gamma)})$
for every $\beta < \gamma, t \in T_{\gamma}^x, |t| \geq L(\beta,\gamma)$.

Suppose that $1 \leq \gamma < \alpha, \gamma = \delta +1$.
Let $\beta < \gamma$ and $t \in T_{\gamma}^x, |t| \geq L(\beta,\gamma)$. Then
there exist $i,j \in \omega$ such that $\beta^i_{|t|} = \gamma = \delta +1$ and $\beta_{|t|}^j = \beta$.
By Proposition~\ref{P:diamond-path} we have $\beta_{|t|}^{i+1} = \delta$ and $L(\beta,\delta+1) \geq L(\delta,\delta+1)$.
Then we have
\[
\begin{split}
\rho(\beta,\gamma,t,x^{(\gamma)}) &= U\bigl(e^*,\beta,\delta,\wh{u}(\delta,t,x^{(\delta+1)})||t|,x^{(\delta+1)}\bigr) = U\bigl(e^*,\beta,\delta,\tpsi_{\delta,\de+1}^x(t),x^{(\delta+1)}\bigr) \\
&= f(e^*,\beta,\delta,\tpsi_{\delta,\de+1}^x(t),x^{(\delta+1)}) \\
&=
\begin{cases}
\tpsi_{\beta,\gamma}^x(t)                                                           &\text{if } \beta = \delta, \\
\tpsi_{\beta,\delta}^x \circ \tpsi_{\delta,\de+1}^x(t) = \tpsi_{\beta^j_{|t|},\beta^{j-1}_{|t|}}^x\circ \cdots \circ \tpsi_{\beta^{i+1}_{|t|},\beta^i_{|t|}}^x(t)=  \tpsi_{\beta,\gamma}^x(t) &\text{if } \beta < \delta.
\end{cases}
\end{split}
\]

Suppose that $\gamma \leq \alpha$ is a limit ordinal.
Let $\beta < \gamma$ and $t \in T_{\gamma}^x, |t| \geq L(\beta,\gamma)$. Then
there exist $i,j \in \omega$ such that $\beta^i_{|t|} = \gamma$ and $\beta_{|t|}^j = \beta$.
We denote $\mu = a(\gamma,|t|-\ell(\gamma))$. Then we have $\ell(\mu) = |t|$ and $\diamond(\mu) = \gamma$  by Lemma~\ref{L:diamond}(b).
Consequently, $\beta_{|t|}^{i+1} = \mu$ and $|t| \geq L(\mu,\gamma)$.
This gives $\tpsi_{\mu,\gamma}^x(t) = t$ by definition of $\tpsi_{\mu,\gamma}^x$. Then we have
\[
\begin{split}
\rho(\beta,\gamma,t,x^{(\gamma)}) &= U\bigl(e^*,\beta,a(\gamma,|t|-\ell(\gamma)),t,r_3\bigl(a(\gamma,|t|-\ell(\gamma)),\gamma,x^{(\gamma)}\bigr)\bigr) \\
&= U\bigl(e^*,\beta,\mu,t,x^{(\mu)}\bigr) = f(e^*,\beta,\mu,t,x^{(\mu)}) = \tpsi_{\beta,\mu}^x(t) = \tpsi_{\beta,\mu}^x \circ \tpsi_{\mu,\gamma}^x(t) \\
&= \tpsi_{\beta^j_{|t|},\beta^{j-1}_{|t|}}^x\circ \cdots \circ \tpsi_{\beta^{i+1}_{|t|},\beta^i_{|t|}}^x(t) = \tpsi_{\beta,\gamma}^x(t).
\end{split}
\]
\end{proof}

\begin{lemma**}\label{L:hPsi-recursivity}
Let $x \in \cN$.
The function $\hPsi^x\colon (\beta,\gamma,z) \mapsto \hPsi_{\beta,\gamma}^x(z)$ is a partial function which is $\Sigma_1^0(\ep)$-recursive on its domain and
defined for $0 \leq \beta < \gamma \leq \alpha$ and $z \in [T_{\ga}^x]$.
\end{lemma**}

\begin{proof}
Let $\pi$ be the $\Si^0_1(\ep)$-recursive mapping from Notation~\ref{N:tree-like}(b) and
$\rright\colon [0,\alpha] \times \cN \rightharpoonup \cN$ be the partial $\Sigma_1^0(\ep)$-recursive function from Proposition~\ref{P:jump-trees-old}(c).
Then $r_*(\ga,z)=\rright(\ga,\pi(\ga,z))$ is a partial $\Sigma_1^0(\ep)$-recursive function.
For every $\gamma \in [0,\alpha]$ and $z \in [T_{\gamma}^x]$ we have $\pi(\gamma,z) = \sigma_{\gamma}^x$ (see Notation~\ref{N:tree-like}(b)) and, consequently,
$r_*(\ga,z)=\rright(\ga,\si^x_\ga)=x^{(\gamma)}$.
Let $\rho$ be the $\Sigma_1^0(\ep)$-recursive function from Lemma~\ref{L:computation-of-small-psi}. Then, for every $0 \leq \beta < \gamma \leq \alpha$, $z \in [T_{\gamma}^x]$, $t \in \cS$, we have
\begin{equation}\label{E:computation-of-Psi}
\begin{split}
\hPsi^x_{\beta,\gamma}(z) \in \cN(t)
&\Leftrightarrow \exists L \in \omega \colon L \geq L(\beta,\gamma) \land t \preceq \tpsi^x_{\beta,\gamma}(z|L) \\
&\Leftrightarrow \exists L \in \omega \colon L \geq L(\beta,\gamma) \land t \preceq \rho(\beta,\gamma,z|L,x^{(\gamma)}) \\
&\Leftrightarrow \exists L \in \omega \colon L \geq L(\beta,\gamma) \land t \preceq \rho\bigl(\beta,\ga,z|L, r_*(\ga,z)\bigr).
\end{split}
\end{equation}
Using Lemma~\ref{L:volba-epsilon}(d),(i),(g),
we get that there exists a $\Si_1^0(\ep)$-set $\Delta \subset [0,\alpha]^2 \times \cN \times \cS$ such that
\[
\begin{split}
\Delta \cap \{(\beta,\gamma,z,t) &\in [0,\alpha]^2 \times \cN \times \cS \set z \in [T^x_{\gamma}]\}  \\
&= \bigl\{(\beta,\gamma,z,t) \in [0,\alpha]^2 \times \cN \times \cS \set z \in [T^x_{\gamma}] \land 0 \leq \beta < \gamma \leq \alpha \\
&\hskip 130pt \land  \exists L \geq L(\beta,\gamma)\colon
t \preceq \rho\bigl(\beta,\ga,z|L, r_*(\ga,z)\bigr)\bigr\}.
\end{split}
\]
By \eqref{E:computation-of-Psi} the set $\Delta$ computes the partial function $\hPsi^x$ and we are done, cf. Lemma~\ref{L:natural-base}.
\end{proof}

\subsection{Properties of \texorpdfstring{$\Xi^x$}{} -- proof of Proposition~\ref{P:hPsi-properties}}\label{SS:properties-of-Xi}

\begin{notation**}
We denote $\Xi^x = \hPsi^x_{0,\al}$ for $x \in \cN$.
\end{notation**}

By Proposition~\ref{P:uzavrenost-F} we have a $\Pi^0_1(\ep)$ set $F\su\cN^2$.
By Remark~\ref{R:hTheta}(a) $\Xi^x\colon E_x\to F_x$.
We are going to prove that $F$ and $\Xi^x$, $x\in\cN$, are the objects having the properties (a) and (b) from Proposition~\ref{P:hPsi-properties}.
Proposition~\ref{P:hPsi-properties}(a) is a particular case of the following lemma.

\begin{lemma**}\label{L:psi-properties}
Let $x \in \cN$.
The function $(\beta,z) \mapsto \bigl(\hPsi_{0,\beta}^x(z)\bigr)^{(\beta)}_x$ is a partial function which is $\Sigma_1^0(\ep)$-recursive on its domain $(0,\alpha]\times [T_{\beta}^x]$.
\end{lemma**}

\begin{proof}
Denote $U = U^{[0,\al]\times\cN,\cN}_{\Si_1^0(\ep)}$.
We use the mapping $r_U$ from Lemma~\ref{L:jumps}(2), where $W = \omega \times [0,\al]$ and $q, v$ are replaced by $U, x$.
Let us define a partial mapping $f \colon \omega \times [0,\alpha] \times \cN \rightharpoonup \cN$ by
\[
f(e,\beta,z) =
\begin{cases}
z                                                                                     &\text{if } \beta = 0, \\
r_U\bigl(e,\ga,({\hPsi_{\gamma,\ga+1}^x}(z))'_x\bigr)                                 &\text{if } \beta = \gamma + 1 < \alpha, \\
\bigvee_{k \in \omega} U\bigl(e,a(\beta,k),\hPsi^x_{a(\be,k),\be}(z)\bigr)            &\text{if $\beta \leq \alpha$ is a limit ordinal.}
\end{cases}
\]

We may deduce that the partial function $f$ is $\Sigma_1^0(\ep)$-recursive on its domain.
By Lemma~\ref{L:volba-epsilon}, we get the $\Si^0_1(\ep)$-recursivity of the relation $\le$ on $\om$, of the partial functions $a$ and  $\ga + 1 \mapsto \ga$, and of projections of product spaces to their factors used in the definition of $f$.
Using Lemmas~\ref{L:bigvee}, \ref{L:recursivity-of-Theta}(b), and \ref{L:hPsi-recursivity}, we get also the $\Si^0_1(\ep)$-recursivity of the mappings $\bigvee_{k\in\om}$, $(\ga,z)\mapsto (\hPsi^x_{\ga,\ga+1}(z))'_x$, and $(\ga,\be,z)\mapsto\hPsi^x_{\ga,\be}(z)$.

By Fact~\ref{F:recursionthms}(b) and Remark~\ref{R:extended-facts} there exists $e^* \in \omega$ such that  $f(e^*,\beta,z) = U(e^*,\beta,z)$  whenever $f(e^*,\beta,z)$ is defined.
So $f_{e^*}$ is a partial function which is $\Sigma_1^0(\ep)$-recursive on its domain.

We put $\hPsi^x_{0,0}(z)=z$ for $z\in [T^x_0]$ and, by induction on $\beta$, we verify that $\bigl(\hPsi_{0,\beta}^x(z)\bigr)^{(\beta)}_x = f_{e^*}(\beta,z)$
for every $\beta \in [0,\alpha]$ and $z \in [T_{\beta}^x]$.
For $\beta = 0$ and $z \in [T_0^x]$ we have
\[
\hPsi_{0,0}^x(z)^{(0)}_x = z = f(e^*,0,z)
\]
by definition of $f$ and Definition~\ref{D:Turing-jump} of the jump.

Let us assume that $\be < \al$ is fixed and the claim holds for all $0\le\ga < \be$.
If $\beta=\ga+1 > 1$ and $z \in [T_\be^x]$, we have
\[
(\hPsi_{0,\ga+1}^x(z))^{(\ga)}_x = \bigl(\hPsi_{0,\ga}^x(\hPsi_{\ga,\ga+1}^x(z))\bigr)^{(\ga)}_x
= f(e^*,\ga,\hPsi_{\ga,\ga+1}^x(z))
= U(e^*,\ga,\hPsi_{\ga,\ga+1}^x(z))
\]
by Lemma~\ref{L:composition}, by the induction hypothesis (and Remark~\ref{R:hTheta}(a)), and by the choice of $e^*$.
Thus we have
\[
\begin{split}
\bigl(\hPsi_{0,\be}^x(z)\bigr)^{(\be)}_x &= \bigl((\hPsi_{0,\gamma+1}^x(z))^{(\gamma)}_x\bigr)_x' = U\bigl(e^*,\gamma,\hPsi_{\gamma,\gamma+1}^x(z)\bigr)'_x \\
&= r_U\bigl(e^*,\gamma,({\hPsi_{\gamma,\ga+1}^x}(z))'_x\bigr) = f_{e^*}(\beta,z)
\end{split}
\]
by the preceding equalities, the choice of $r_U$, and the definition of $f$.

Finally assume that $\be$ is a limit ordinal and $z \in [T_{\beta}^x]$.
Using Definition~\ref{D:Turing-jump}, Lemma~\ref{L:composition}, the induction hypothesis, the choice of $e^*$, and the definition of $f$, we get
\[
\begin{split}
\bigl(\hPsi_{0,\beta}^x(z)\bigr)^{(\beta)}_x &=  \bigvee_{k \in \omega}\bigl(\hPsi_{0,\beta}^x(z)\bigr)^{(a(\beta,k))}_x =
\bigvee_{k \in \omega} \bigl(\hPsi_{0,a(\beta,k)}^x (\hPsi_{a(\beta,k),\beta}^x(z))\bigr)^{(a(\beta,k))}_x \\
&= \bigvee_{k \in \omega} f\bigl(e^*,a(\beta,k),\hPsi_{a(\beta,k),\beta}^x(z)\bigr) =
   \bigvee_{k \in \omega} U\bigl(e^*,a(\beta,k),\hPsi_{a(\beta,k),\beta}^x(z)\bigr) \\
&=  f_{e^*}(\beta,z).
\end{split}
\]
\end{proof}

To prove Proposition~\ref{P:hPsi-properties}(b) let us point out that the continuity of $\Xi^x$ follows from the $\Si^0_1(\ep)$-recursivity on its domain proved in Lemma~\ref{L:hPsi-recursivity}.
In fact, we used it in the preceding proof of (a) and we do not need it elsewhere.

Now we prove bijectivity of the mappings $\Xi^x_k$ from Proposition~\ref{P:hPsi-properties}(b).
The mapping $\Xi^x_k$ maps $[T_{\al}^x]_k$ into $[T_{0}^x]_k$ due to Lemma~\ref{L:Psi-copy} since $\ell(\de)\ge 1$ for every $\de\in [0,\alpha)$.
Thus it is sufficient to prove that each $\Xi^x$ is a bijection. This is a particular case of the following lemma.

\begin{lemma**}\label{P:hPsi-bijekce}
Let $0\le\ga<\be\le\al$ and $x \in \cN$.
Then the mapping $\hPsi_{\ga,\be}^x$ is a bijection of $[T_{\be}^x]$ onto $[T_{\ga}^x]$.
\end{lemma**}

\begin{proof}
We put $\hPsi_{\be,\be}^x(z)=z$ for $z\in [T^x_\be]$ and $\be \in [0,\alpha]$.
By transfinite induction on $\be\in [0,\al]$ we prove the statement $S(\be)$ claiming that for every $\ga$ such that $0\le\ga\le\be$ the mapping $\hPsi_{\ga,\be}^x$ is a bijection of $[T_{\be}^x]$ onto $[T_{\ga}^x]$.
The statement $S(0)$ holds due to the definition of $\hPsi_{0,0}^x$.
Let us assume that $\be<\al$ and the statement $S(\de)$ holds for all $\de\in [0,\be)$.

Suppose first that $\be$ is an isolated ordinal of the form $\be = \de + 1$.
By induction hypothesis $\hPsi_{\ga,\delta}^x$ is a bijection if $0\le\ga\le\de$.
This implies that $\hPsi_{\ga,\be}^x$ is also a bijection since $\hPsi_{\ga,\be}^x = \hPsi_{\ga,\delta+1}^x =
\hPsi_{\ga,\delta}^x \circ \hPsi_{\delta,\delta+1}^x$ is a composition of two bijections by Lemma~\ref{L:composition}, the induction hypothesis, and the statement of Lemma~\ref{L:bijectivity-of-Theta}(b).

Assume now that  $\be \leq \alpha$ is a limit ordinal.
We prove first the injectivity of $\hPsi^x_{\ga,\be}$ for all $\ga\in [0,\be)$.
Suppose that $z^1,z^2 \in [T_{\be}^x]$ and $z^1 \neq z^2$. We fix $\ga<\be$ and find $k \geq L(\ga,\be)$ such that
$z^1|(\ell(\be)+k) \neq z^2|(\ell(\be)+k)$ and $a(\be,k)>\ga$.
Using Lemma~\ref{L:Psi-copy} and Lemma~\ref{L:diamond}(b),(c), we get
\[
\hPsi^x_{a(\be,k),\be}(z^1)|\ell(a(\be,k)) = z^1|\ell(a(\be,k)) \neq z^2|\ell(a(\be,k)) = \hPsi^x_{a(\be,k),\be}(z^2)|\ell(a(\be,k)).
\]
This gives $\hPsi^x_{a(\be,k),\be}(z^1) \neq \hPsi^x_{a(\be,k),\be}(z^2)$.
Using induction hypothesis and Lemma~\ref{L:composition}, we get
\[
\hPsi^x_{\ga,\be}(z^1) = \hPsi^x_{\ga,a(\be,k)}\bigl(\hPsi^x_{a(\be,k),\be}(z^1)\bigr) \neq \hPsi^x_{\ga,a(\be,k)}\bigl(\hPsi^x_{a(\be,k),\be}(z^2)\bigr) =
\hPsi^x_{\ga,\be}(z^2).
\]

It remains to prove surjectivity of $\hPsi_{\ga,\be}^x$ for the limit ordinal $\be$ and any $\ga\in [0,\be)$.
Suppose $w \in [T_{\ga}^x]$.
Using Lemma~\ref{L:diamond}(b), we find $k_0 \in \omega$ such that $\ga < a(\be,k_0)$.
Using that $S(a(\be,k_0))$ holds by the induction hypothesis, we find $w' \in [T_{a(\be,k_0)}^x]$ with $\hPsi_{\ga,a(\be,k_0)}(w') = w$.
Since also $S(a(\be,k))$ holds for every $k > k_0$, we get that the mapping $\hPsi_{a(\be,k_0),a(\be,k)}^x$ is a bijection
and we find $z^k \in [T_{a(\be,k)}^x]$ such that $\hPsi_{a(\be,k_0),a(\be,k)}(z^k) = w'$.
Applying this and Lemma~\ref{L:composition}, we get for every $j > k > k_0$ the equalities
\[
\begin{split}
\hPsi_{a(\be,k),a(\be,j)}^x(z^{j}) &=
\bigl(\hPsi_{a(\be,k_0),a(\be,k)}^x\bigr)^{-1} \circ \hPsi_{a(\be,k_0),a(\be,j)}^x(z^{j}) \\
&= \bigl(\hPsi_{a(\be,k_0),a(\be,k)}^x\bigr)^{-1}(w') = z^k.
\end{split}
\]
Further $\ell(\delta) \geq \ell(a(\be,k))$ for every $\delta \in [a(\be,k),\be)$ by Lemma~\ref{L:diamond}(c).
Using Lemma~\ref{L:Psi-copy}, we infer
$z^j|\ell(a(\be,k)) = z^k|\ell(a(\be,k))$ for every $j \geq k$.
This shows that there is a unique $z^* \in \cN$ such that
\begin{equation}\label{E:z-star}
z^*|\ell(a(\be,k)) = z^k|\ell(a(\be,k)) \text{ for every } k > k_0.
\end{equation}
Since $z^k|\ell(a(\be,k)) \in T_{a(\be,k)}^x$ and $\diamond(a(\be,k)) = \be$, using Lemma~\ref{P:object-P}(b) we obtain that $z^k|\ell(a(\be,k)) \in T_{\be}^x$.
This implies that $z^*|\ell(a(\be,k)) \in T^x_{\be}$ for every $k > k_0$. Consequently, we have $z^* \in [T_{\be}^x]$.

For every $k > k_0$ we have
\begin{equation*}
\begin{split}
\hPsi_{a(\be,k_0),\be}^x(z^*)|l(a(\be,k))
&= \hPsi_{a(\be,k_0),\be}^x(z^k)|l(a(\be,k)) \qquad \text{(\eqref{E:z-star} and Lemma~\ref{L:initial-segment-hPsi})}  \\
&= \hPsi_{a(\be,k_0),a(\be,k)}^x (\hPsi_{a(\be,k),\be}^x (z^k))|l(a(\be,k))   \qquad \text{(Lemma~\ref{L:composition}})  \\
&= \hPsi_{a(\be,k_0),a(\be,k)}^x (z^k)|l(a(\be,k)) \qquad \text{(Lemma~\ref{L:Psi-copy} and \ref{L:initial-segment-hPsi})} \\
&= w'|l(a(\be,k)).
\end{split}
\end{equation*}
This gives that $\hPsi_{a(\be,k_0),\be}^x(z^*) = w'$
since $\{\ell(a(\be,k))\}_{k=k_0}^{\infty}$ converges to infinity by Lemma~\ref{L:diamond}(b) and, consequently,
$\hPsi_{\ga,\be}^x(z^*) = \hPsi_{\ga,a(\be,k_0)}^x(w') = w$.
\end{proof}

Finally, we prove the continuity of the inverse mappings to $\Xi^x$, $x\in\cN$, and thus also of the inverse mappings to $\Xi^x_k$ for $x\in\cN$ and $k\in\om$ from Proposition~\ref{P:hPsi-properties}(b).

\begin{lemma**}\label{L:hPsi-inverse-continuity}
Let $x \in \cN$.
\begin{enumerate}[\upshape (a)]
\item For every $s\in\cS$ there is $i^s\in\om$ and an injective mapping $\io^s\colon \om \to \om$ such that if $y\in E_x=[T^x_\al]$ is of the form $y = s \w j \w z$,
where $j \in \omega$ and $z \in \cN$, then $\Xi^x(y)_{i^s} = \io^s(j)$.

\item The inverse mapping to $\Xi^x$ is continuous.
\end{enumerate}
\end{lemma**}

\begin{proof}
(a) Let $s\in\cS$ be fixed. We find $N \in \om$, $\al_n \in [0,\al]$ and $s_n \in \cS$ for $n\in\{0,\dots,N\}$, and $\io_n \colon \om \to \om$ for $n \in \{1,\dots,N\}$ such that
\begin{enumerate}
\item $\al_0 = \al > \dots > \al_N = 0$,

\item $s_0 = s$,

\item $\io_n$ is injective for every $n \in \{1,\dots,N\}$, and

\item for every $n \in \{0,\dots,N-1\}$ and $s_n \w j \preceq y$, where $j \in \om$ and $y \in [T^x_{\al_n}]$, we have
\[
s_{n+1}\w \bigl(\io_{n+1}(j)\bigr) = \hPsi^x_{\al_{n+1},\al_n}(y) | (|s_{n+1}|+1).
\]
\end{enumerate}

We will proceed inductively on $n$. Setting $\al_0=\al$ and $s_0=s$ we have all from the conditions (1)--(4) concerning just $n=0$ satisfied.
Suppose that $\al_n$ and $s_n$ for $n\ge 0$ have been already chosen such that all what (1)--(4) state about them is satisfied
and that $\al_n>0$.

If $\al_n$ is a limit ordinal, we find $\al_{n+1}\in [0,\al]$ such that $\diamond(\al_{n+1})=\al_n$ and $\ell(\al_{n+1}) \ge |s_n|+1$.
If $\al_n$ is an isolated ordinal, we choose $\al_{n+1}$ such that $\al_n=\al_{n+1}+1$.

If $\ell(\al_{n+1}) \ge |s_n|+1$, we put $s_{n+1} = s_n$ and $\io_{n+1}$ is the identity mapping on $\om$.
Then Lemma~\ref{L:Psi-copy} applied to $y\in [T^x_{\al_n}]$ with $s_n\w j\preceq y$ gives
\[
\hPsi^x_{\al_{n+1},\al_n}(y)|(|s_{n+1}|+1) = s_n \w j = s_{n+1} \w \bigl(\io_{n+1}(j)\bigr).
\]

Suppose now that $\ell(\al_{n+1})<|s_n|+1$. Then we have $\alpha_n = \alpha_{n+1}+1$. Let $s_n\w j\preceq y\in [T^x_{\al_n}]$.
We observe that $\psi^x_{\al_{n+1},\al_n}(s_n\w j)$ is defined by Definition~\ref{D:psi} since $\ell(\al_n)=\ell(\al_{n+1})+1\le |s_n|+1$ by Lemma~\ref{L:diamond}(d) and our assumption.
We set $s_{n+1}=\hpsi^x_{\al_{n+1},\al_n}(s_n)$ which is also defined by Definition~\ref{D:psi} since $s_n=(s_n\w j)^-$.
Using Definition~\ref{D:psi}(a),(b), we get $l\in\om$ and $\eta\in\cS^*$ such that
$\psi^x_{\al_{n+1},\al_n}(s_n\w j)=\psi^x_{\al_{n+1},\al_n}(s_n)\w (l,\eta)$, where $[l] = \bigl(j,\si^x_{\al_{n+1}}\bigl(|s_{n+1}|-\ell(\al_{n+1})\bigr)\bigr)$.
Thus we have
\[
s_{n+1}\w l\preceq\hpsi^x_{\al_{n+1},\al_n}(s_n\w j)\preceq\hPsi^x_{\al_{n+1},\al_n}(y).
\]
The above $l$ is determined uniquely by $j$ for given $n$ and $s$ since the mapping $[\cdot]$ is a bijection (see Notation~\ref{N:identification}).
Thus $\io_{n+1}(j) = l$ defines an injective mapping $\io_{n+1}\colon \om \to \om$ such that
$s_{n+1}\w \bigl(\io_{n+1}(j)\bigr)=\hPsi^x_{\al_{n+1},\al_n}(y)|(|s_{n+1}|+1)$.

Since the sequence of $\al_n$'s is strictly decreasing, the induction necessarily ends by finding $\al_N=0$, $s_N$, and $\io_N$ for some $N\in\om$
such that (1)--(4) is satisfied for $\al_n$, $s_n$, and $\io_n$, where $n \in \{0,\dots,N\}$.

Let $y \in [T^x_\al]=E_x$ and $s\w j\preceq y$. Using Lemma~\ref{L:composition}, we may use the condition (4) inductively.
We get $( y_n := )\hPsi^x_{\al_n,\al}(y)\in [T^x_{\al_n}]$ by Remark~\ref{R:hTheta}(a)
and due to (4) we also get
\[
y_n|(|s_n|+1)=\bigl(\hPsi^x_{\al_n,\al_{n-1}}\circ\dots\circ\hPsi^x_{\al_1,\al}\bigr)(y)|(|s_n|+1)=s_n\w \bigl(\io_n\circ\dots\circ\io_1(j)\bigr)
\]
for $n =1,\dots,N$. Thus using the notation $\io^s=\io_N\circ\dots\circ\io_1$ and $i^s=|s_N|$,
we get in the particular case $n=N$ that $\Xi^x(y)=\hPsi^x_{0,\al}(y)\succeq s_N\w \bigl(\io^s(j)\bigr)$ and so
$\Xi^x(y)_{i^s}=\io^s(j)$.

\medskip\noindent
(b) Let $z \in F_x$. Denote $y = (\Xi^x)^{-1}(z)$.
It follows from (a) used for $s = y|k$ and $j=y_k$ that $y_k = (\iota^{y|k})^{-1}\bigl(z(i^{y|k})\bigr)$
for $k \in \om$.
Let $n \in \omega$. We set $M = \{i^{y|0},i^{y|1},\dots,i^{y|n}\}$.
Then for every $z' \in F_x$ with $z'|M = z|M$ we get
$(\Xi^x)^{-1}(z)|n = (\Xi^x)^{-1}(z')|n$. Thus $(\Xi^x)^{-1}$ is continuous.
\end{proof}

This conludes the proof of Proposition~\ref{P:hPsi-properties}.

\bigskip

\textbf{Acknowledgement.} We thank to Alexander Kechris for drawing our attention to the manuscripts \cite{Harrington, Harrington2}
and to Ond\v{r}ej Kurka, Dominique Lecomte, and Alain Louveau for helpful discussions.

\end{document}